\documentclass[aos,preprint]{imsart}

\RequirePackage[OT1]{fontenc}
\RequirePackage{amsthm,amsmath}
\RequirePackage[numbers]{natbib}
\RequirePackage[colorlinks,citecolor=blue,urlcolor=blue]{hyperref}
\usepackage{graphicx, amssymb, mathrsfs,amsfonts,url}
\usepackage{hyperref, cleveref, autonum}
\usepackage[subrefformat=parens,labelformat=parens]{subfig}
\usepackage{color, soul}
\usepackage{enumerate}

\newcommand{\be}{\begin{equation}}
\newcommand{\ee}{\end{equation}}
\newcommand{\beqa}{\begin{eqnarray*}}
\newcommand{\eeqa}{\end{eqnarray*}}
\newcommand{\beqn}{\begin{eqnarray}}
\newcommand{\eeqn}{\end{eqnarray}}
\newcommand{\ba}{\begin{array}}
\newcommand{\ea}{\end{array}}
\newcommand{\bc}{\begin{center}}
\newcommand{\ec}{\end{center}}
\newcommand{\btab}{\begin{tabular}}
\newcommand{\etab}{\end{tabular}}

\newcommand{\mb}{\makebox}

\newcommand{\bm}[1]{\mb{\boldmath ${#1}$}}
\newcommand{\bb}[1]{\mathbb{#1}}

\newcommand{\Ind}{1\!\mathrm{l}}

\renewcommand{\P}{{\rm P}}

\newcommand{\E}{{\rm E}}

\newcommand{\nin}{\not\in}

\newcommand{\bs}{\boldsymbol}

\newcommand{\bsxi}{\bs{\xi}}
\newcommand{\GP}{\mathrm{GP}}
\newcommand{\sphere}{\mathbb{S}^{d - 1}}
\newcommand{\scale}{a}

\newcommand{\comment}{}
\newcommand{\commentA}{}

\arxiv{arXiv:1508.05847}

\numberwithin{equation}{section}
\theoremstyle{plain}
\newtheorem{theorem}{Theorem}[section]
\newtheorem{lemma}[theorem]{Lemma}
\newtheorem{proposition}[theorem]{Proposition}
\newtheorem{example}[theorem]{Example}
\theoremstyle{remark}
\newtheorem{remark}[theorem]{Remark}
\begin{document} \sloppy

\begin{frontmatter}
\title{Bayesian Detection of Image Boundaries %\thanksref{T1}
}
\runtitle{Bayesian Detection of Image Boundaries}
%\thankstext{T1}{Footnote to the title with the ``thankstext'' command.}

\begin{aug}
\author{\fnms{Meng} \snm{Li} \thanksref{grant} \ead[label=e1]{ml371@stat.duke.edu}}
\and 
\author{\fnms{Subhashis} \snm{Ghosal}\thanksref{grant}\ead[label=e2]{ sghosal@stat.ncsu.edu}}
%\and
%\author{\fnms{Third} \snm{Author}\thanksref{t1,m2}
%\ead[label=e3]{third@somewhere.com}
%\ead[label=u1,url]{http://www.foo.com}}

%\thankstext{t1}{Some comment}
\thankstext{grant}{Research partially supported by NSF grant number DMS-1106570.}
\runauthor{Li, M. and Ghosal, S.}

\affiliation{Duke University and North Carolina State University}

\address{Department of Statistical Science\\
Duke University\\
% Box 90251\\
Durham, North Carolina 27708-0251\\
USA \\
\printead{e1}\\
\phantom{E-mail:\ }}

\address{Department of Statistics\\
North Carolina State University \\
% 4276 SAS HALL, 2311 STINSON DRIVE \\
Raleigh, North Carolina 27695-8203 \\
% RALEIGH, NORTH CAROLINA 27695-8203 \\
USA \\
\printead{e2}\\
\phantom{E-mail:\ }}

%\address{S. \\
%DEPARTMENT OF STATISTICS\\
%NORTH CAROLINA STATE UNIVERSITY\\
%4276 SAS HALL, 2311 STINSON DRIVE \\
%RALEIGH, NORTH CAROLINA 27695-8203 \\
%USA 
%\printead{e3}\\
%\printead{u1}}
\end{aug}

\begin{abstract}
Detecting boundary of an image based on noisy observations is a fundamental problem of image processing and image segmentation. For a $d$-dimensional image ($d = 2, 3, \ldots$), the boundary can often be described by a closed smooth $(d - 1)$-dimensional manifold. In this paper, we propose a nonparametric Bayesian approach based on priors indexed by $\mathbb{S}^{d - 1}$, the unit sphere in $\mathbb{R}^d$. We derive optimal posterior contraction rates for Gaussian processes or finite random series priors using basis functions such as trigonometric polynomials for 2-dimensional images and spherical harmonics for 3-dimensional images. For 2-dimensional images, we show a rescaled squared exponential Gaussian process on $\mathbb{S}^1$ achieves four goals of guaranteed geometric restriction, (nearly) minimax optimal rate adapting to the smoothness level, convenience for joint inference and computational efficiency. We conduct an extensive study of its reproducing kernel Hilbert space, which may be of interest by its own and can also be used in other contexts. Several new estimates on the modified Bessel functions of the first kind are given. Simulations confirm excellent performance and robustness of the proposed method. 
\end{abstract}



\end{frontmatter}

\section{Introduction}
\label{sec:intro}
The problem of detecting boundaries of image arise in a variety of areas including epidemiology~\citep{Waller+Gotway:04}, geology~\citep{Lu+Car:05}, ecology~\citep{Fit+:10}, forestry, marine science. A general $d$-dimensional $(d \geq 2)$ image can be described as $(X_{i}, Y_{i})_{i = 1}^{n}$ , where $X_i \in T = [0, 1]^d$ is the location of the $i$th observation and $Y_i$ is the corresponding pixel intensity. Let $f(\cdot; \phi)$ be a given regular parametric family of densities with respect to a $\sigma$-finite measure $\nu$, indexed by a $p$-dimensional parameter $\phi \in \Theta$, then we assume that there is a closed region $\Gamma \subset T$ such that
\begin{equation}
\label{eq:model}
Y_{i} \sim \left\{
\begin{array}{l l}
f(\cdot; {\xi}) & \text{if } X_{i} \in \Gamma; \\
f(\cdot; {\rho}) & \text{if } X_{i} \in \Gamma^c,
\end{array} \right.
\end{equation}
where $\xi, \rho$ are distinct but unknown parameters. We assume that both $\Gamma$ and $\Gamma^c$ have nonzero Lebesgue measures. The goal here is to recover the boundary $\gamma = \partial \Gamma$ from the noisy image where $\gamma$ is assumed to be a smooth $(d-1)$-dimensional manifold without boundary, and derive the contraction rate of $\gamma$ at a given true value $\gamma_0$ in terms of the metric defined by the Lebesgue measure of the symmetric difference between the regions enclosed by $\gamma$ and $\gamma_0$. 
When the boundary itself is of interest such as in image segmentation, we  can view the problem as a generalization of the change-point problem in one-dimensional data to images.

%For example, in~\cite{Donoho:99}, the wedgelets, as a multiscale approaches, suffers from the so-called blocking effects which leads to discontinuous boundary reconstruction resulting in bad visualization. Note that it is rate optimal in terms of $L^2$-loss of the entire intensity function.

A significant part of the literature focuses on the detection of boundary pixels, based on either first-order or second-order derivatives of the underlying intensity function~\cite[Ch. 6]{Qiu:05} or Markov random fields~\citep{Geman+Geman:84,Geman+Geman:93}, resulting in various edge detectors or filters. This approach is especially popular in computer vision~\citep{Basu:02, Bha+Mit:12}. However, the detected boundary pixels are scattered all over the image and do not necessarily lead to a closed region, and hence cannot be directly used for image segmentation. A post-smoothing step can be applied, such as Fourier basis expansion, principal  curves~\citep{Hastie+Stu:89} or a Bayesian multiscale method proposed by~\cite{Gu+:14}. However the ad-hoc two-step approach makes the theoretical study of convergence intractable. In addition, as pointed out by~\cite{Banerjee+Gelfand:06}, many applications produce data at irregular spatial locations and do not have natural neighborhoods.

Most existing methods are based on local smoothing techniques~\citep{Car+Kri:92, Rud+Str:94b, Hall+:01, Pol+Spo:03, Qiu+Sun:07}, which lead to convenient study of theoretical properties benefiting from well established results. However, local methods suffer when the data is sparse and thus the usage of the global information becomes critical. More importantly, it often leads to {\it local} (or pointwise) inference such as marginal confidence bands losing the joint information.

A relevant and intensively studied problem is to estimate the underlying intensity function $\E(Y|X)$ with discontinuity at the boundary~\citep{Muller+Song:94,Qiu+Yan:97, Donoho:99, Hall+:01,Qiu:07,Qiu+Sun:09}. These two problems are different for at least two reasons. Firstly, there are many important applications where $\xi$ and $\rho$ affect $f(\cdot)$ not (or not only) in the mean but some other characteristics such as variance~\citep{Car+Kri:92}. Secondly, the reconstruction of $\E(Y|X)$ is essentially a curve (or surface) fitting problem with discontinuity and the corresponding asymptotics are mostly on the entire intensity function rather than the boundary itself. Therefore, we may refer the latter as image denoising when boundaries are present, not necessarily guaranteeing the geometric restrictions on the boundary such as closedness and smoothness.

In this paper,  we propose a nonparametric Bayesian method tailored to detect the boundary $\gamma_0$, which is viewed as a closed smooth $(d - 1)$-dimensional manifold without boundary.  This paper has three main contributions. 

The first main contribution is that the proposed method is, to our best knowledge, the first one in the literature that achieves all the following four goals (i)--(iv) when estimating the boundary. 
\begin{enumerate}[(i).]
\item Guaranteed geometric restrictions on the boundary such as closedness and smoothness.
\item Convergence at the (nearly) minimax rate~\citep{Kor+Tsy:93, Mammen+Tsybakov:95}, adaptively to the smoothness of the boundary.
\item Possibility and convenience of joint inference.
\item Computationally efficient algorithm.
\end{enumerate}
To address (i) and (iii), the Bayesian framework has its inherent advantages. For (i), we note that Bayesian methods allow us to put the restrictions on the boundary conveniently via a prior distribution. Specifically, we propose to use a Gaussian process (GP) prior indexed by the unit sphere in $\mathbb{R}^d$, i.e. the $(d-1)$-sphere $\mathbb{S}^{d - 1} = \{x = (x_1, \ldots, x_d) \in \mathbb{R}^d: x_1^2 + \cdots + x_d^2 = 1 \},$
 or a random series prior on $\sphere$. For (iii), Bayesian methods allow for joint inference since we draw samples from the joint posterior distributions, as demonstrated by the numerical results in Section~\ref{section:simulation}.  The proposed method achieves the (nearly) minimax optimal rate adapting to the unknown smoothness level based on a random rescaling incorporated by a hierarchical prior~\citep{van+van:09,Shen+Ghosal:14}. Furthermore, Goal (ii) is achieved for any regular family of noise and general dimensions. In contrast, for instance, \comment{the method in~\citep{Mammen+Tsybakov:95} is presented only for binary images and does not adapt to the unknown smoothness level. }
 Although the quantification of uncertainty and adaptivity of a method is appealing, the computation in goal (iv) is important when implementing it. Many adaptive methods are hard to implement since inverses of covariance matrices need to be calculated repeatedly. In the proposed Bayesian approach, an efficient Markov chain Monte Carlo (MCMC) sampling is designed based on the analytical eigen decomposition of the squared exponential periodic (SEP) kernel (see Section~\ref{section:bd.sampling}), for various noise distributions. 
In addition, we conduct extensive numerical studies to confirm the good performance of the proposed method and indicate that it is robust under model misspecification.

As the second main contribution, we conduct an extensive study on the reproducing kernel Hilbert space (RKHS) of the SEP Gaussian process, which is essential to obtain the optimal rate and adaptation in Goal (ii). For the most important case in applications $d = 2$, by a simple mapping, the squared exponential (SE) Gaussian process on $\mathbb{S}^1$ is equivalent to the SEP Gaussian process on $[0,1]$ since their RKHS's are isometric (see Lemma~\ref{lemma:bd.isometric}). Recently developed theory of posterior contraction rates implies that nonparametric Bayesian procedures can automatically adapt to the unknown smoothness level using a rescaling  factor via a hyperparameter in a stationary Gaussian processes on [0, 1] or $[0, 1]^d$~\citep{van+van:07,van+van:09}. Rescaled SE Gaussian process is one popular example of this kind. In contrast, the literature lacks results on the rescaling scheme and the resulting properties of  the SEP Gaussian process, even though it has been implemented in many applications~\cite{MacKay:98}. It may due to the apparent similarity shared between the SEP Gaussian process and the SE Gaussian process.  
However, these two processes have fundamental differences because the rescaling of the argument on $\mathbb{S}^1$ cannot be transformed as a rescaling of the mapped argument on the Euclidean domain. In addition, the spectral measure of the SEP Gaussian process is discrete (see Lemma~\ref{lemma:spectral.measure}) thus lacking the absolute continuity of that of the SE Gaussian process which is critical in establishing many of its properties~\citep{van+van:09}. As a result, the RKHS of the SEP Gaussian process for different scales do not follow the usual nesting property. We overcome these issues by using the special eigen structure of the SEP kernel and intriguing properties of the modified Bessel functions of the first kind. Some of the properties of the SE Gaussian process still hold, however, the proofs are remarkably different. 
Nevertheless, we show that the posterior contraction rate of the boundary by using the SEP Gaussian process is nearly minimax-optimal, which is $n^{-\alpha/(\alpha + 1)}$ up to a logarithmic factor, adaptively to the smoothness level $\alpha$ of the boundary. 
Section~\ref{sec:rkhs.GP} establishes a list of properties on the RKHS of the SEP Gaussian process, along with the contraction rate calculation and adaptation.

The third main contribution is that we provide some new estimates on Bessel functions, which are critical when establishing properties on the RKHS of the SEP Gaussian process. Similar to the second main contribution, these new estimates may be of interest by their own and are useful in broader contexts such as function estimation on spheres in addition to the boundary detection problem discussed here. 

In addition to establishing key theoretical properties, we also develop an efficient MCMC method for sampling posterior distribution based on a SEP Gaussian process prior using the explicit eigen structure of the SEP Gaussian process obtained in this paper \comment{(taking $O(n)$ time in each MCMC run)}. The algorithm is generic and hence can be used for posterior computation in other curve estimation problems on the circle such as directional data analysis using the SEP Gaussian process prior.

%The proposed method with theoretical supports and numerical confirmation is intriguing in the sense that it could be applied, with or without modifications, to many applications. Depending on the concrete context, an image may have multiple objects or objects with boundaries that are not star-shaped. If only one object is of interest in a multiple-object image \comment{such as tumor boundary detection in brain imaging}, the proposed method can be applied by using a mixture distribution as the noise distribution $f$.  If multiple objects are of interest, some adjustment may be needed, for example, in a ``divide-and-conquer" fashion by applying the method multiple times for each objects. \comment{There are many other interesting directions that the current paper leads to, such as when the boundary has heterogeneous smoothness (for example, the smoothness parameter $\alpha$ varies spatially), to address unknown number of objects and how to handle non star-shaped objects.} By providing a Bayesian framework to boundary detection problem whose good performance and robustness to misspecification of $f$ are confirmed by the numerical experiments, these extensions are possible based on the current work, for instance, using a hierarchical approach for the number of objects.  These are far from trivial and may depend on the specific application problems, and we here regard them as future directions.  

The paper is organized as follows. \comment{Section~\ref{sec:model} introduces the model and notations. }The general results on the posterior contraction rate are given in Section~\ref{section:main},  along with examples of priors and posterior rate calculation including a finite random series prior (for $d = 2$ and 3) and the squared exponential Gaussian process prior on $\mathbb{S}^1$ (for $d = 2$).
In Section~\ref{sec:rkhs.GP}, we study the corresponding RKHS of a squared exponential Gaussian process prior on $\mathbb{S}^1$, or equivalently, a squared exponential periodic Gaussian process on $[0,1]$, heavily relying on the properties of modified Bessel functions of the first kind. Section~\ref{section:bd.sampling} proposes an efficient Markov Chain Monte Carlo methods for computing the posterior distribution of the boundary using a randomly rescaled Gaussian process prior, for various noise distributions. Section~\ref{section:simulation} studies the performance of the proposed Bayesian estimator via simulations, under various settings for both binary images and Gaussian noised images. Section~\ref{sec:ch3.proof} contains proofs to all theorems and lemmas. Section~\ref{section:bessel} provides several results on the modified Bessel functions of the first kind.

%~\cite{Rud+Str:94b} considers star-shaped objects and proposed a nonparametric histogram-like contour estimators. The theory there assumes location-shift relationship between the parameters at two sides; the theory assumes known parameter values. They apply the method to weed data, which we can use. (two papers: 1st has more applications, but basically same thing) - Frequentist approach; theory doesn't take the full uncertainty into account. - this one is local method (estimate each $\gamma(\theta)$ using neighboring information)

%We are not to estimate the support of a function - a lot of literature on that. - say, estimation on star shaped sets.

%we can handle sparse data, to borrow information globally. ($m_1, m_2$ could differ, say $m_2$ very small - the marine data set)
%our rate is based on $n$, that means we are fine when $m_2$ is small.

% \cite{Candes+Donoho:00} discussed the wavelet fails in 2D
% Pol+Spo's AOS paper is very important.
% ~\cite{Banerjee+Gelfand:06} is important - read again.

%\begin{itemize}
%\item Maybe the fixed design cannot lead to the optimal rate
%\end{itemize}
%We want to use the Theorem 1 in~\cite{Ghosal+van:07}, or the simpler version Theorem 4 where the observations are independent. We shall establish the connection between the entire parameter space with the modelling of the boundary.

%ordered independent beta (OIB) priors for the parameters.

\section{Model and Notations}
\label{sec:model}
We consider a $d$-dimensional image $(X_{i}, Y_{i})_{i = 1}^n$ for $d = 2, 3, \ldots$, where $X_i$ is the location of the
$i$th observation and $Y_i$ is the image intensity. We consider the
locations within a $d$-dimensional \comment{fixed size hypercube, and we specifically use unit hypercube $T = [-1/2, 1/2]^d$ without loss of generality}. Depending on the scheme of collecting data, we have the following options for the distribution $P_{X_i}$ of $X_{i}$:
% $i \in I = \{(i_1, \ldots, i_d): i_k = 1,\ldots,m \text{ for } k = 1, \ldots, d  \}$, the pixel intensity $Y_{i} \in \{0,1\}$, and $X_{i}$  We denote the total number of observations by $n = m^d$.

\begin{itemize}
\item {\it Completely Random Design.} $X_{i}
  \overset{i.i.d.}{\sim} \mathrm{Uniform}(T).$
\item {\it Jitteredly Random Design.} Let $T_i$ be the $i$th \comment{block when partitioning $T$ into equal-spaced grids}. Then $X_{i}$ is chosen randomly at $T_i$, i.e. $X_{i} \sim \mathrm{Uniform}(T_{i})$ independently.
\end{itemize}

% We could assume $\xi_o$ and $\rho$ depend on $n$, but the gap
% doesn't decrease too fast.

\comment{
The region $\Gamma$ is assumed to be star-shaped with a known reference point $O \in \Gamma$, namely, for any point in $\Gamma$ the line segment from $O$ to that point is in $\Gamma$; see~\citep{Donoho:99} and~\cite[ Ch5]{Kor+Tsy:93}. If the image is start-shaped, this assumption is mild since in general a reference point can be easily detected by a preliminary estimator of the boundary or it could be directly given in many cases according to some subject matter knowledge. Images of more general shape can possibly be addressed by and ad-hoc ``divide and conquer'' strategy. 
The boundary $\gamma = \partial \Gamma$ is a $(d - 1)$-dimensional closed manifold.}
% new definition
In view of a converse of the Jordan curve theorem we represent the closed boundary $\gamma$ as a function indexed by $\sphere,$ i.e. $\gamma: \sphere \rightarrow \mathbb{R}^+: s \rightarrow \gamma(s)$.  We further assume that the boundary $\gamma$ is $\alpha$-smooth, i.e. $\gamma \in \bb{C}^\alpha(\mathbb{S}^{d-1})$, where $\bb{C}^\alpha(\mathbb{S}^{d-1})$ is the $\alpha$-H\"{o}lder class on $\sphere$. Specifically, let $\alpha_0$ be the largest integer strictly smaller than $\alpha$, then
\begin{equation}
\begin{split}
\bb{C}^{\alpha}(\sphere) = \{ f: \sphere \rightarrow \mathbb{R}^+, |f^{(\alpha_0)}(x) & -  f^{(\alpha_0)}(y)| \leq L_f \|x - y\|^{\alpha - \alpha_0} \\ & \text{ for } \forall x, y \in \sphere \text{ and some } L_f > 0 \},
\end{split}
\end{equation}
where $\|\cdot\|$ is the Euclidean distance.
A different definition of smoothness was used by~\cite{Mammen+Tsybakov:95} based on the class of sets in~\cite{Dudley:74}, which cover cases of unsmooth boundary but with smooth parameterization. Here we focus on the class of smooth boundary therefore it may be more natural to use the definition of $\bb{C}^{\alpha}(\sphere)$ directly. It may be noted that in our set-up, the boundary is not affected by reparameterization. 
%Alternatively, we may use a lower dimensional reparamterization, which has been investigated by Sederber (1983) and followed by Pati and Dunson (2008).
  % check the reference and their definitions.

%\emph{?2. A GP can be indexed by an arbitrary set. For $\gamma_0$ indexed by $S$, it can be regarded as a GP indexed by $[0, 2\pi]$ but with the restriction $\gamma(2\pi) = \gamma(0)$. Therefore it's subspace of $(W_t: t \in [0, 2\pi])$, and thus we can use the existing results of $W_t$ to bound $\gamma_0$. But how about the $\alpha$-smoothness? Is it transferable when the indexes change?}

We use $\theta$ to denote the triplet $(\xi, \rho, \gamma)$.  Let $\phi_i$ be the parameters at the $i$th location, i.e. $ \phi_i = \xi \Ind(X_{i}\in \Gamma) + \rho \Ind(X_{i} \in \Gamma^c)$ where $\Ind(\cdot)$ is the indicator function. The model assumes that $Y | X \sim P_{\theta}^n$ for some
$\theta$, where $P_{\theta}^n$ has density $\prod_{i = 1}^n p_{\theta, i}(Y_i) = \prod_{i = 1}^n f(Y_i; \phi_i)$ with respect to $\nu^n$. Let 
\begin{equation}
  \label{eq:2}
  d_n^2(\theta, \theta') = \frac{1}{n} \sum_{i} \int
  (\sqrt{p_{\theta,i}} - \sqrt{p_{\theta',i}})^2 d \nu
\end{equation}
be the average of the squares of the Hellinger distance for the
distributions of the individual observations.  Let $K(f, g) = \int f \log (f/g) d \nu$, $V(f, g) = \int f |\log (f/g)|^2 d \nu,$ and $\|\cdot\|_p$ denote the $L_p$-norm ($1 \leq p \leq \infty$).
We use $f \lesssim g$ if there is an universal constant $C$ such that $f \lesssim Cg$, and $f \asymp g$ if $f \lesssim g \lesssim f$. For a vector $x \in \mathbb{R}^d$, define $\|x\|_p = \{\sum_{i}^{d} |x_i|^p\}^{1/p}$ and $\|x\|_{\infty} = \max_{1 \leq i \leq p} |x_i|$.
For two sets $\Gamma$ and $\Gamma'$, we use $\Gamma \bigtriangleup \Gamma'$ for their symmetric difference and $\lambda(\Gamma \bigtriangleup \Gamma')$ for its corresponding Lebesgue measure. We also use $\lambda(\gamma, \gamma')$ for $\lambda(\Gamma \bigtriangleup \Gamma')$ when $\gamma = \partial \Gamma$ and $\gamma' = \partial \Gamma'$. 
% Observe that $\lambda(\gamma, \gamma') \lesssim \|\gamma - \gamma'\|_{\infty}$.
%For any $a, a' \in \bb{R}$, let $a \wedge a' = \min(a, a'),$ and $ a \vee a' = \max(a, a').$ 
% For $\epsilon > 0$, define $\log_- \epsilon = 0 \vee \log(1/\epsilon)$.
% For a function $f$ defined on $\mathbb{R}^d$, define $\|f\|_p = \{\int |f|^p d \mu \}^{1/p}$ and $\|f\|_{\infty} = \sup_{x} |f(x)|$ where $\mu$ is the Lebesgue measure on $\mathbb{R}^d$.

\section{Posterior convergence}
\label{section:main}

In the following sections, we shall focus on the jitteredly random design; the completely random design is more straightforward and follows the same rate calculation with minor modifications. 

\subsection{General theorem}
\label{section:assumption}
The likelihood function is given by
\begin{equation}
\label{eq:likelihood.general}
L(Y | X, \theta) = \prod_{i \in I_1} f(Y_i; \xi) \prod_{i \in I_2} f(Y_i; \rho),
\end{equation}
where $I_1 = \{i: X_i \in \Gamma\}$ and $I_2 = \{i: X_i \in \Gamma^c\}$. The parameters $(\xi, \rho) \in \Theta^*$, where $\Theta^*$ is a subset of $\Theta \times \Theta = \{(\xi, \rho): \xi \in \Theta, \rho \in \Theta \}$. The set $\Theta^*$ is typically given as the full parameter space $\Theta \times \Theta$ with some order restriction between $\xi$ and $\rho$. For instance, when $f(\cdot)$ is the Bernoulli distribution, then $\Theta^* = \{(\xi, \rho) \in \mathbb{R}^2: 0 < \rho < \xi < 1\}$ if the inside probability $\xi$ is believed to be larger than the outside probability. We assume that the distribution $f(\cdot)$ has the following regularity conditions:
\begin{itemize}
\item[(A1).] For fixed $\phi_0$,  we have $K(f(\cdot; \phi_0), f(\cdot; \phi)) \lesssim \|\phi - \phi_0\|^2$ and $V(f(\cdot; \phi_0), f(\cdot; \phi)) \lesssim \|\phi - \phi_0\|^2$ as $\|\phi - \phi_0\|^2 \rightarrow 0$;
\item[(A2).] \comment{There exist constants $C_0,b_0>0$ such that for $\phi_1,\phi_2$ with $\|\phi_1\|,\|\phi_2\|\le M$, we have 
$h^2(f(\cdot; \phi_1), f(\cdot; \phi_2)) \le C_0(1+M^{b_0})\|\phi - \phi_0\|^2$, where $h(f(\cdot; \phi), f(\cdot; \phi'))$ is the Hellinger distance between the two densities $f(\cdot; \phi)$ and $f(\cdot; \phi')$.}
\end{itemize}

\commentA{Assumptions (A1) and (A2) relate the divergence and distances between two distributions to the Euclidean distance between the corresponding parameters. Most common distributions where the parameters are bounded away from the boundary of their supports satisfy these two assumptions, particularly including all the distribution families discussed in the paper.   }

The observations $Y_{i}$'s are conditionally independent given parameters. In the following sections, we let $\theta_0$ denote the true value of the  parameter vector $(\xi_0, \rho_0, \gamma_0)$ generating the data, and the corresponding region with boundary $\gamma_0$ is denoted by $\Gamma_0$.

We shall denote the prior on $\theta$ by $\Pi$. By a slight abuse of notations, we denote the priors on $(\xi, \rho)$ and $\gamma$ also by $\Pi$.
We next present the abstract forms of the required prior distributions in order to satisfy the minimax-optimal posterior contraction rate later on.
The prior on $(\xi, \rho)$ is independent with the prior on $\gamma$ and satisfies that
\begin{itemize}
\item[(B1).] $\Pi(\xi, \rho)$ has a positive and continuous density on $\Theta^*$;
\item[(B2).] Sub-polynomial tails: there are some constants $t_1, t_2 > 0$ such that for any $M > 0,$ we have $\Pi(\xi: \xi \notin [-M, M]^p) \leq t_1 M^{- t_2}$ and  $\Pi(\rho: \rho \notin [-M, M]^p) \leq t_1 M^{- t_2}$.
\end{itemize}

The estimation and inference on $\gamma$ is of main interest. Therefore $(\xi, \rho)$ are considered as two nuisance parameters. When $\gamma$ is modeled nonparametrically, the contraction rate for $\theta$ is primarily influenced by $\gamma$. The following condition is critical to relate $d_n(\theta, \theta')$ to $\lambda(\gamma, \gamma')$, which will lead to the contraction rate for $\gamma$.
\begin{itemize}
\item[(C).] For given $(\xi_0, \rho_0) \in \Theta^*$, there exists a positive constant $c_{0,n}$ such that for arbitrary $(\xi, \rho) \in \Theta^*$,
$h(f(\cdot; \xi_0), f(\cdot; \rho)) + h(f(\cdot; \rho_0), f(\cdot; \xi)) \geq c_{0,n} > 0.$

\end{itemize}

\comment{Above we allow the constant $c_{0,n}$, which is usually $h(\xi_0,\rho_0)$, to depend on $n$ if we consider a sequence of true values $(\xi_0,\rho_0)$. Assumption (C) can be interpreted as quantifying the separation of the inside and outside densities in terms of the Hellinger distance. The separation becoming smaller with $n$ indicates the increasing level of difficulty of the problem. 
%Although we mainly discuss the situation that $c_0$ is a constant, the posterior contraction rates given by this paper are applicable to the case when $c_0 = c_{0, n}$ as a function of the sample size with minor modification; see the discussion in Remark~\ref{remark:separation} as an example. 
Assumption (C) holds for most commonly used distribution families $\{f(\cdot; \phi); \phi \in \Theta\}$ when $\Theta^*$ considers the order restriction between $\xi$ and $\rho$. Examples include but are not limited to:
\begin{itemize}
\item One-parameter family such as Bernoulli, Poisson, exponential distributions, and  $\Theta^* = \Theta^2 \cap \{(\xi, \rho): \rho < \xi \}$, or $\Theta^* = \Theta^2 \cap \{(\xi, \rho): \rho > \xi \}$.
\item Two-parameter family such as Gaussian distributions, and $\Theta^* = \Theta^2 \cap \{((\mu_1, \sigma_1), (\mu_2, \sigma_2)): \mu_1 > \mu_2, \sigma_1 = \sigma_2 \}$, or $\Theta^* = \Theta^2 \cap \{((\mu_1, \sigma_1), (\mu_2, \sigma_2)): \mu_1 > \mu_2, \sigma_1 > \sigma_2 \}$, or $\Theta^* = \Theta^2 \cap \{((\mu_1, \sigma_1), (\mu_2, \sigma_2)): \mu_1 = \mu_2, \sigma_1 > \sigma_2 \}$.
\end{itemize}
The assertions above can be verified by noting that, by keeping one argument fixed, the Hellinger distance increase in the other argument in each direction as that moves away from the fixed value in terms of the Euclidean distance.}
In practice, the order restriction is often naturally obtained depending on the concrete problems. For instance, in brain oncology, a tumor often has higher intensity values than its surroundings in a positron emission tomography scan, while for astronomical applications objects of interest emit light and will be brighter. In this paper, we use the abstract condition (C) to provide a general framework for various relevant applications.

Throughout this paper, we shall use $h(\phi, \phi')$ to abbreviate $h(f(\cdot; \phi), f(\cdot; \phi'))$.
The following general theorem gives a posterior contraction rate for parameters $\theta$ and $\gamma$. 

\begin{theorem}
\label{th:rate}
Let a sequence $\epsilon_n \rightarrow 0$ be such that $n \epsilon_n^2 / \log n$ is bounded away from $0$. Under Conditions (A1), (A2), (B1), (B2), if there exists Borel measurable subsets $\Sigma_n \subset \bb{C}^{\alpha}(\sphere)$ with $\sigma_n = \sup\{\|\gamma\|_{\infty}: \gamma \in \Sigma_n\}$ such that
\begin{align}
\label{eq:priormass}
-\log \Pi(\gamma: \lambda( \Gamma_0
\bigtriangleup \Gamma) \leq \epsilon_n^2) & \lesssim  n \epsilon_n^2, \\
 \label{eq:sieve}
-\log \Pi(\gamma \in \Sigma_n^c) & \gtrsim  n\epsilon_n^2, \\
\label{eq:entropy}
\log N(\epsilon_n^2/\sigma_n^{d - 1}, \Sigma_n, \|\cdot\|_{\infty}) & \lesssim n \epsilon_n^2,
\end{align}
then for the entire parameter $\theta = (\xi, \rho, \gamma)$, we have that for every $M_n \rightarrow \infty$,
\begin{equation}
\label{eq:theta.rate}
\P_{\theta_0}^{(n)} \Pi(\theta: d_n(\theta, \theta_0) \geq M_n \epsilon_n | X^{(n)},Y^{(n)}) \rightarrow 0.
\end{equation}
Further, \comment{if also Condition (C) holds}, then for the boundary $\gamma$, we have that for every $M_n \rightarrow \infty$,
\begin{equation}
\label{eq:gamma.rate}
\P_{\theta_0}^{(n)} \Pi(\gamma: \lambda(\gamma, \gamma_0) \geq M_n \epsilon_n^2/c_{0,n}^2 | X^{(n)}, Y^{(n)}) \rightarrow 0.
\end{equation}
\end{theorem}

Equation~\eqref{eq:gamma.rate} claims that if the rate for $\theta$ is $\epsilon_n$, then the boundary $\gamma$ has the rate $\epsilon_n^2/c_{0,n}^2$ in terms of the discrepancy metric  $\lambda(\cdot, \cdot)$ and can be faster than $n^{-1/2}$ which is an interesting aspect of a boundary detection problem. \comment{The condition that $ n \epsilon_n^2 / \log n$ is bounded away from 0 is not restrictive since its sufficient condition $\epsilon_n \gtrsim n^{-c}$ for some $c < 1/2$ is expected for nonparametric problems. It \comment{ensures} that the parametric components with priors of sub-polynomial tails are not influential for the posterior contraction rate compared to the nonparametric part.  }

\begin{remark}
\comment{ 
It follows immediately that $ \|\xi_0 - \xi\| \lesssim \epsilon_n$ and $\|\rho_0 - \rho\| \lesssim \epsilon_n$.  These two parameters are not of our interest and are actually estimable at $n^{-1/2}$ rate.  To see this, \commentA{we assume the separation $c_{0, n}^2$ decays at a rate such that  $\epsilon_n^2/c_{0, n}^2 \rightarrow 0$ (particularly $c_{0, n} = c_0$ satisfies this)}, one may use  a two-step semi-parametric procedure: first estimate the boundary curve consistently using Theorem~\ref{th:rate}, remove a small section of pixels neighboring the estimated boundary and then estimate $(\xi,\rho)$ based on observations at the remaining pixels. The $n^{-1/2}$-rate possibly also holds for the original posterior of $(\xi,\rho)$ and will follow if a semiparametric Bernstein–-von Mises theorem can be established using the rate in Theorem~\ref{th:rate} as a preliminary rate; see \cite{Castillo}. }
\end{remark}

%\comment{\begin{remark}
%	\label{remark:separation} 
%	If the Assumption (C) uses sample size dependent $c_0 = c_{0, n}$, then  equation~\eqref{eq:gamma.rate} will be changed to 
%	\begin{equation}
%	\P_{\theta_0}^{(n)} \Pi(\gamma: \lambda(\gamma, \gamma_0) \geq M_n \epsilon_n^2/c_{0, n}^2 | X^{(n)}, Y^{(n)}) \rightarrow 0.
%	\end{equation}
%	Therefore the posterior contraction rate of $\gamma$ is $\epsilon_n^2/c_{0, n}^2$ in terms of the metric $\lambda(\cdot, \cdot)$.  The proof of this result just requires substitutions of $\epsilon_n^2$ by $\epsilon_n^2/c_{0, n}^2$ in all equations after ~\eqref{eq:dummy5} in the proof to Theorem~\ref{th:rate}. 
%\end{remark}}

In the next two subsections, we consider two general classes of priors suitable for applications of Theorem~\ref{th:rate}. 

\subsection{Rate calculation using finite random series priors}
\label{section:random.series.prior}

The boundary $\gamma_0$ is a function on $\sphere$, which can be regarded also as a function on $[0, 1]^{d-1}$ with periodicity restrictions. \comment{We construct a sequence of finite dimensional approximation for functions in $\bb{C}^{\alpha}([0,1]^{d-1})$ with periodicity restriction by linear combinations of the first $J$ elements of a collection of fixed basis functions.} Let $\bsxi = \bsxi_J = (\xi_1, \ldots, \xi_J)^T$ be the vector formed by the first $J$ basis functions, and $\bs{\beta}_{0,J}^T \bsxi$ be a linear approximation to $\gamma_0$ with $\|\bs{\beta}_{0,J}\|_{\infty} < \infty$.
%which follows naturally when the basis functions are orthogonal (i.e. )and further $\|\xi_j\|_2 \; (j = 1, 2, \cdots)$ is bounded.
%For example, we shall use trigonometric polynomial basis for $d = 2$ and spherical harmonics for $d = 3$. By the orthogonality of these functions, it follows that the $k$th element of $\bs{\beta}_{0,J}$ is given by the inner product $\int_{[0,1]^{d-1}} \gamma_0(\bs{\omega}) \xi_k(\bs{\omega}) d \bs{\omega}$.
We assume that the basis functions satisfy the following condition: 
\begin{itemize}
\item[(D).] $\underset{1 \leq j \leq J}{\max}\|\xi_j\|_{\infty} \leq t_3 J^{t_4}$ for some constants $t_3, t_4 \geq 0$. 
\end{itemize} 
%This includes most commonly used basis functions, such as B-splines, Fourier series, polynomials and wavelets. %~\citep[Remark 2 in][]{Shen+Ghosal:14}. 

{\it Priors.} \comment{We use a random series prior for $\gamma$ induced from $\bs{\beta}^T \bsxi$ through the number of basis functions $J$ and
the corresponding coefficients $\bs{\beta}$ given $J$. For the simplicity of notations, we use $\bs{\beta}$ ($\bs{\beta}_0$) for $\bs{\beta}_J$ ($\bs{\beta}_{0,J}$) when $J$ is explicit from the context.}
Let $\Pi$ stand for the probability mass function (p.m.f.) of $J$, and also for the prior for $\bs{\beta}$.  \comment{satisfying the following conditions:}
\begin{itemize}
\item[(E1).] $
-\log \Pi(J > j) \gtrsim j\log j,$ and $-\log \Pi(J = j) \lesssim j \log j. 
$
\item[(E2).] $-\log \Pi(\| \bs{\beta} - \bs{\beta}_0 \|_1 \leq \epsilon | J) \lesssim J \log (1/\epsilon),$ and $ \Pi(\bs{\beta} \notin [-M, M]^J | J) \leq J \exp \{-C M^{2}\}
$
for some constant $C$.
\end{itemize}
\comment{
For instance, a Poisson prior on $J$ and a multivariate normal for $\bs{\beta}$ meet the required conditions. }

\comment{
Finite random series priors with a random number of coefficients form a very tractable flexible class of priors offering alternative to Gaussian process priors. Their properties including asymptotic behavior of posterior distributions have been thoroughly studied by \cite{Arbel+:13} and \cite{Shen+Ghosal:14}.}

We derive conditions to obtain the posterior contraction rate as follows.
\begin{theorem}
\label{th:random.series}
Let $\epsilon_n$ be a sequence such that $\epsilon_n \rightarrow 0$ and $n \epsilon_n^2 / \log n$ is bounded away from $0$, and  $J_n\le n$ be a sequence such that $J_n \rightarrow \infty$. Under conditions \comment{(A1), (A2), (B1), (B2), (C)}, (D), (E1) and (E2), if $\epsilon_n, J_n$ satisfy 
\begin{equation}
\| \gamma_0 - \bs{\beta}^T_{0,J_n} \bsxi \|_{\infty} \leq \epsilon_n^2/2,  \quad
c_1 n \epsilon_n^2 \leq J_n \log J_n  \leq  J_n \log n \leq c_2 n \epsilon_n^2,
\end{equation}
\comment{for some constants $c_1 > 0$ and $c_2 > 0$, then  the posterior contraction rate for $\gamma$ in terms of the distance $\lambda(\cdot, \cdot)$ is $\epsilon_n^2/c_{0,n}^2$}.
\end{theorem}

\begin{remark}
The optimal value of $J_n$, say $J_{n, \alpha}$ typically depends on the degree of smoothness
$\alpha$. We can use a fixed value $J = J_n$ when $\alpha$ is given. The posterior distribution can be easily computed, for example, by a Metropolis-Hastings algorithm. If $\alpha$ is unknown, one will need to put a prior on $J$ and reversible-jump MCMC may be needed for computation.
\end{remark}

\begin{example}[Trigonometric polynomials] 
\label{example:d=2}
	\rm
	\comment{
For the case $d = 2$ (2D image), we use trigonometric polynomials $\{1, \cos 2\pi j \omega, \sin 2\pi j \omega, \ldots :\omega \in [0,1]  \}$ as the basis. It is known if $\gamma_0$ is $\alpha$-smooth, we have $\|\gamma_0  - \bs{\beta}^T_{0,j} \bsxi \|_{\infty} \lesssim j^{-\alpha}$ for some appropriate choice of $\bs{\beta}_{0,j}$~\citep[cf.][]{Jackson:30}.} Therefore according to  Theorem~\ref{th:random.series}, we can obtain the rate $\epsilon_n$ by equating $J_n^{-\alpha} \asymp \epsilon_n^2$ and $ J_n \log J_n \asymp n \epsilon_n^2 $, which gives the following rate $\epsilon_n$ and the corresponding $J_n$:
\begin{equation}
J_n \asymp n^{1/(\alpha + 1)} (\log n)^{-1/(\alpha + 1)}, \quad \epsilon_n^2\asymp n^{-\alpha/(\alpha + 1)}( \log n)^{\alpha/(\alpha + 1)}.
\end{equation}
\end{example}

\begin{example}[Spherical harmonics]\rm
	\label{example:d=3} 
For 3D images $(d = 3)$, periodic functions on the sphere can be expanded in the spherical harmonic basis functions. Spherical harmonics are eigenfunctions of the Laplacian on the sphere. It satisfies condition (D) and more technical details and the analytical expressions of spherical harmonics can be found in~\cite[Chapter 2]{Terras:13}, while MATLAB implementation is available in~\cite{Ennis:05}.
Let $K_n$ be degree of the spherical harmonics, then the number of basis functions are $J_n = K_n^2$. The approximation error for spherical harmonics is $J_n^{-\alpha/2}$~\citep[Theorem 4.4.2]{Dai+Xu:13}. Therefore we can obtain the posterior contraction rate by equating
$
J_n^{-\alpha/2} \asymp \epsilon_n^2$ and $J_n \log J_n \asymp n \epsilon_n^2,
$
which gives
\begin{equation}
J_n \asymp n^{2/(\alpha + 2)} (\log n)^{-2/(\alpha + 2)}, \quad \epsilon_n^2 \asymp n^{-\alpha/(\alpha + 2)} (\log n)^{\alpha/(\alpha + 2)}.
\end{equation}
\end{example}

\subsection{Rescaled squared exponential Gaussian process prior on $\mathbb{S}^{1}$}
\label{subsection:rescaled.GP}
We use a rescaled squared exponential Gaussian process (GP) to induce priors on $\gamma$ when $d = 2$. Specifically, let $W$ be a GP with the squared exponential kernel function $K(t, t') =  \exp(- \| t - t' \|^2)$, where $t, t' \in \mathbb{S}^{1}$ and $\|\cdot\|$ is the Euclidean distance. Let $W^{a} = (W_{a t}, t \in \mathbb{S}^{1})$ be the scaled GP with scale $a > 0$, whose covariance kernel becomes $K_{a}(t, t')=  \exp(- a^2 \| t - t' \|^2).$ The rescaling factor $a$ acts as a smoothing parameter and allows us to control smoothness of a sample path from the prior distribution. 
%Throughout the paper, the rescaling factor $a > 0$. 

When $d = 2$, it is natural to use the map $ Q: [0,1] \rightarrow \mathbb{S}^1, \omega \rightarrow (\cos 2 \pi \omega, \sin 2 \pi \omega)$ as in~\cite{MacKay:98}, then  by Lemma~\ref{lemma:bd.isometric}, the squared exponential kernel $K_a(\cdot, \cdot)$ on $\mathbb{S}^1$ has the equivalent RKHS as of the kernel $G_a(t_1, t_2)$ on $[0,1]$ defined by 
\begin{eqnarray*}
 G_a(t_1, t_2) &=&  
\exp({-a^2 \{(\cos 2\pi t_1 - \cos 2 \pi t_2)^2 + (\sin 2\pi t_1 - \sin 2\pi t_2)^2 \} }) \\
&=& \exp\{- 4\scale^2 \sin^2(\pi t_1 - \pi t_2) \}.
\end{eqnarray*}
We call $G_{\scale}(\cdot, \cdot)$ on the unit interval as {\it squared exponential periodic} (SEP) kernel.  Theorem~\ref{thm:GP.rate} gives the posterior contraction rate if a rescaled SEP Gaussian process is used as the prior. 

\begin{theorem}
	\label{thm:GP.rate}
	Let Conditions (A1), (A2), (B1), (B2) and (C) hold. 
	
	 (i). Deterministic rescaling:  If the smoothness level $\alpha$ is known, and we choose $a = a_n = n^{1/(\alpha + 1)} (\log n)^{-2/(\alpha + 1)}$, then the posterior contraction rate in Theorem~\ref{th:rate} is determined by $\epsilon_n^2 = n^{-\alpha/(\alpha + 1)} (\log n)^{2\alpha/(\alpha + 1)}$. 
	
	(ii). Random rescaling: If the rescaling factor $a$ follows a gamma prior, then the contraction rate in Theorem~\ref{th:rate} is determined by \comment{$\epsilon^2_n = n^{-\alpha/(\alpha + 1)} (\log n)^{2\alpha/(\alpha + 1)}$} for any $\alpha > 0$.  
\end{theorem}

Therefore, when the underlying smoothness level $\alpha$ is unknown, the SEP Gaussian process prior can adapt to $\alpha$ in a hierarchical Bayesian approach by assigning the rescaling parameter an appropriate prior such as a gamma distribution~\citep{van+van:09}. 

The proof to Theorem~\ref{thm:GP.rate} relies on an extensive study of the corresponding RKHS of the rescaled SEP Gaussian process (see Section~\ref{sec:rkhs.GP}). We also obtain the eigen structure of the SEP Gaussian process analytically, leading to efficient MCMC method for posterior sampling for various distribution families (see Section~\ref{section:bd.sampling}).

\comment{
\begin{remark}
	\label{remark:minimax}
	The rates obtained in Examples~\ref{example:d=2}, \ref{example:d=3} and Theorem~\ref{thm:GP.rate} are optimal in the minimax sense up to a logarithmic factor; see~\cite[Chapter 7]{Kor+Tsy:93}. By a uniform strengthening of Theorem~3 of \cite{Ghosal+van:07}, the conclusion can be strengthened to uniform in $(\xi_0, \rho_0, \gamma_0)\in \Theta_0$ if on  $\Theta_0$ Assumptions (A1) and (C) hold uniformly and $\|\gamma_0\|_{\infty} \leq C_0$ for a universal constant $C_0 > 0$.  
\end{remark} 
}

\comment{ 
\subsection{Rescaled Gaussian process with the heat kernel} 
For a Gaussian process prior, various kernels may be used in addition to the squared exponential kernel, for example, the heat kernel Gaussian processes studied by~\cite{Castillo+:14}.  Without introducing the technical details on heat kernel theory, a Gaussian process with the heat kernel on $\mathbb{S}^{d - 1}$ can be represented as 
\begin{equation}
W^T(x) = \sum_{k = 0}^{\infty} \sum_{l = 1}^{N_k(d)} e^{-\lambda_k T/2} Z_{k, l} e_{k,l}(x), 
\end{equation}
for any $x \in \mathbb{S}^{d - 1}$, where $Z_{k, l}$ are independent standard normal variables indexed by $(k, l)$, $e_{k,l}(\cdot)$ are the eigenvalues and orthonormal eigenfunctions of spherical harmonics of order $k$ whose dimension is $N_k(d) = (2k + d - 2)/(d - 2){{d + k - 3}\choose{k}}$, and $\lambda_k$ are the eigenvalues of the natural Laplacian on $\mathbb{S}^{d - 1}$.   Starting from eigen decomposition of the heat kernel, the rescaling is applied directly to the eigenvalues $\lambda_k$. 
% which is similar to the rescaling in~\cite{van+van:09} but the indexing space becomes a manifold. 
Note that this is a different rescaling strategy compared to the SEP Gaussian process where the rescaling is applied to the distances between two points. If we randomly rescale the process by letting $T$ follow a gamma prior, then similar results as in Theorem~\ref{thm:GP.rate} hold (possibly with a different logarithmic factor) based on the study of RKHS of $W^T(\cdot)$ available in~\cite{Castillo+:14} which are parallel to Lemma~\ref{lemma:approx.RHKS} to Lemma~\ref{lemma:constant.rkhs}, following the argument in proving Theorem~\ref{thm:GP.rate}. 
% Although the rescaled GP prior with the heat kernel achieves the theoretical adaptivity as the rescaled GP prior with SEP, its implementation may be challenging since it requires analytical forms of  the eigenvalues and eigenfunctions. 
}

\section{RKHS of SEP Gaussian processes}
\label{sec:rkhs.GP}
The RKHS of a GP plays a critical role in calculating the posterior contraction rate. There has been an extensive study of the RKHS of a GP indexed by $[0,1]^{d - 1}$~\citep[e.g.][]{van+van:07,van+van:09}. \comment{A GP indexed by $\sphere$ can be naturally related to a GP indexed by $[0, 1]^{d - 1}$} by a surjection $Q: [0, 1]^{d - 1} \rightarrow \sphere$ (for example, using the spherical coordinate system). Define the following kernels on $[0, 1]^{d - 1}$: $G(s_1, s_2) = K(Qs_1, Qs_2)$ for any $s_1, s_2 \in [0, 1]^{d - 1}$. Let $\mathbb{H}$ be the RKHS of the GP defined by the kernel $K$, equipped with the inner product $\langle \cdot, \cdot \rangle_{\mathbb{H}}$ and the RKHS norm $\|\cdot\|_{\mathbb{H}}$. \comment{For the GP with covariance kernel $G$, we denote the RKHS, its inner product and norm by  $\mathbb{H}'$, $\langle \cdot, \cdot\rangle_{\mathbb{H}'}$ and $\|\cdot\|_{\mathbb{H}'}$ respectively.} Then the following lemma shows that the two RKHSs related by the map $Q$ are isomorphic. 
\begin{lemma}
	\label{lemma:bd.isometric}
	$(\mathbb{H}', \|\cdot\|_{\mathbb{H}'})$ and $(\mathbb{H}, \|\cdot\|_{\mathbb{H}})$ are isometric; the conclusion also holds when we use the $\|\cdot\|_{\infty}$ norm. 
\end{lemma}

However, if $K$ is the squared exponential kernel on $\sphere$, the kernel $G(\cdot, \cdot)$ is no longer a squared exponential kernel on $[0, 1]^{d - 1}$. More importantly, it is not even stationary for general $d > 2$. The case $d = 2$ is an exception, for which
the RKHS can be studied via an explicit treatment such as  analytical eigen decompositions of its equivalent kernel on the unit interval. 
We next focus on the case $d = 2$, and study the RKHS of a GP $W^a = \{W^a_t: t \in [0, 1]\}$ with the SEP kernel $G_a(\cdot, \cdot)$ which was used in Section~\ref{subsection:rescaled.GP}.

The SEP kernel $G_a(\cdot, \cdot)$ is stationary since
$
G_{\scale}(t_1, t_2) = \phi_{\scale}(t_1 - t_2),$ where $\phi_{\scale}(t) = \exp\{- 4 \scale^2 \sin^2(\pi t) \}.
$
The following result gives the explicit form of the spectral measure $\mu_\scale$ of the process $W_t^{\scale}$. Let $\delta_x$ be the Kronecker delta function and $I_n(x)$ be the modified Bessel function of the first kind with order $n$ and argument $x$ where $n \in \mathbb{Z}$ and $x \in \mathbb{R}$. 
\begin{lemma}
\label{lemma:spectral.measure} 
We have $
\phi_a(t) = \int e^{-its} d \mu_{\scale}(s),
$ where $\mu_{\scale}$ is a symmetric and finite measure and given by $\mu_{\scale} = \sum_{n = -\infty}^{\infty} e^{-2\scale^2} I_n(2\scale^2) \delta_{2\pi n}.$

In addition, the Karhunen--Lo\`{e}ve expansion of the covariance kernel is $
G_a(t, t') = \sum_{k = 1}^{\infty} v_k(a)\psi_k(t) \psi_k (t'),$ where the eigenvalues are given by 
$$v_1(a) = e^{-2\scale^2} I_0(2\scale^2), \quad v_{2j}(a) = v_{2j + 1}(a) = e^{-2\scale^2} I_j(2\scale^2), \; j \geq 1,$$ with eigenfunctions $\psi_j(t)$, $j=1,2,\ldots$ given by the Fourier basis functions $\{1, \cos 2\pi t, \sin 2\pi t, \ldots\}$ in that order. 

\end{lemma}

The measure $\mu_a$ is the so-called spectral measure of $W^a$. Existing literature~\cite[e.g.][]{van+van:07,van+van:09} studied convergence properties of rescaled GP on $[0, 1]^{d - 1}$ relying on the absolute continuity of the  spectral measure and \comment{the scaling relation} $\mu_a(B) = \mu_1(a B)$. However, Lemma~\ref{lemma:spectral.measure} shows that the spectral measure of a SEP Gaussian process is discrete and  the simple relationship $\mu_a(B) = \mu_1(a B)$ does not hold any more. We instead heavily use properties of modified Bessel functions to study the RKHS of a SEP Gaussian process. 

\comment{
Note that the discrete measure $\mu_{\scale}$ has subexponential tails since  
\begin{equation}
\int e^{|s|} \mu_{\scale} (d s) = \sum_{n = -\infty}^{\infty} e^{2 \pi |n|} e^{-2 \scale^2} I_n(2 \scale^2) \leq 2 e^{-2 \scale^2} \sum_{n = 0}^{\infty}  e^{2 \pi n }I_n(2 \scale^2)
\end{equation}
which is bounded by $ 2 e^{\scale^2 (e^{2 \pi} + e^{-2 \pi} - 2)} < \infty$ (Proposition~\ref{prop:bessel.basic} (a)). 
} 
\comment{The following Lemma~\ref{lemma:rkhs} describes the RKHS $\mathbb{H}^{\scale}$ of the rescaled process $W^a$,  as real parts of a closed set containing complex valued functions. }
\begin{lemma}
\label{lemma:rkhs}
The RKHS $\mathbb{H}^{\scale}$  of the process $W^a$ is the set of real parts of all functions in 
\begin{align}
\{h: [0, 1] \rightarrow \mathbb{C}, h(t) & = \sum_{n = -\infty}^{\infty} e^{-i t 2 \pi n} b_{n, \scale} e^{-2\scale^2} I_n (2\scale^2), \\ &
b_{n, \scale} \in \mathbb{C},  \sum_{n = -\infty}^{\infty} |b_{n, \scale}|^2 e^{-2\scale^2} I_n(2\scale^2) < \infty
 \},
\end{align} 
and it is equipped with the squared norm
\begin{equation}
\|h\|^2_{\mathbb{H}^{\scale}} = \sum_{n = -\infty}^{\infty} |b_{n, \scale}|^2 e^{-2\scale^2} I_n(2\scale^2) =  \sum_{n = -\infty}^{\infty}  \frac{1}{e^{-2\scale^2}I_n(2\scale^2)} \left|\int_0^1 h(t) e^{it2\pi n} dt \right|^2. 
\end{equation}
\end{lemma}

We then consider the approximation property of $\mathbb{H}^{\scale}$ to an arbitrary smooth function $w \in \mathbb{C}^{\alpha}[0, 1]$. Unlike the approach approximating $w$ by a convolution of $w_0$ with a smooth function as used in~\cite{van+van:07,van+van:09}, we use a finite Fourier approximation to $w$. 
\begin{lemma}
\label{lemma:approx.RHKS}
For any function $w \in \mathbb{C}^{\alpha}[0, 1]$, there exists constants $C_w$ and $D_w$ depending only on $w$ such that $
\inf\{\|h\|^2_{\mathbb{H}^{\scale}}: \|w - h\|_{\infty} \leq C_w \scale^{-\alpha}\} \leq D_w \scale,
$  as $\scale \rightarrow \infty.$
\end{lemma}

Lemma~\ref{lemma:entropy.lambda} obtains an entropy estimate using Proposition~\ref{prop:bessel.moment} on modified Bessel functions. 
\begin{lemma}
\label{lemma:entropy.lambda}
Let $\mathbb{H}^{\scale}_1$ be the unit ball of the RHKS of the process $W^{\scale} = (W^{\scale}_t: 0 \leq t \leq 1)$, then we have 
$
\log N(\epsilon, \mathbb{H}^{\scale}_1, \|\cdot\|_{\infty}) \lesssim \max(\scale, 1) \cdot \left\{\log ({1}/{\epsilon})\right\}^2. 
$
\end{lemma}

As a corollary of Lemma~\ref{lemma:entropy.lambda}, using the connection between the entropy of the unit ball of the RKHS and the small ball probability~\citep{Kue+Li:93,Li+Linde:99}, we have the following estimate of the small ball probability. 
\begin{lemma}[Lemma 4.6 in~\cite{van+van:09}]
\label{lemma:small.ball.prob}
For any $a_0 > 0$, there exits constants $C$ and $\epsilon_0$ that depend only on $a_0$ such that, for $a \geq a_0$ and $\epsilon \leq \epsilon_0$, 
\begin{equation}
-\log \P \left( \underset{0 \leq t \leq 1}{\sup} |W_t^{\scale}| \leq \epsilon \right) \leq C a \left(\log \frac{a}{\epsilon}\right)^2. 
\end{equation}
\end{lemma}

The proof of Theorem~\ref{thm:GP.rate}(ii) needs a nesting property of the RKHS of $W^{\scale}$ for different values of $\scale$. Lemma 4.7 in~\cite{van+van:09} proved that $\sqrt{a}\mathbb{H}_1^a \subset \sqrt{b} \mathbb{H}_1^b$ if $a \leq b$ for a squared exponential GP indexed by $[0, 1]^{d - 1}$. For the SEP Gaussian process prior, this does not hold but can be modified up to a global constant. 
\begin{lemma}
\label{lemma:nesting.rkhs} 
If $a \leq b$, then $\sqrt{a} \mathbb{H}_1^a \subset \sqrt{c b}  \mathbb{H}_1^b$ for a universal constant $c$.  
\end{lemma}
When $a \downarrow 0$, sample paths of $W^a$ tend to concentrate to a constant value by the following lemma. This property is crucial in controlling the variation of sample paths for small $a$. 
\begin{lemma}
\label{lemma:constant.rkhs}
For $h \in \mathbb{H}_1^{\scale}$, we have $|h(0)| \leq 1$ and $|h(t) - h(0)| \leq 2 \sqrt{2} \pi \scale t$ for every $t \in [0, 1]$. 
\end{lemma}

%With minor modification of Theorem 3.1 in~\cite{van+van:09}, we can obtain the following result.
%
%Let $W$ be a stationary GP on $R^d$ satisfying conditions specified by Theorem 3.1 in~\cite{van+van:09}. For the randomly rescaled process $W^A$, where $A^d$ follows a Gamma prior independently, we can show that: if $\gamma_0 \in \bb{C}^{\alpha}([0,1]^d)$, then there exist a sieve $\Sigma_n$ such that equation~\eqref{eq:entropy}, \eqref{eq:priormass} and~\eqref{eq:sieve} hold with $\epsilon_n = n^{-\alpha/(2\alpha + 2d )}$ up to a log factor, and therefore the rate is
%\begin{equation}
%\epsilon_n^2 = n^{-\alpha/(\alpha + d)},
%\end{equation}
%up to a log factor.
%
%Note that this rate is technically different from the rate obtained in Theorem 3.1~\citep{van+van:09}. This is because of that $\epsilon_n$-close in $d_n^2$ corresponds to $\epsilon_n^2$-close in Lebesgue measure (or $L_\infty$ distance). Therefore, equation~\eqref{eq:priormass} is to equate $n \epsilon_n^2$ with the ball $\Sigma_n \bigcap B(\gamma_0, \epsilon_n^2)$ instead of $\Sigma_n \bigcap B(\gamma_0, \epsilon_n)$ in the original theorem.
%This difference is also present when using the random series prior.

\section{Sampling Algorithms}
\label{section:bd.sampling}

% \subsection{Priors and Posterior Sampling}
We assume that the origin (center of the image) is inside the boundary, and thus use it as the reference point to represent the observed image in a polar coordinate system as $(\bs{\omega}, \bs{r}; \bs{Y})$, where \comment{$(\bs{\omega}, \bs{r}) = \{(\omega_i, r_i)\}_{i = 1}^n$ are the locations using polar coordinates and $\bs{Y} = \{Y_i\}_{i = 1}^n$ are the image intensities. }Let $\gamma$ be a closed curve, and $\bs{\gamma}$ be values of $\gamma$ evaluated at each $\bs{\omega}$.

% \subsection{Squared exponential periodic kernels for 2D images} 
For most kernels, the eigenfunctions and eigenvalues are challenging to obtain although there are several exceptions~\citep[Ch.~4.3]{Rasmussen+Williams:06}. Therefore a randomly rescaled GP prior may be infeasible in practice since the numerical inversion of a covariance matrix is often needed when no analytical forms are available. However, thanks to the analytical eigen decomposition the SEP kernel in Lemma~\ref{lemma:spectral.measure}, we can implement this theoretically appealing prior in an computationally efficient way. If a curve  $\gamma(\omega) \sim \GP(\mu(\omega), G_a(\cdot, \cdot)/\tau)$, we then have the equivalent representation
$
\gamma(\omega) = \mu(\omega) + \sum_{k = 1}^{\infty} z_k \psi_k(\omega),
$
where $z_k \sim N(0, v_k(a) / \tau)$ independently.

The modified Bessel function of the first kind used in $v_k(a)$'s is a library function in most software, such as {\tt besselI} in {\tt R} language. Figure~\subref*{fig:decay} shows that eigenvalues decay very fast when $a = 1, 10$. When $a$ increases, the smoothness level of the kernel decreases. In practice we typically do not use values as large as 100 since then the kernel becomes very close to the identity matrix and thus the resulting prior path becomes very rough. 
The fast decay rate of \comment{$v_k(a)$ for fixed $a$ (Proposition~\ref{prop:bessel.basic}} (c3)) guarantees that some suitable finite order truncation to the Karhunen-Lo\`{e}ve expansion is able to approximate the kernel function well. Suppose we use $L = 2J + 1$ basis functions, then the truncated process is given by
$
\gamma(\omega) = \sum_{k = 1}^{L} z_k \psi_k(\omega) + \mu(\omega).
$
\comment{Let ${\rm PVE}_a = \sum_{k = 1}^L v_k(a) / \sum_{k = 1}^{\infty} v_k(a)$ be the percentage of variance explained by the first $L$ basis functions, where the denominator $\sum_{k = 1}^{\infty} v_k(a) = e^{-2 a^2} \sum_{k = -\infty}^{\infty} I_k(2 a^2) = 1$ according to the definition of $v_k$ in Lemma~\ref{lemma:spectral.measure} and the properties of the modified Bessel function of the first kind in Proposition~\ref{prop:bessel.basic}. }
Figure~\subref*{fig:varianceExplained} shows that with $J = 10$, we are able to explain at least $98 \%$ of all the variability for a reasonable range of $a$'s from 0 to 10. 
\begin{figure}[h]
\centering
\subfloat[]{\includegraphics[width = 0.45\textwidth]{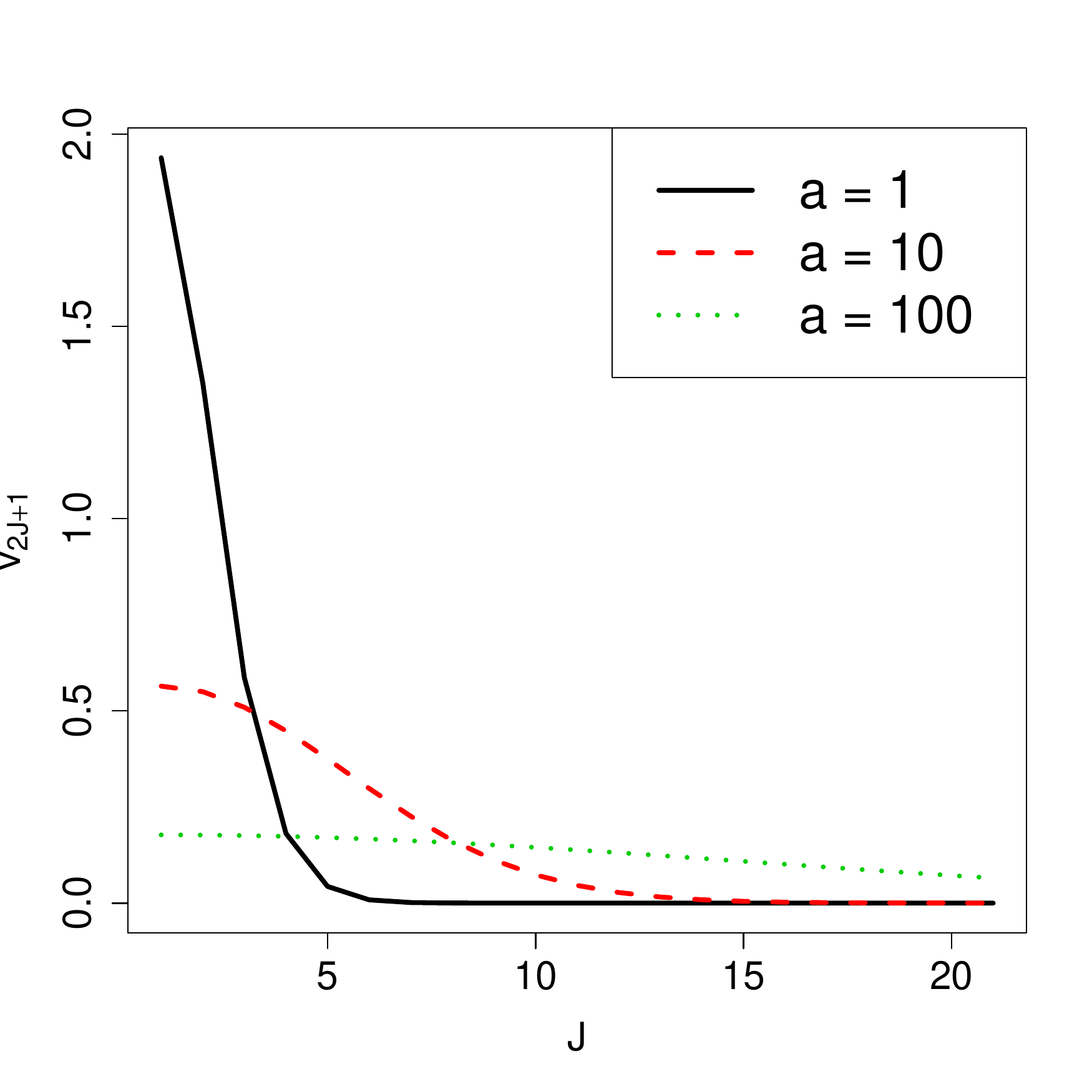}\label{fig:decay}}
\subfloat[]{\includegraphics[width = 0.45\textwidth]{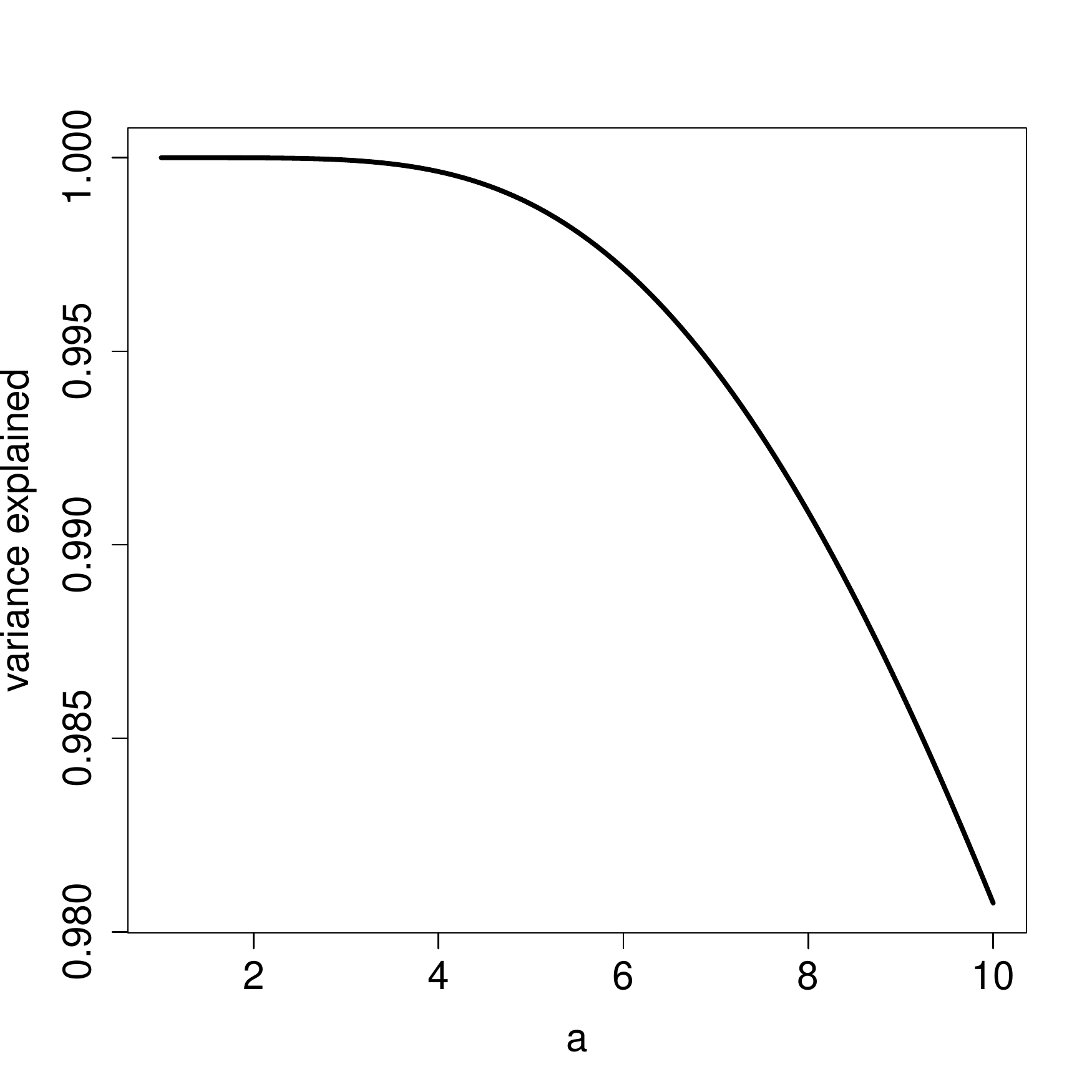} \label{fig:varianceExplained}}
\caption{Decay rate of the eigenvalues of the squared exponential kernel. Figure (a) plots the values of $v_{2J + 1}(a)$ at $J = 0, \ldots, 20$ when $a = 1, 10, 100$.  
	Figure (b) plots the percentage of the variance explained (i.e. ${\rm PVE}_a$) if using the first 21 ($J = 10$) basis functions at different values of $a$. }
%\label{fig:eigen}
\end{figure}

Let $\boldsymbol{\Psi}$ be the $n$ by $L$ matrix with the $k$th column comprising of the evaluations of $\psi_k(\cdot)$ at the components of $\bs{\omega}$, and $\bs{\mu}$ comprising of the evaluations of $\mu(\cdot)$ at the components of $\bs{\omega}$. Then the Gaussian process prior for the boundary curve can be expressed as
\begin{equation}
\bs{\gamma} = \boldsymbol{\Psi} \bs{z} + \bs{\mu}; \quad \bs{z} \sim N(0, \bs{\Sigma}_a/\tau)
\end{equation}
where $\bs{\Sigma}_a =  \mathrm{diag}(v_1(a), \ldots, v_L(a)).$ We use the following priors for the hyper-parameters involved in the covariance kernel: $
\tau \sim \mathrm{Gamma}(500, 1)$ and $
a \sim \mathrm{Gamma}(2, 1)$. For the mean $\mu(\cdot)$, we use a constant 0.1. Note that here we also can use empirical Bayes to estimate the prior mean by any ordinary one dimensional change-point method and an extra step of smoothing. However, our numerical investigation shows that our method is robust in terms of the specification of $\mu(\cdot)$.

The priors for $(\xi, \rho)$ depend on the error distributions. We also need to use the order information between the parameters to keep the two regions distinguishable. We use OIB, OIN, OIG for ordered independent \comment{beta}, normal and gamma distributions respectively. If not specified explicitly, the parameters are assumed to be in a decreasing order. It is easy to see that this convention is for simplicity of notations, and any order between the two region parameters are allowed in practice.
Below we give the conjugate priors for $(\xi, \rho)$ for some commonly used noise distributions:
\begin{itemize}
\item Binary images: the parameters are the probabilities $(\pi_1, \pi_2) \sim \mathrm{OIB}(\alpha_1, \beta_1, \alpha_1, \beta_1)$;
\item Gaussian noise: the parameters are the mean and standard deviation $(\mu_1, \sigma_1, \mu_2, \sigma_2)$ with the priors to be $ (\mu_1, \mu_2) \sim \mathrm{OIN}(\mu_0, \sigma_0^2, \mu_0, \sigma_0^2)$ and $(\sigma_1^{-2}, \sigma_2^{-2}) \sim \mathrm{OIG}(\alpha_2, \beta_2,  \alpha_2, \beta_2)$.
\item Poisson noise: the parameters are the rates $(\lambda_1, \lambda_2) \sim \mathrm{OIG}(\alpha_3, \beta_3, \alpha_3, \beta_3);$
\item Exponential noise: the parameters are the rates $(\lambda_1, \lambda_2) \sim \mathrm{OIG}(\alpha_4, \beta_4, \alpha_4, \beta_4).$
\end{itemize}
In fact, any error distributions with conjugacy properties conditionally on the boundary can be directly used. \comment{We specify the hyper-parameters such that the corresponding prior distributions are spread out.} For example, in the simulation, we use $\alpha_1 = \beta_1 = 0$ for binary images; we use $\mu_0 = \bar{y}, \sigma_0 = 10^3$ and $\alpha_2 = \beta_2 = 10^{-2}$ for Gaussian noise.

We use the slice sampling technique~\citep{Neal:03} within the Gibbs sampler to draw samples from the posterior distribution for $(\bs{z}, \xi, \rho, \tau, a)$. Below is a detailed description of the sampling algorithms for binary images.
\begin{enumerate}
\item Initialize the parameters to be $\bs{z} = \bs{0}, \tau = 500$ and $a = 1$. The parameters $(\xi, \rho) = (\pi_1, \pi_2)$ are initialized by the maximum likelihood estimates (MLE) given the boundary to be $\mu(\cdot)$.
\item $\bs{z} | (\pi_1, \pi_2, \tau, a, \bs{Y})$ : the conditional posterior density of $\bs{z}$ (in a logarithmic scale and up to an additive constant) is equal to
\begin{eqnarray*}
 \lefteqn{\sum_{i \in I_1} \log f(Y_i; \pi_1) + \sum_{i \in I_2} \log f(Y_i; \pi_2) -\frac{\tau \bs{z}^T
\bs{\Sigma}_{a}^{-1} \bs{z}}{2} }
\\ && =  N_1 \log \frac{\pi_1 (1 - \pi_2)}{\pi_2(1 - \pi_1)}
+ n_1 \log \frac{1 - \pi_1}{1 - \pi_2} - \frac{\tau \bs{z}^T
\bs{\Sigma}_{a}^{-1} \bs{z}}{2},
\end{eqnarray*}
where $n_1 = \sum_i \Ind({r}_i < {\bs \gamma}_i)$ and $N_1 = \sum_i \Ind({r}_i < {\bs{\gamma}}_i)  Y_i$. We use slice sampling one-coordinate-at-a-time for this step.
\item $\tau | (\bs{z}, a) \sim \mathrm{Gamma}(a^*, b^*)$, where $a^* = a + L/2$ and $b^* = b + \bs{z}^T \bs{\Sigma}_{a}^{-1} \bs{z}/2$;
\item $(\pi_1, \pi_2) | (\bs{z}, \bs{Y}) \sim \mathrm{OIB}(\alpha_1 + N_1,
\beta_1 + n_1 - N_1, \alpha_1 + N_2, \beta_1 + n_2 - N_2)$, where $N_2$ is the count of 1's outside $\gamma$ and $n_2$ is the number of observations outside $\gamma$.
\item $a|\bs{z}, \tau$: use slice sampling by noting that the conditional posterior density of $a$ (in a logarithmic scale and up to an additive constant) is equal to
\begin{equation}
-\frac{\log |\bs{\Sigma}_{a}|}{2}  - \frac{\tau \bs{z}^T \bs{\Sigma}_{a}^{-1} \bs{z}}{2}  + \log \frac{a}{ e^{a}} = -\sum_{k = 1}^{L} \frac{\log v_k(a) }{2} -  \sum_{k = 1}^{L} \frac{\tau z_k^2}{2 v_k(a)}  + \log a - a.
\end{equation}

\end{enumerate}

The above algorithm is generic beyond binary images. For other noise distributions, the update of $\tau$ and $a$ are the same. The update of $\bs{z}$ and $(\xi, \rho)$ in Step 2 and Step 4 will be changed using the corresponding priors and conjugacy properties. For example, for Gaussian noise, the parameters $(\xi, \rho)$ are $(\mu_1, \sigma_1, \mu_2, \sigma_2)$, and the conditional posterior density (in the logarithmic scale and up to an additive constant) used in Step 2 is changed to
	\begin{equation}
	- n_1 (\log \sigma_1 - \log \sigma_2) -  \sum_{i \in I_1} \frac{(y_i - \mu_1)^2}{2 \sigma_1^2} -  \sum_{i \in I_2} \frac{(y_i - \mu_2)^2}{2 \sigma_2^2}  - \frac{\tau \bs{z}^T
	\bm{\Sigma}_{a}^{-1} \bs{z}}{2}.
	\end{equation}
For Step 4, the conjugacy is changed to $$(\mu_1, \mu_2)|(\bs{z}, \sigma_1, \sigma_2, \bs{Y}) \sim \mathrm{OIN}, \quad (\sigma_1^{-2}, \sigma_2^{-2})|(\bs{z}, \mu_1, \mu_2, \bs{Y}) \sim \mathrm{OIG}.$$
Similarly, it is straightforward to apply this algorithm to images with Poisson noise, exponential noise or other families of distributions with ordered conjugate prior.

\section{Simulations}
\label{section:simulation}
\subsection{Numerical results for binary images}
We use jitteredly random design for locations $(\bs{\omega}, \bs{r})$ and three cases for boundary curves:
\begin{itemize}
\item Case B1. Ellipse given by   
$
r(\omega) = {b_1 b_2}/{\sqrt{(b_2 \cos \omega)^2 + (b_1 \sin \omega)^2}},
$
where $b_1 \geq b_2$ and $\omega$ is the angular coordinate measured from the major axis. We set $b_1 = 0.35$ and $b_2 = 0.25$.

\item Case B2. Ellipse with shift and rotation: centered at (0.1, 0.1) and rotated by $60^{\circ}$ counterclockwise. We use this setting to investigate the influence of the specification of the reference point.

\item Case B3. Regular triangle centered at the origin with the height to be 0.5. We use this setting to investigate the performance of our method when the true boundary is not smooth at some points.
\end{itemize}
We keep using $\pi_2 = 0.2$ and vary the values of $\pi_1$ to be $(0.5, 0.25)$. For each combination of $(\pi_1, \pi_2)$, the observed image is $m \times m$ where $m = 100, 500$ (therefore the total number of observations is $n = m^2$). The MCMC procedure is iterated 5,000 times after 1,000 steps burn-in period. For the estimates, we calculate the Lebesgue error (area of mismatched regions) between the estimates and the true boundary. For the proposed Bayesian approach, we use the posterior mean as the estimate and construct a variable-width uniform credible band. Specifically, let $\{{\gamma}_i(\omega)\}_{1000}^{5000}$ be the posterior samples and $(\widehat{\gamma}(\omega), \widehat{s}(\omega))$ 
be the posterior mean and standard deviation functions derived from $\{{\gamma}_i(\omega)\}$. For each MCMC run, we calculate the distance $u_i = \|(\gamma_i-\widehat{\gamma})/s\|_\infty=\sup_\omega\{ |\gamma_i(\omega) - \widehat{\gamma}(\omega)| /\widehat{s}(\omega) \}$ and obtain the 95th percentile of all the $u_i$'s, denoted as $L_0$. Then a 95\% uniform credible band is given by $[\widehat{\gamma}(\omega) - L_0 \widehat{s}(\omega), \widehat{\gamma}(\omega) + L_0 \widehat{s}(\omega)]$.

We compare the proposed approach with a maximum contrast estimator (MCE) which first detect boundary pixels followed by a post-smoothing via a penalized Fourier regression. In the 1-dimensional case, the MCE selects the location which maximizes the differences of the parameter estimates at the two sides, which is similar to many pixel boundary detection algorithms discussed in~\citep{Qiu:05}.
% Given the observation $(x_i, y_i)_{i = 1}^n$ in 1-dimensional case where $x_i$ is the location in an increasing order and $y_i$ is the binary intensity values.
% Let $(\xi_i, \rho_i)$ be the MLE of $(\xi, \rho)$ given that $x_i$ is the change point, then it follows that $\E \xi_i - \E \rho_i$ will be maximized at the true change point. The MCE is defined as $x_0$ which maximizes $\xi_i - \rho_i$.
% One can interpret the MCE as an estimator to maximize the contrast at the two sides of the boundary.
In images, for a selected number of angles (say 1000 equal-spaced angles from 0 to $2\pi$), we choose the neighboring bands around each angle and apply MCE to obtain the estimated radius and then smooth those estimates via a penalized Fourier regression. Note that unlike the proposed Bayesian approach, a joint confidence band is not conveniently obtainable for the method of MCE, due to its usage of a two-step procedure.

As indicated by Table~\ref{table:numerical}, the proposed Bayesian method has Lebesgue errors typically less than $2.5\%$. In addition, the proposed method outperforms the benchmark method MCE significantly. We also observe that the MCE method is highly affected by the number of basis functions; in contrast, the proposed method adapts to the smoothness level automatically.  The comparison between Case B1 and Case B2 shows that the specification of the reference point will not influence the performance of our methods since the differences are not significant compared to the standard error.

% latex table generated in R 3.0.2 by xtable 1.7-3 package
% Wed Feb 04 08:08:56 2015
\begin{table}[ht]
\centering
\caption{Lebesgue errors ($\times 10^{-2}$) of the methods based on $100$ simulations. The standard errors are reported below in the parentheses.}
\label{table:numerical}
% latex table generated in R 3.0.2 by xtable 1.7-3 package
% Wed Feb 04 08:11:31 2015
\begin{tabular}{l|ccc|ccc}
  \hline
 & \multicolumn{3}{r|}{$m = 100, (\pi_1, \pi_2 ) = (0.50, 0.20)$} & \multicolumn{3}{r}{$m = 100, (\pi_1, \pi_2 ) = (0.25, 0.20)$} \\ \hline
& Case B1 & Case B2 & Case B3 & Case B1 & Case B2 & Case B3  \\ \hline
Bayesian method & 0.64 & 0.67 & 2.26 & 0.71 & 0.8 & 2.36 \\
 & (0.02) & (0.02) & (0.02) & (0.03) & (0.03) & (0.03) \\
MCE with 5 bases & 6.57 & 6.58 & 6.03 & 6.39 & 10.09 & 7.03 \\
   & (0.25) & (0.21) & (0.07) & (0.19) & (0.20) & (0.11) \\
MCE with 31 bases & 8.75 & 7.84 & 5.96 & 9.19 & 11.8 & 7.86 \\
   & (0.18) & (0.19) & (0.10) & (0.16) & (0.19) & (0.14) \\ \hline
\end{tabular}
\end{table}

Figures~\ref{fig:ellipse} and~\ref{fig:triangle} confirm the superior performance of the proposed method compared with the smoothed MCE method with 5 and 31 basis functions when the true boundary curve is an ellipse, an ellipse with shift and rotation and a triangle. Even in the case of $\pi_1 = 0.25$ where the contrast at two sides of the boundary is small, the proposed method is still  able to capture the boundary when $m = 500$. This observation is consistent with the result derived from the infill asymptotics \comment{when the number of data points increase within images of fixed size}. In addition, we also obtain joint credible bands using the samples drawn from the joint posterior distribution. \comment{Figure~\ref{fig:trace} plots the trace plots of $(a, \pi_1, \pi_2)$ and the histogram of $a$ to illustrate the mixing and convergence of the posterior samples for Case B1 when $m = 500$ and the true parameters $(\pi_1, \pi_2) = (0.25, 0.20)$.  Similar plots are obtained for other scenarios but are not presented here. }

\begin{figure}
\captionsetup[subfloat]{farskip=0pt,captionskip=0pt}
\subfloat[Case B1: $m = 100, \pi_1 = 0.50$]{\includegraphics[width = 0.25\textwidth]{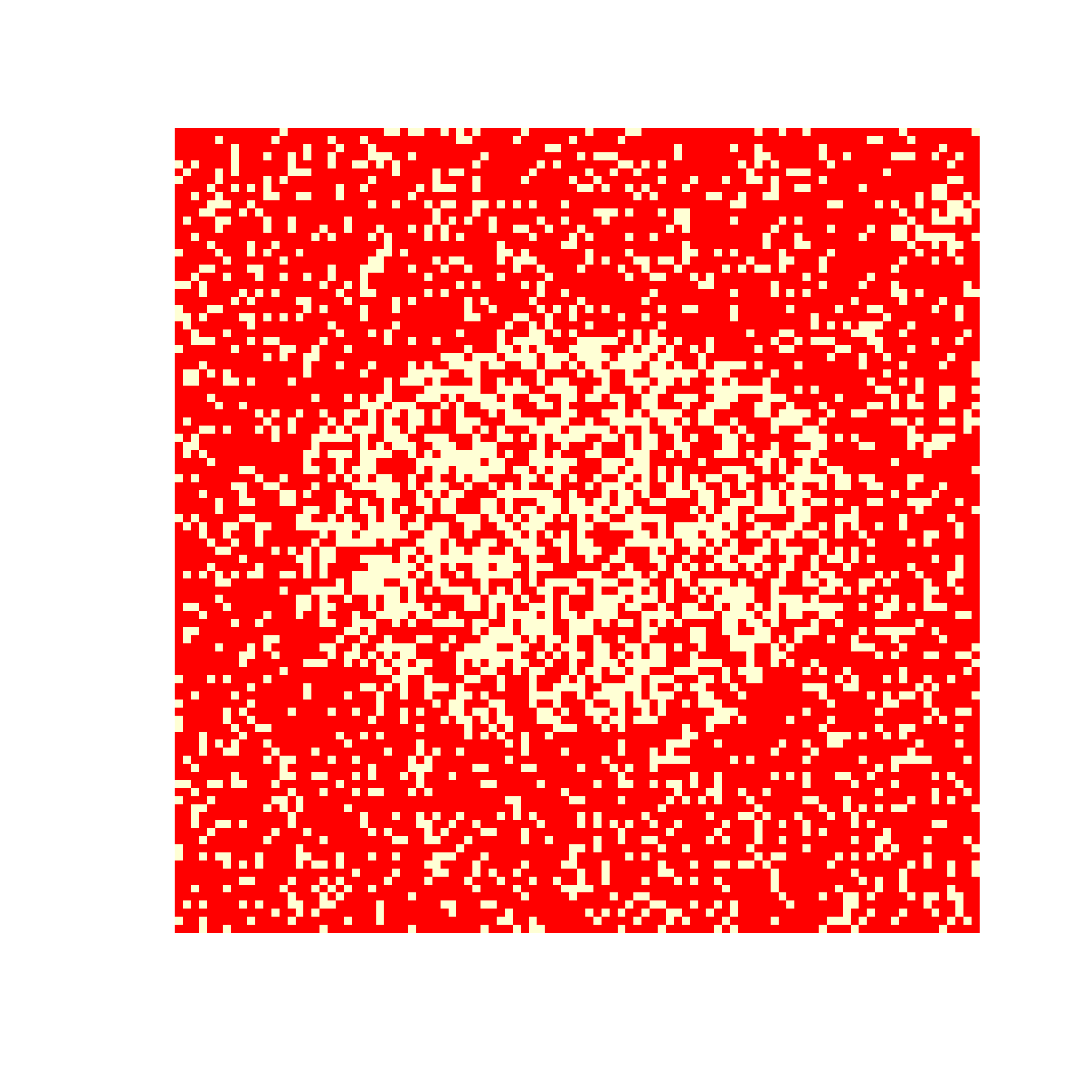}}
\subfloat[Bayesian Est.]{\includegraphics[width = 0.25\textwidth]{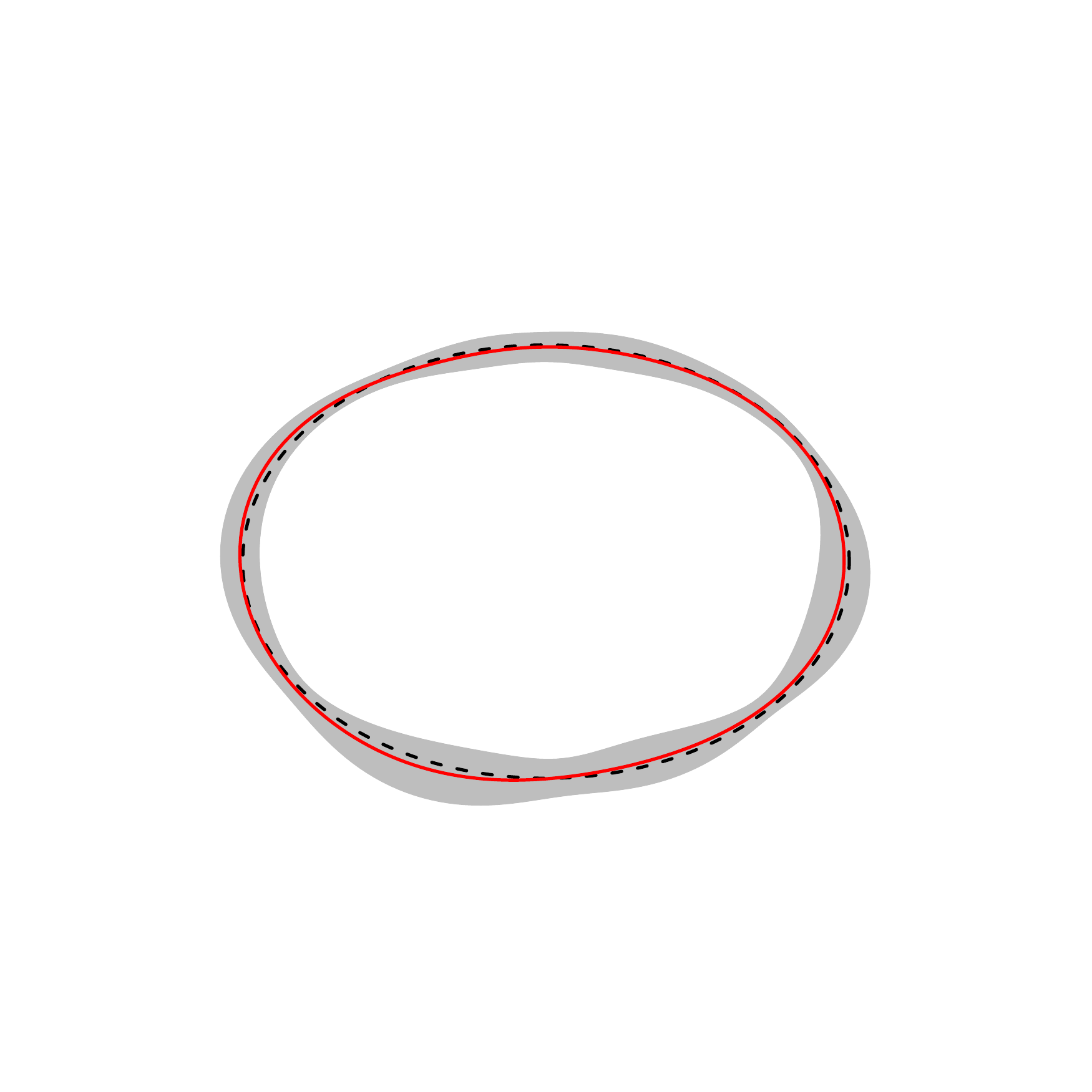}}\subfloat[MCE (5 basis)]{\includegraphics[width = 0.25\textwidth]{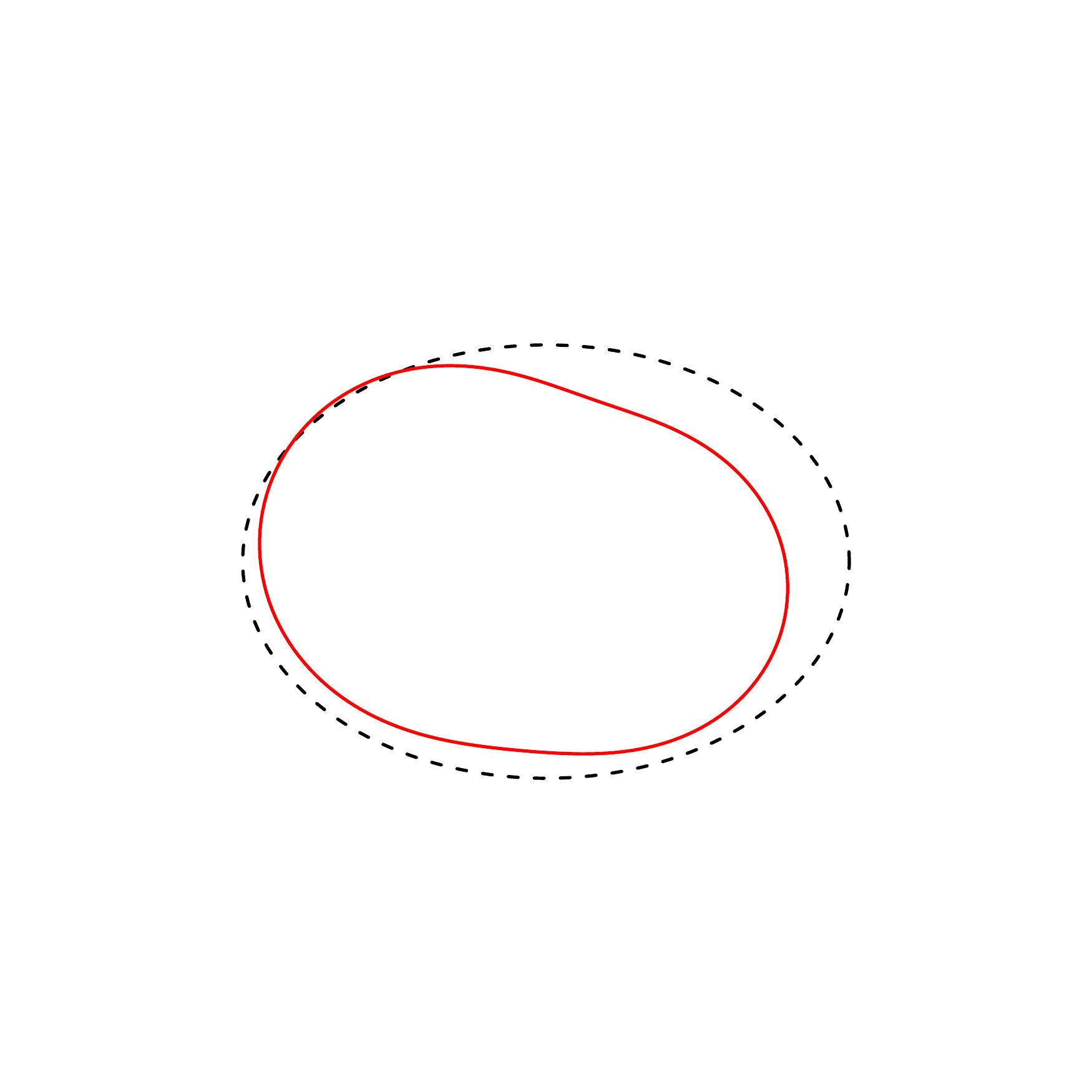}}
\subfloat[MCE (31 basis)]{\includegraphics[width = 0.25\textwidth]{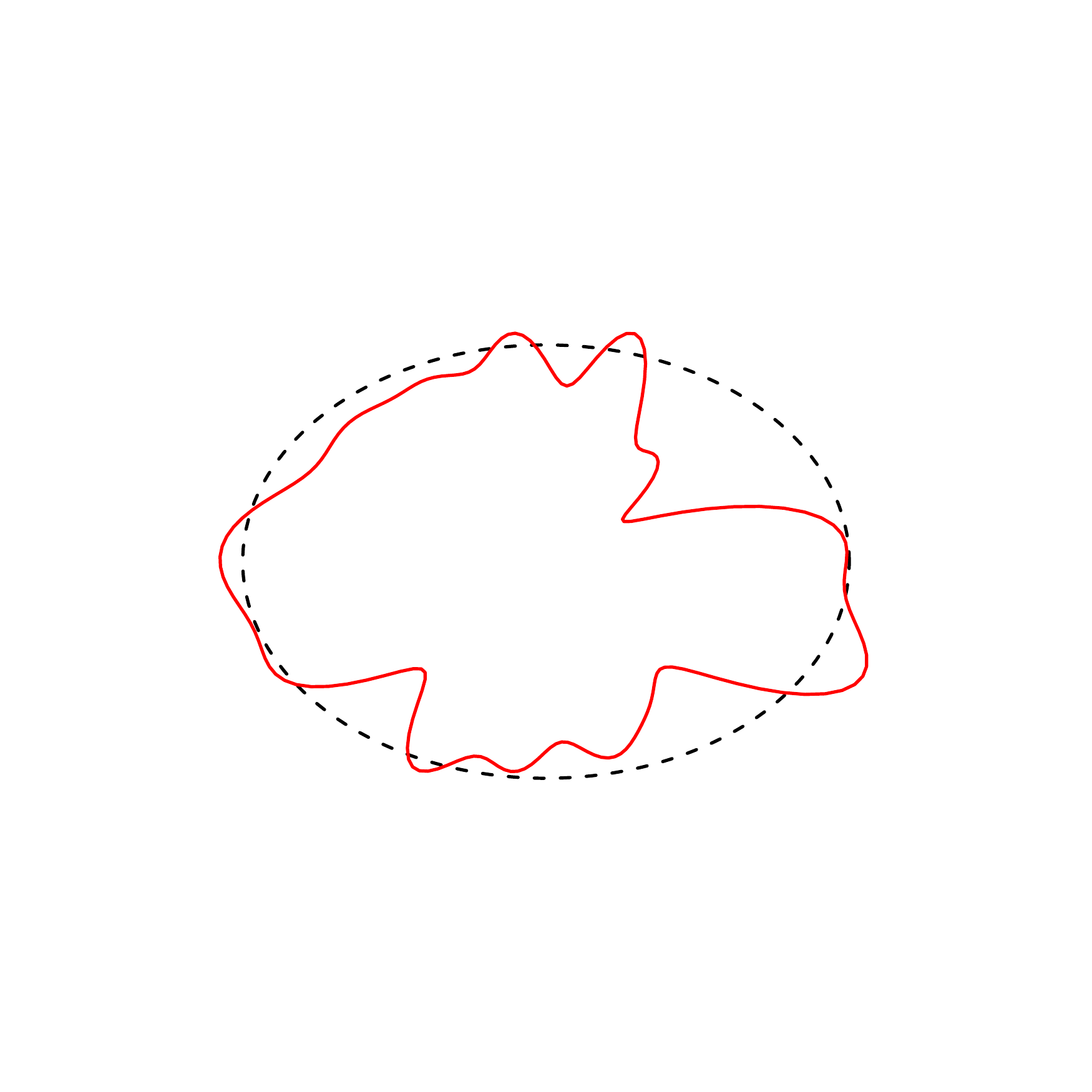}} \\
\subfloat[Case B1: $m = 500$, $\pi_1 = 0.25$]{\includegraphics[width = 0.25\textwidth]{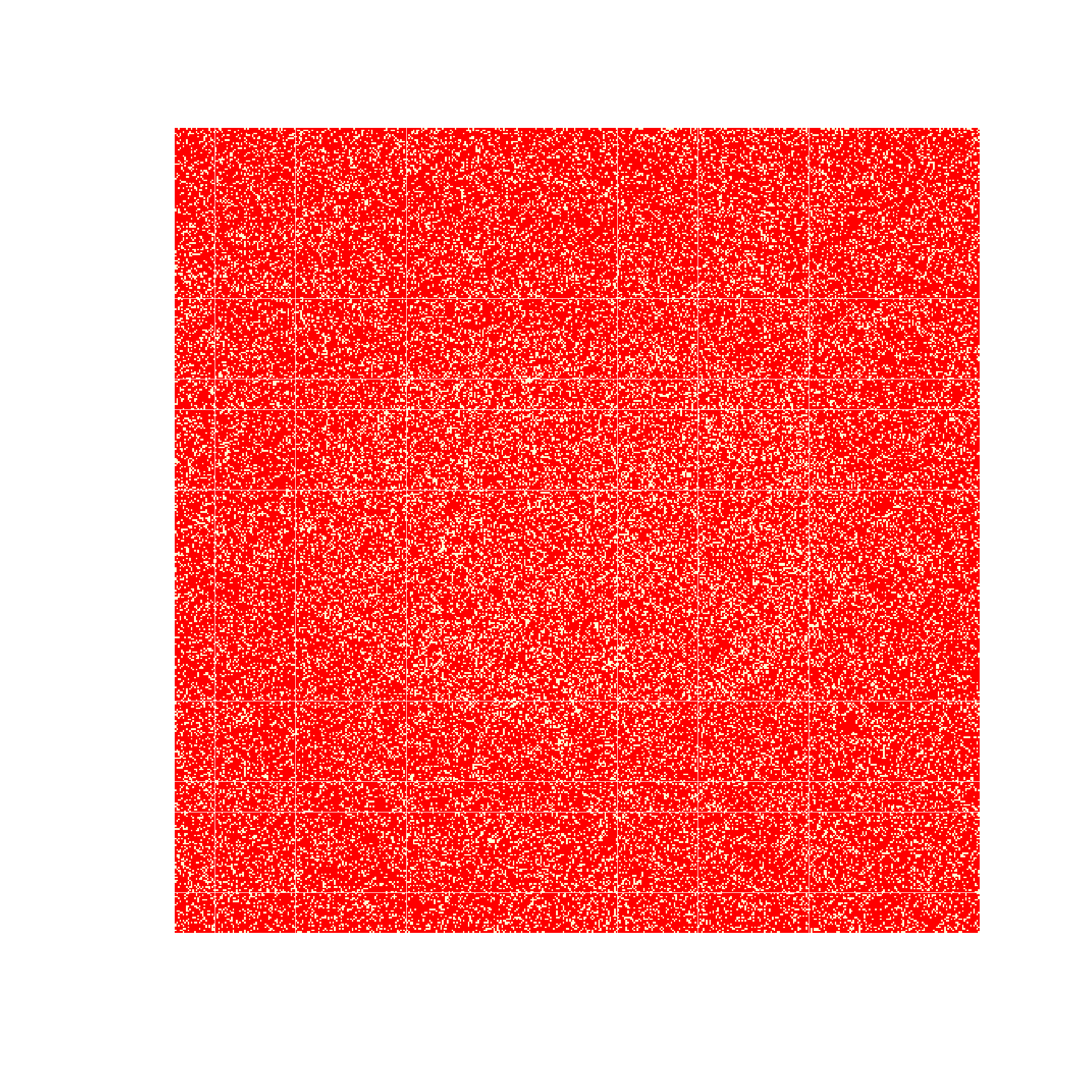}}
\subfloat[Bayesian Est.]{\includegraphics[width = 0.25\textwidth]{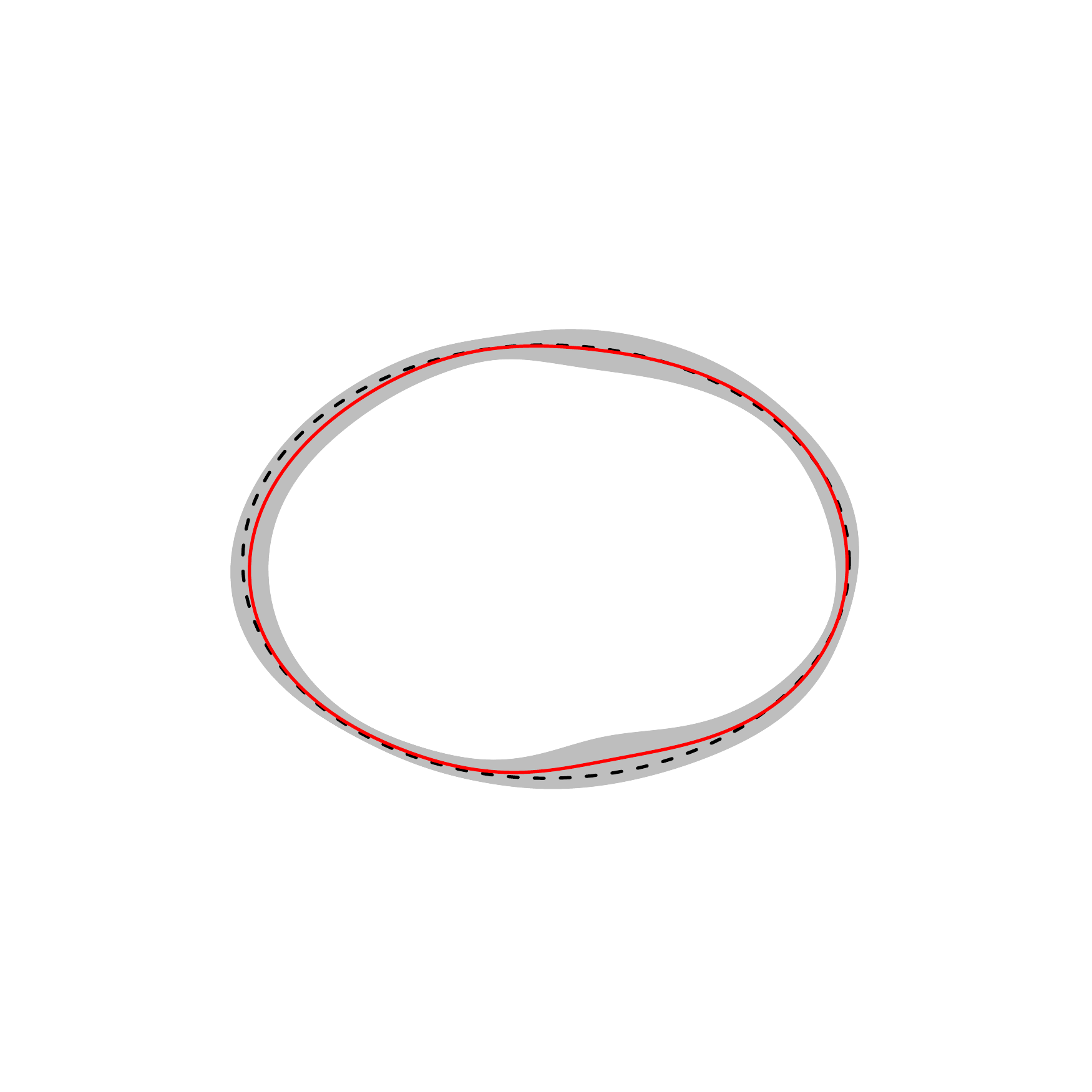}}\subfloat[MCE (5 basis)]{\includegraphics[width = 0.25\textwidth]{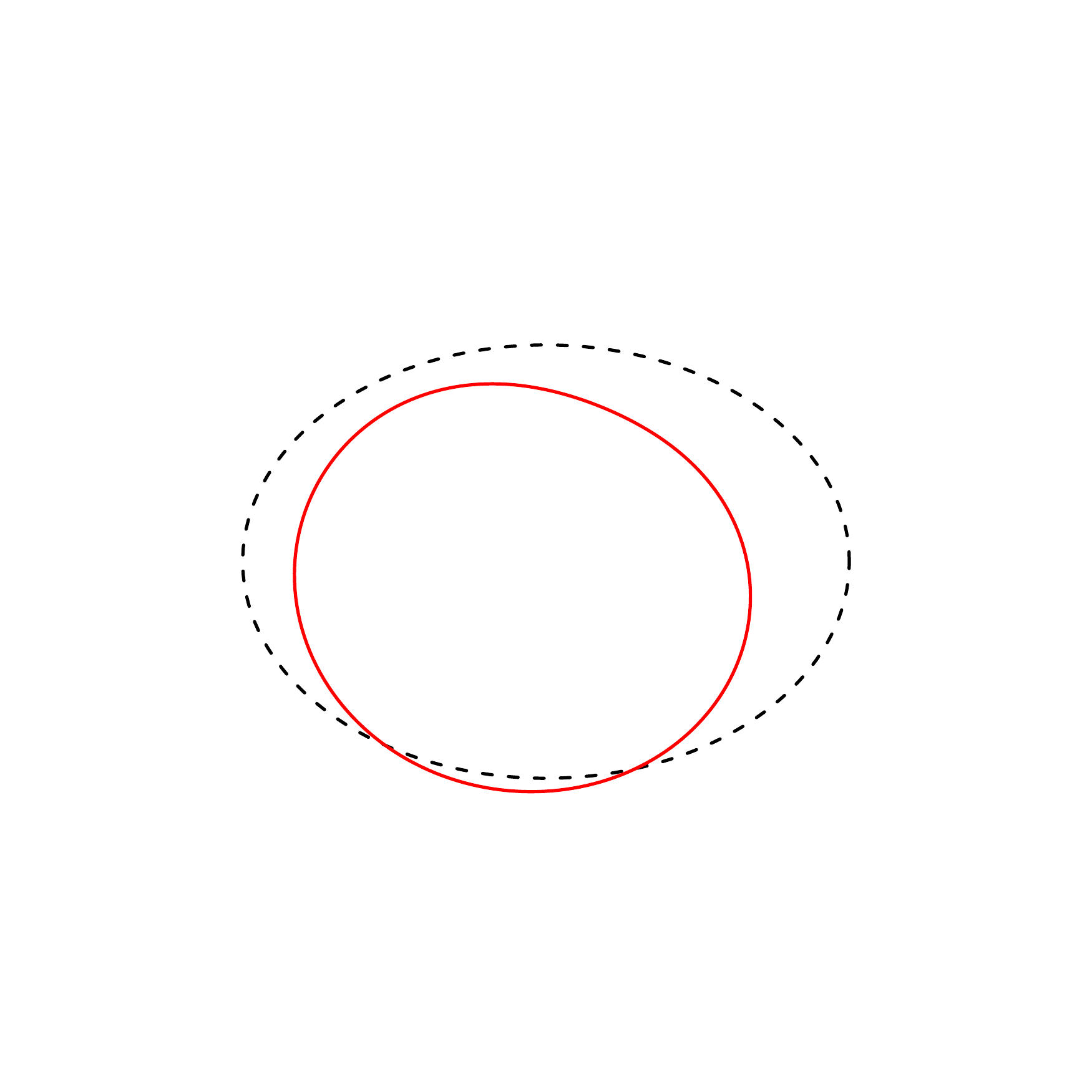}}
\subfloat[MCE (31 basis)]{\includegraphics[width = 0.25\textwidth]{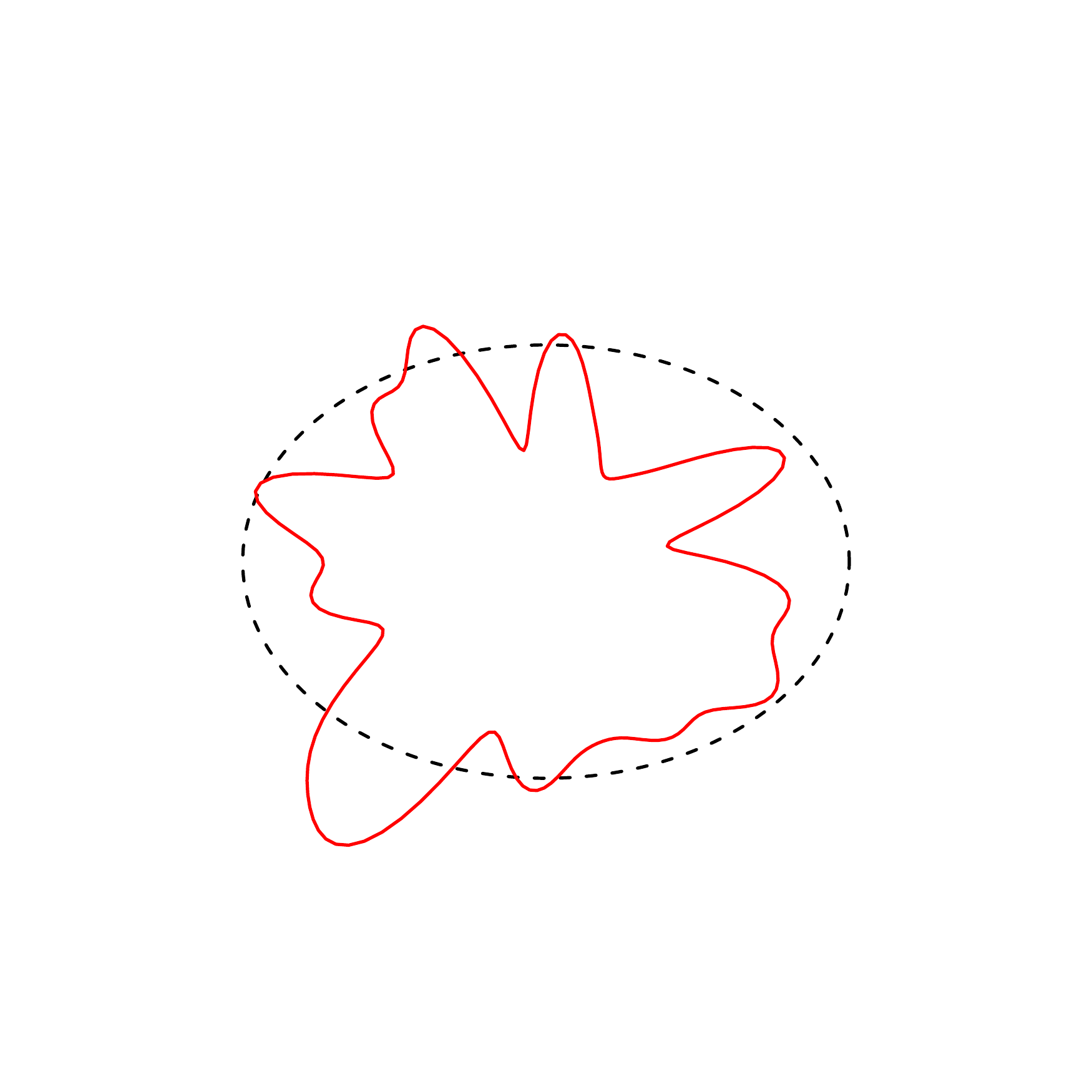}} \\
\subfloat[Case B2: $m = 100$, $\pi_1 = 0.50$]{\includegraphics[width = 0.25\textwidth]{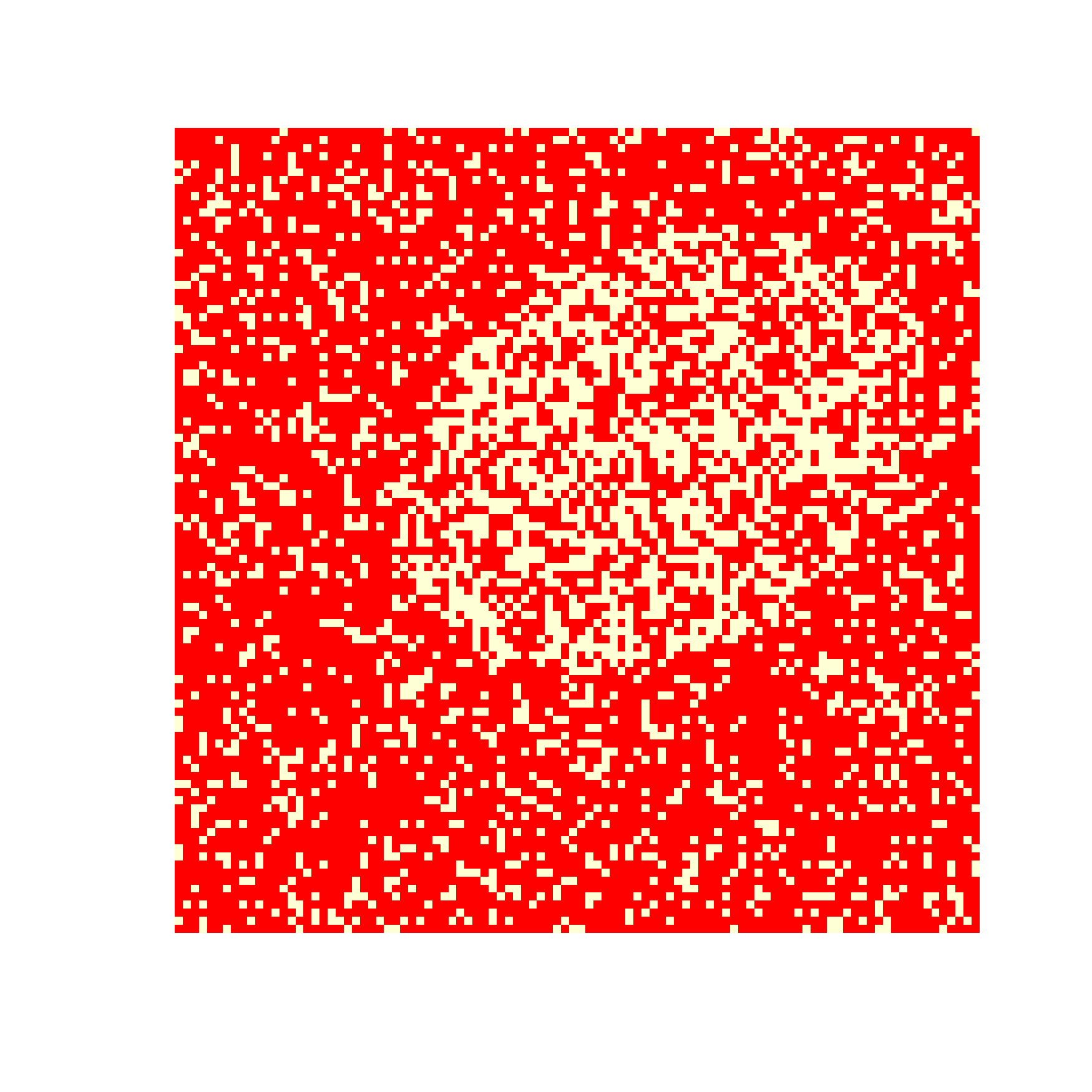}}
\subfloat[Bayesian Est.]{\includegraphics[width = 0.25\textwidth]{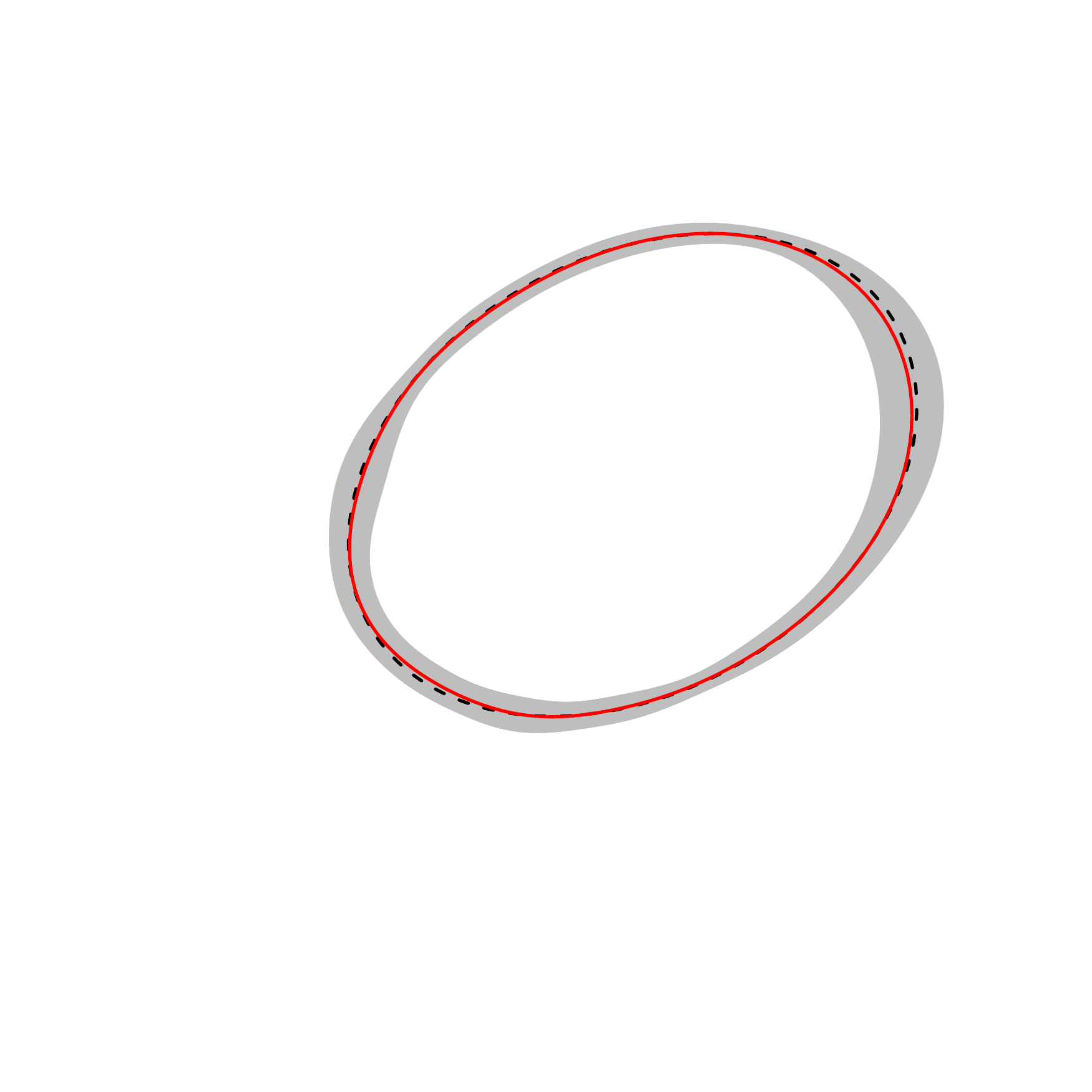}}\subfloat[MCE (5 basis)]{\includegraphics[width = 0.25\textwidth]{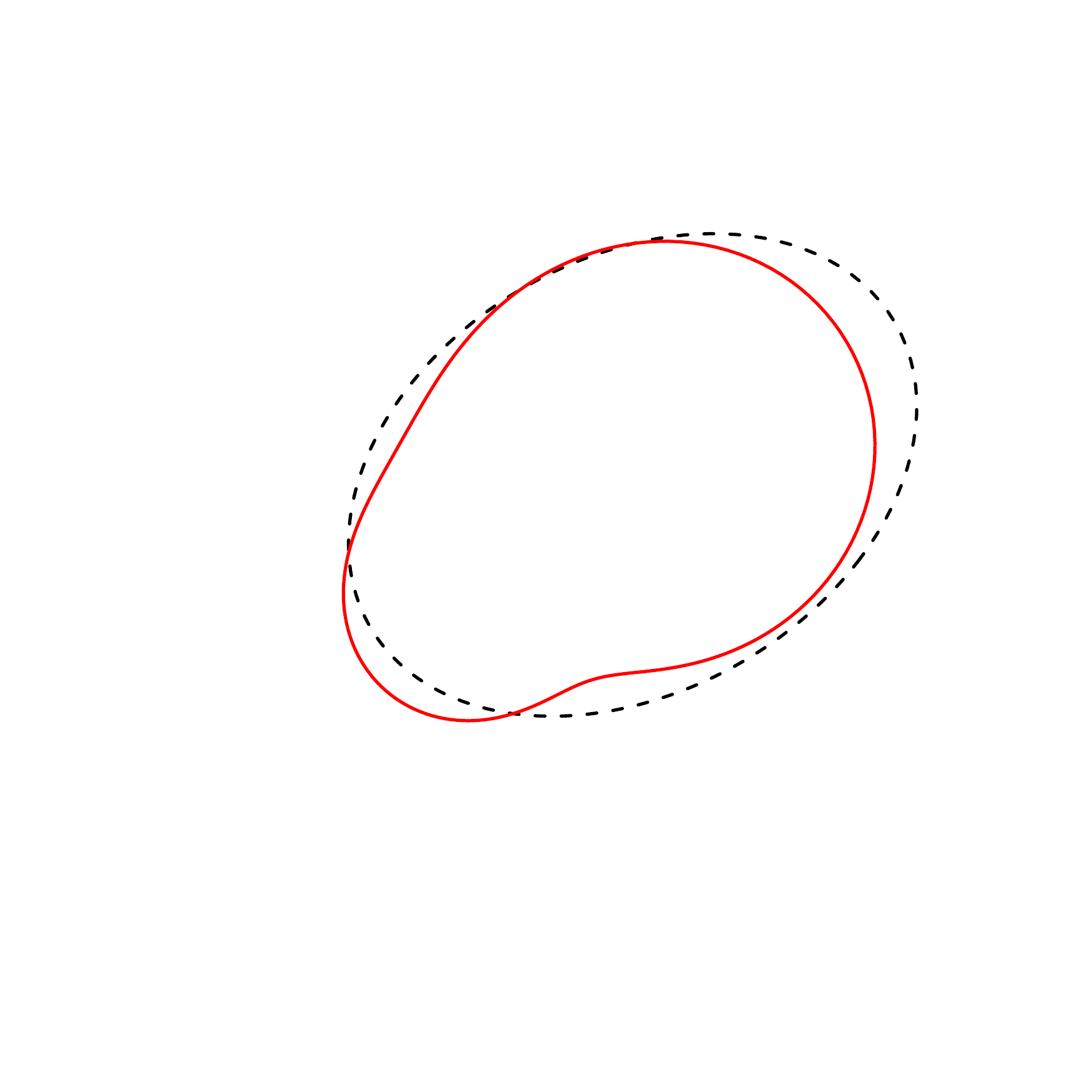}}
\subfloat[MCE (31 basis)]{\includegraphics[width = 0.25\textwidth]{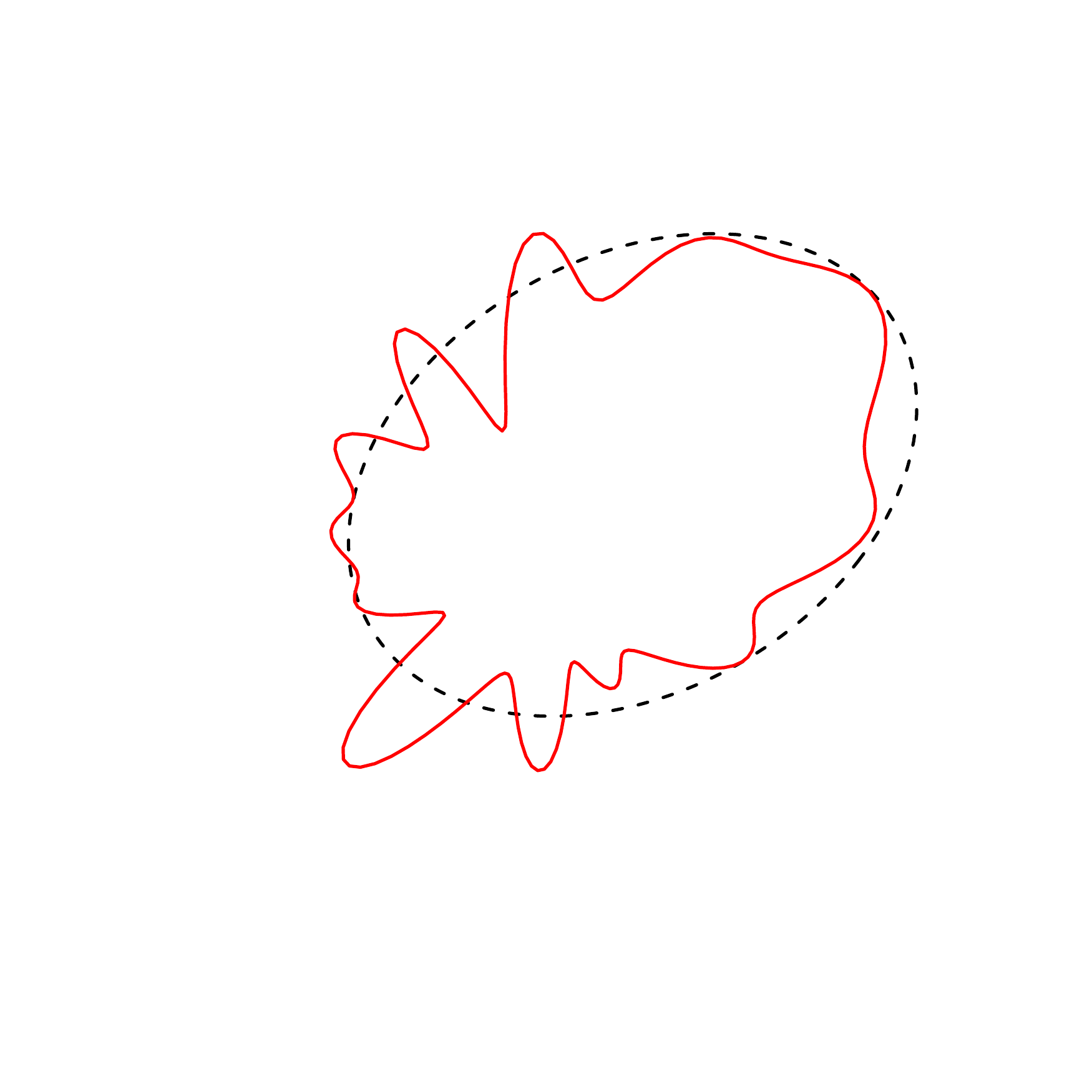}} \\
\subfloat[Case B2: $m = 500$, $\pi_1 = 0.25$]{\includegraphics[width = 0.25\textwidth]{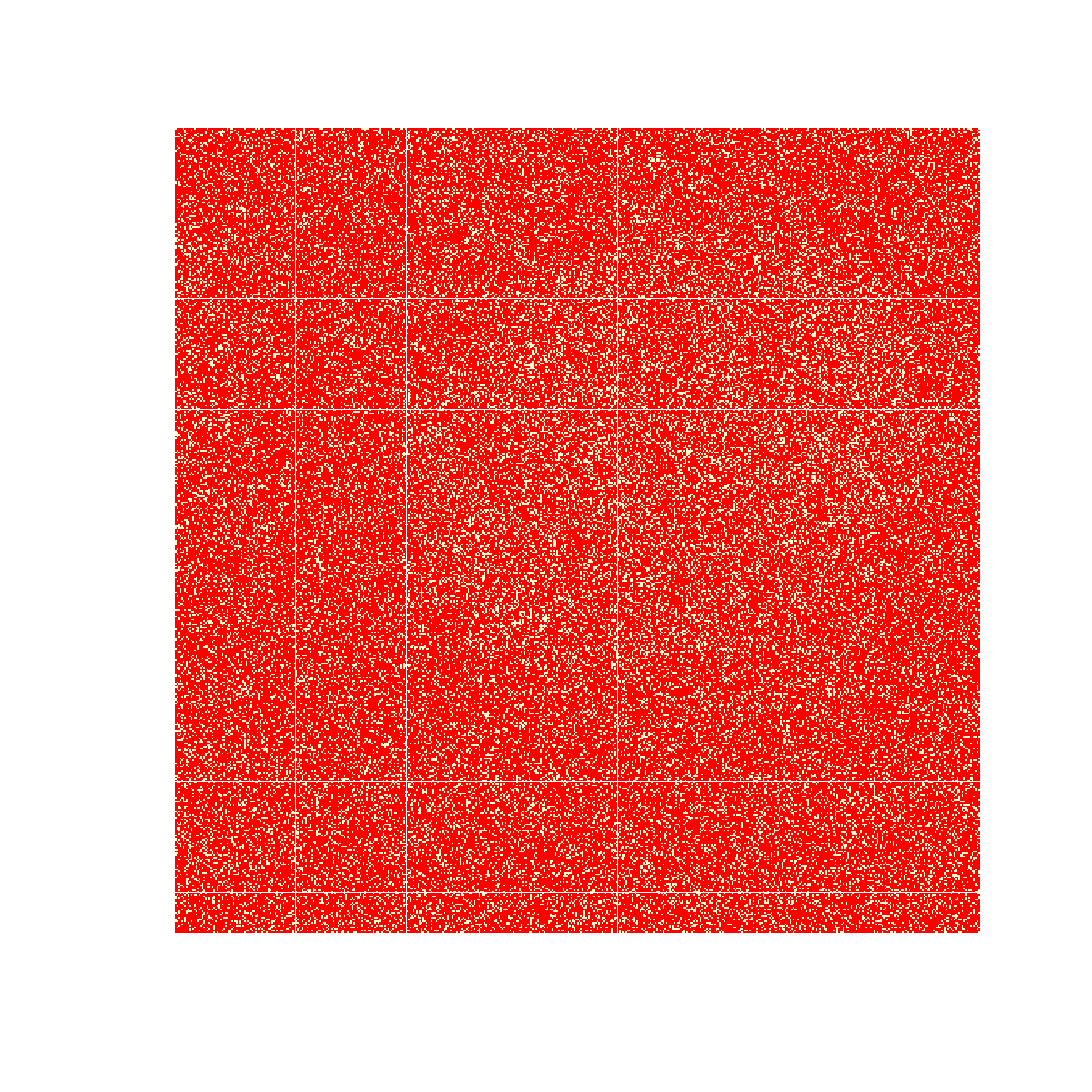}}
\subfloat[Bayesian Est.]{\includegraphics[width = 0.25\textwidth]{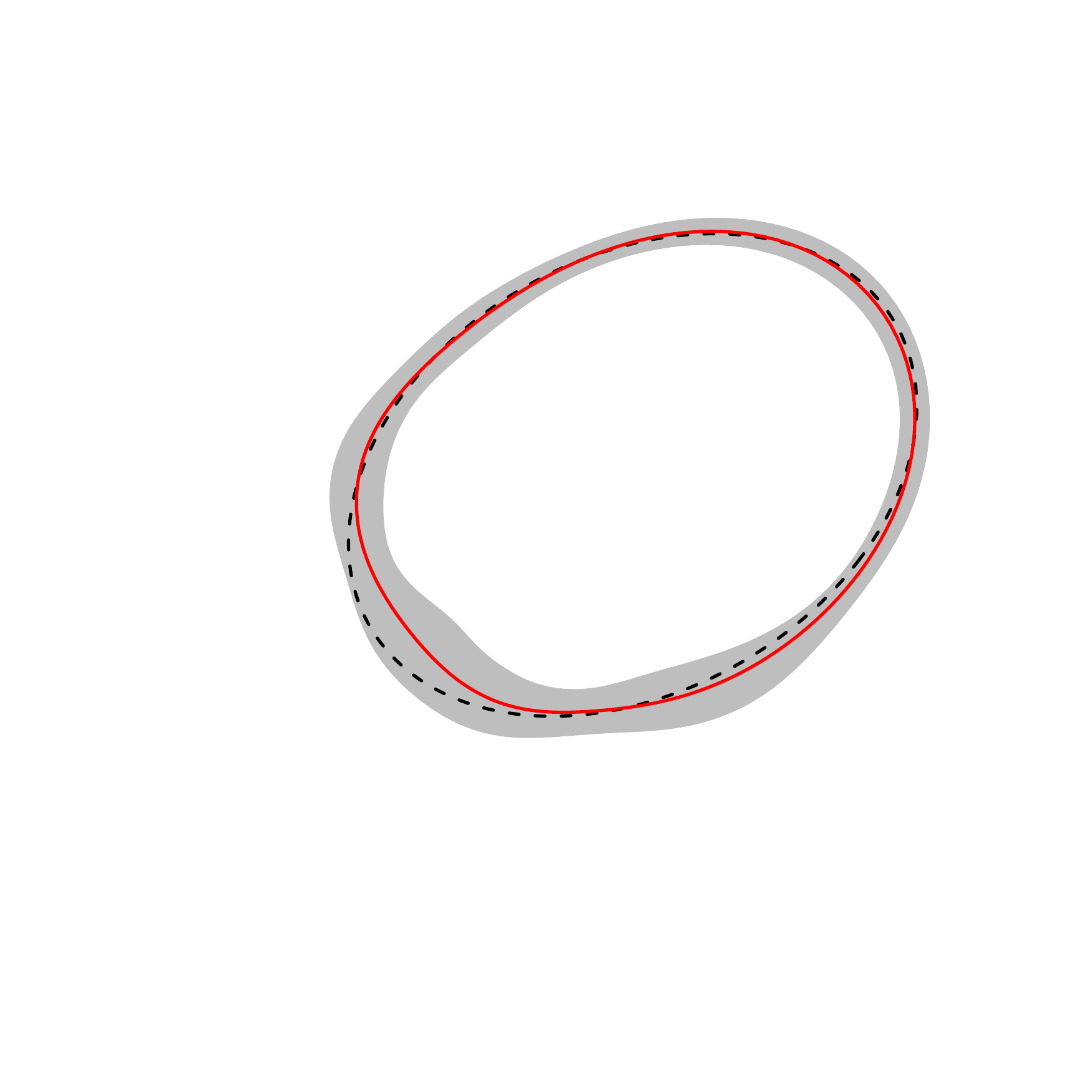}}\subfloat[MCE (5 basis)]{\includegraphics[width = 0.25\textwidth]{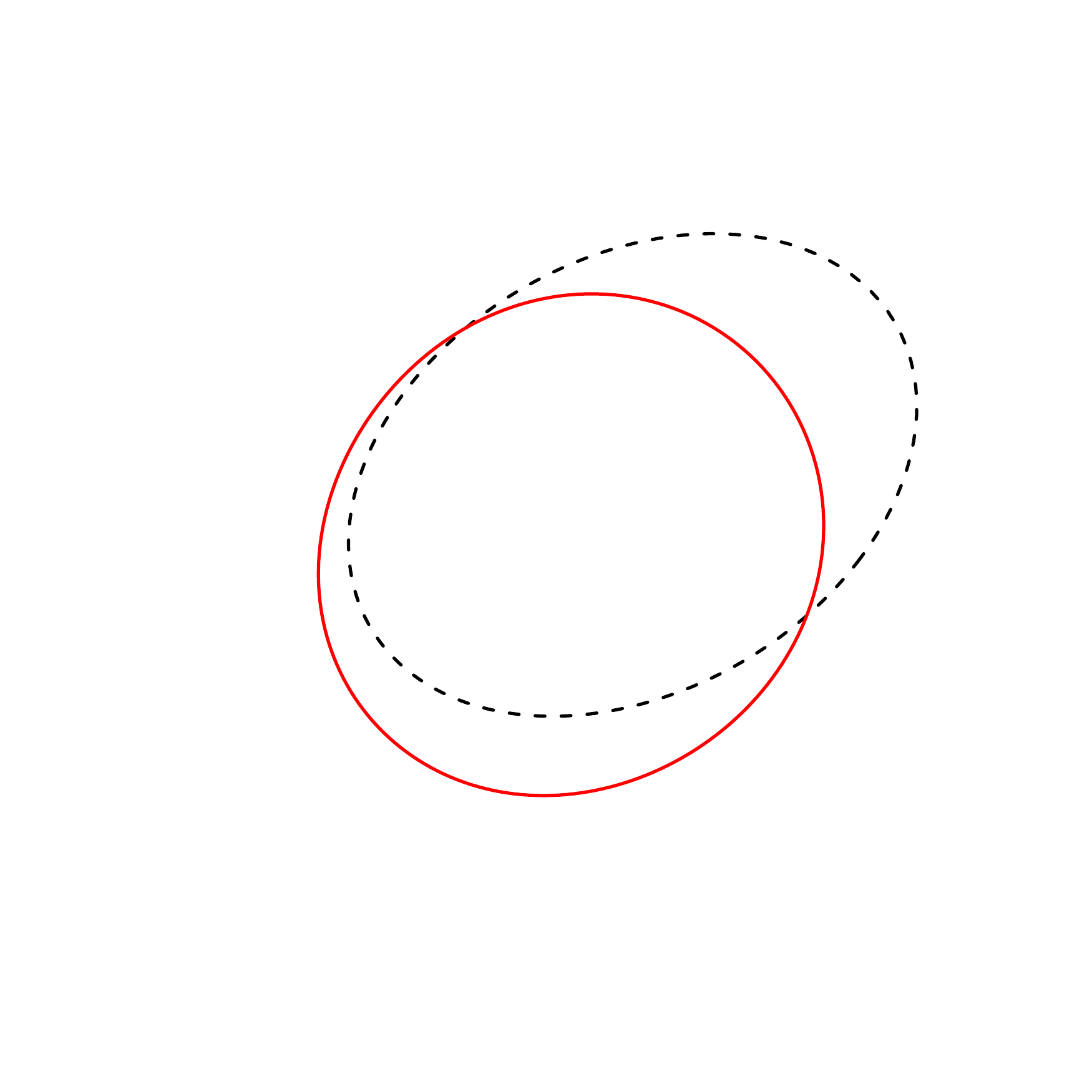}}
\subfloat[MCE (31 basis)]{\includegraphics[width = 0.25\textwidth]{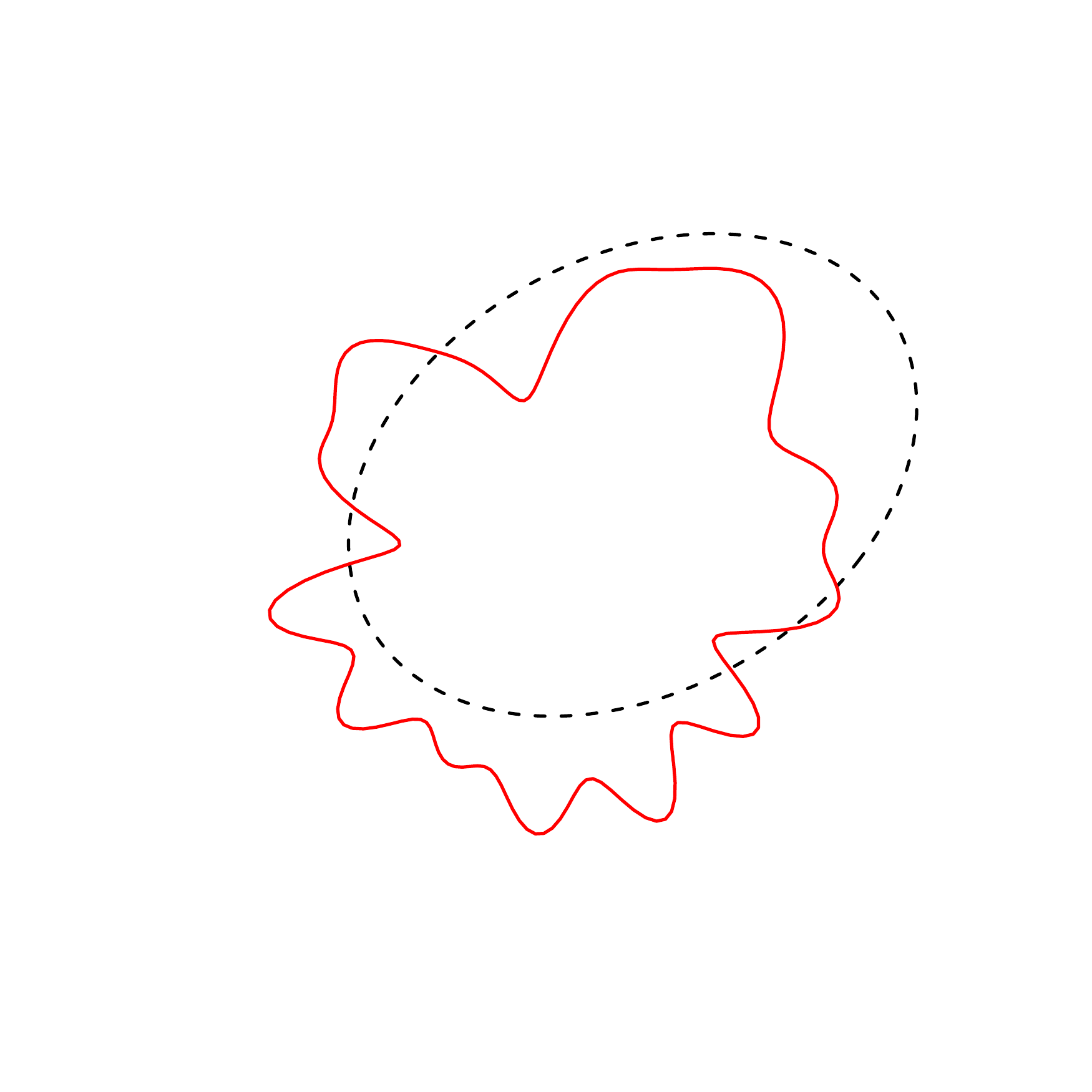}} \\
\caption{Performance on binary images (Column 1) with elliptic boundary. Column 2--4 plot the estimate (solid line in red) against the true boundary (dotted line in black). A 95\% uniform credible band (in gray) is provided for the Bayesian estimate (Column 2). }
\label{fig:ellipse}
\end{figure}

\begin{figure}
\subfloat[Case B3: $m = 100$, $\pi_1 = 0.50$]{\includegraphics[width = 0.25\textwidth]{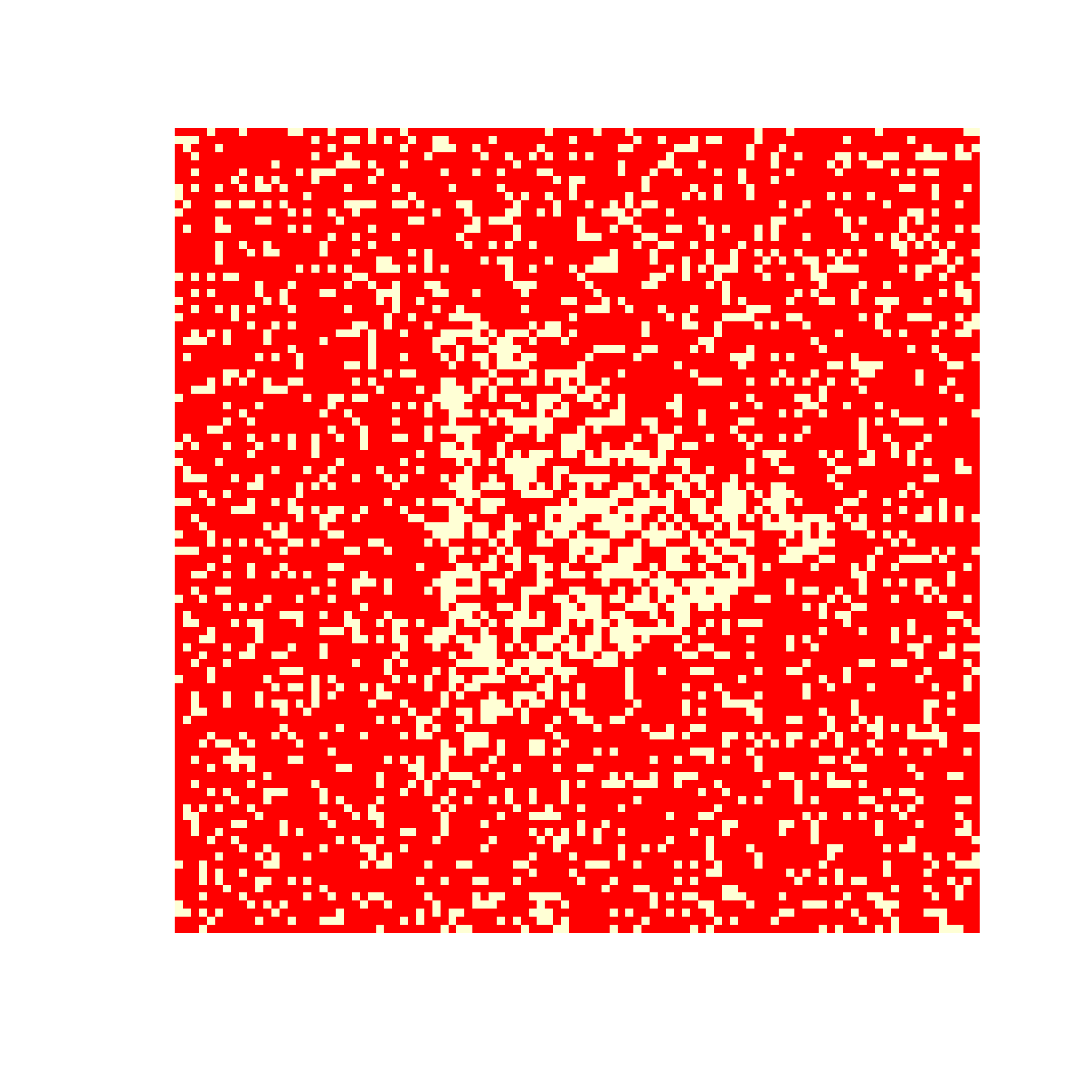}}
\subfloat[Bayesian Est.]{\includegraphics[width = 0.25\textwidth]{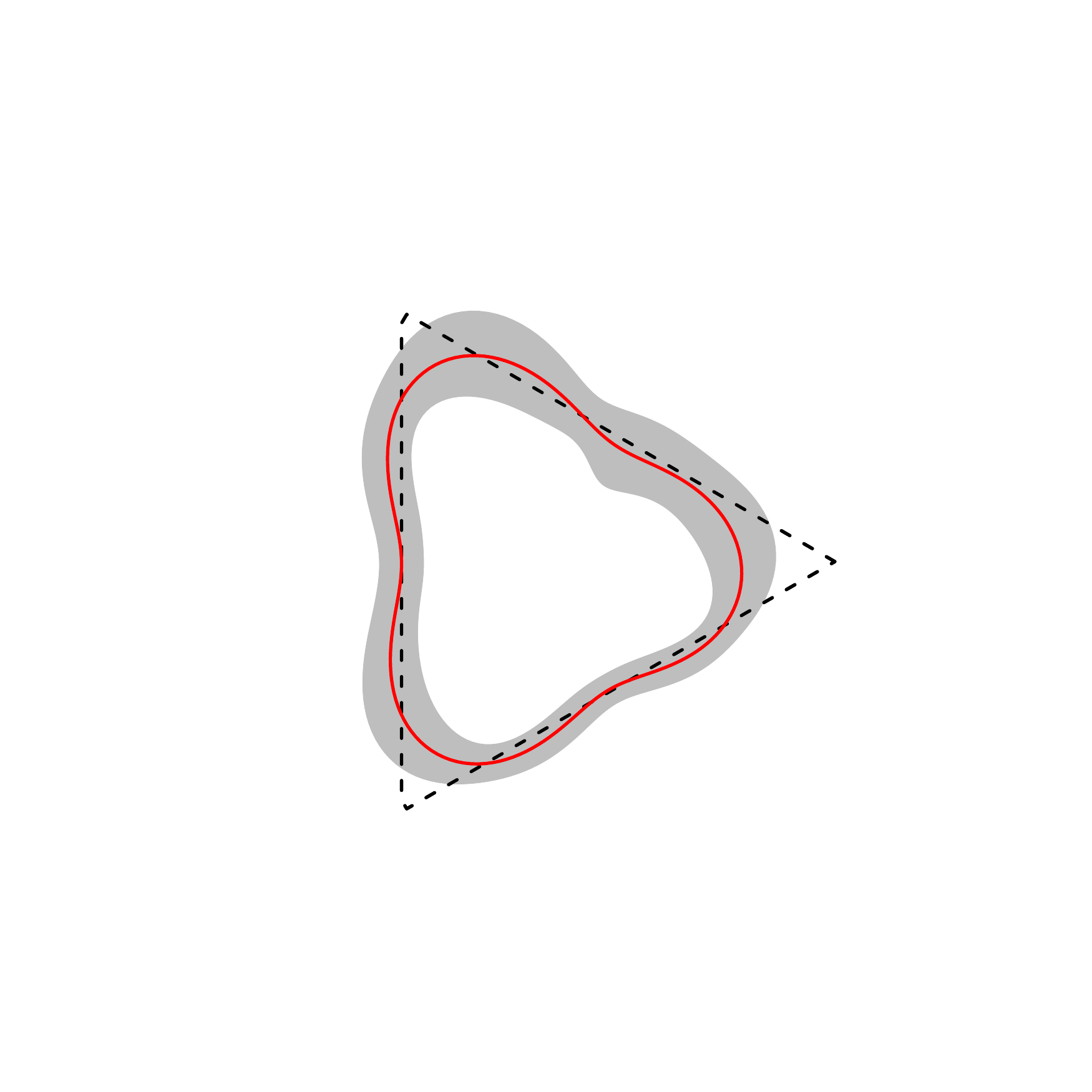}}\subfloat[MCE (5 basis)]{\includegraphics[width = 0.25\textwidth]{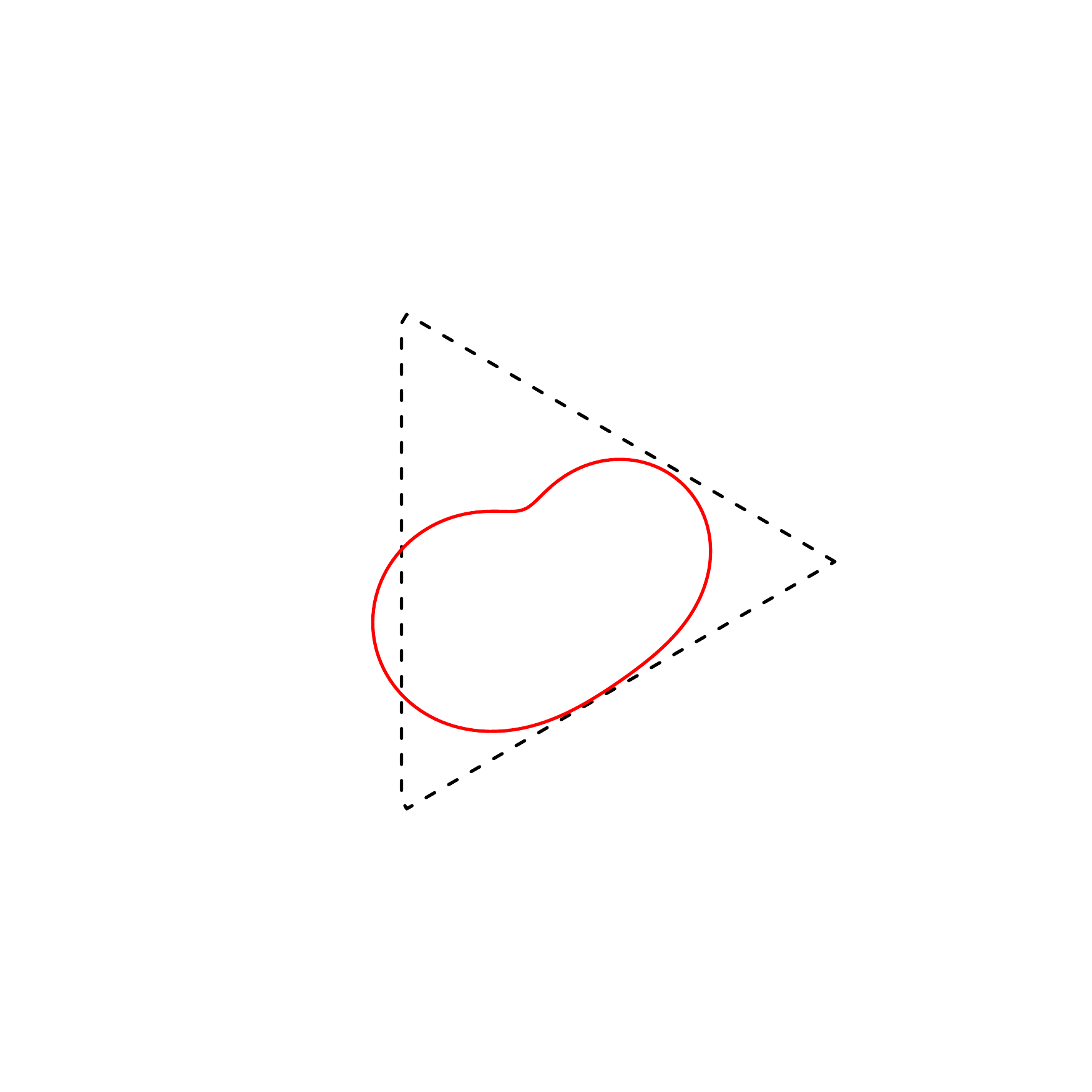}}
\subfloat[MCE (31 basis)]{\includegraphics[width = 0.25\textwidth]{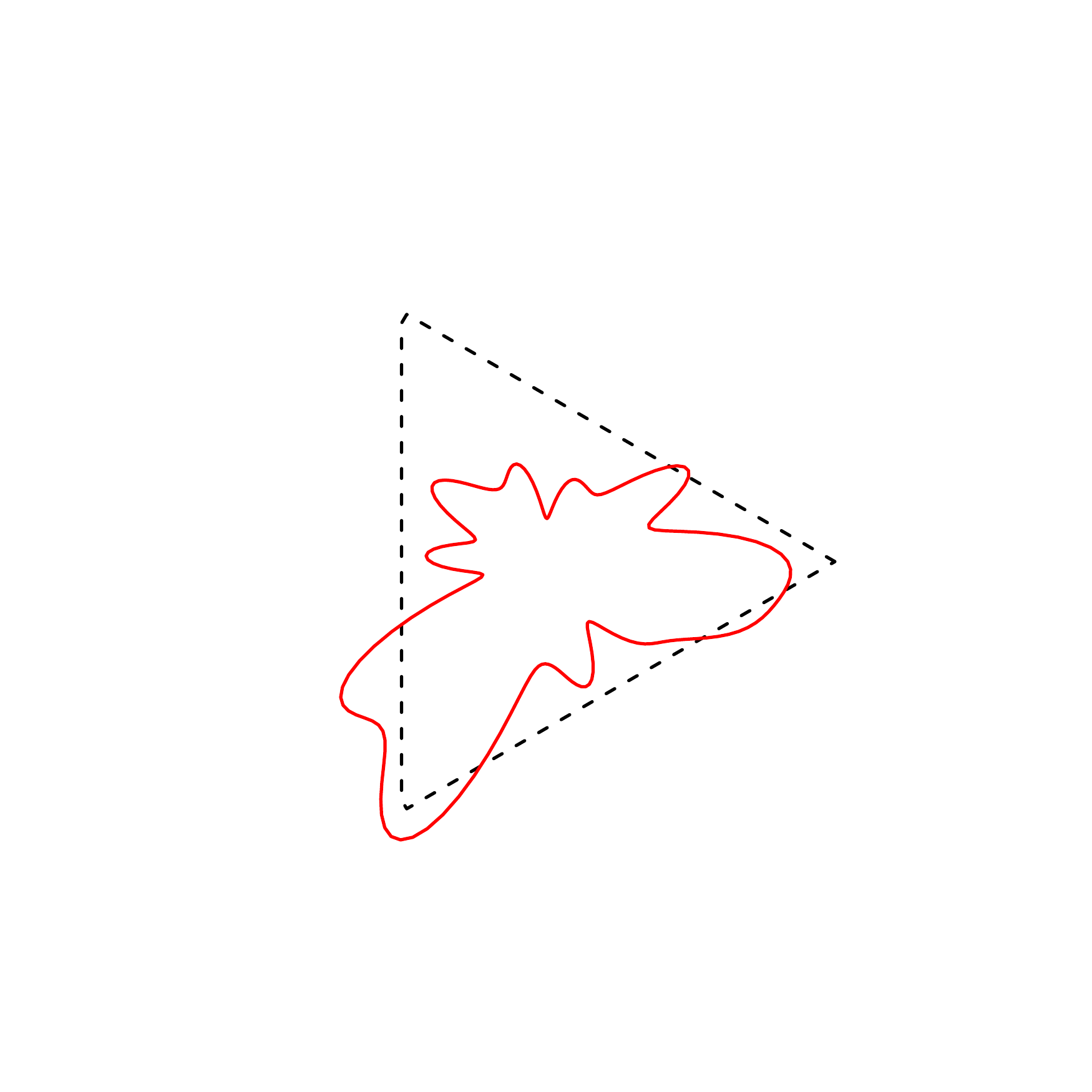}} \\
\subfloat[Case B3: $m = 500$, $\pi_1 = 0.25$]{\includegraphics[width = 0.25\textwidth]{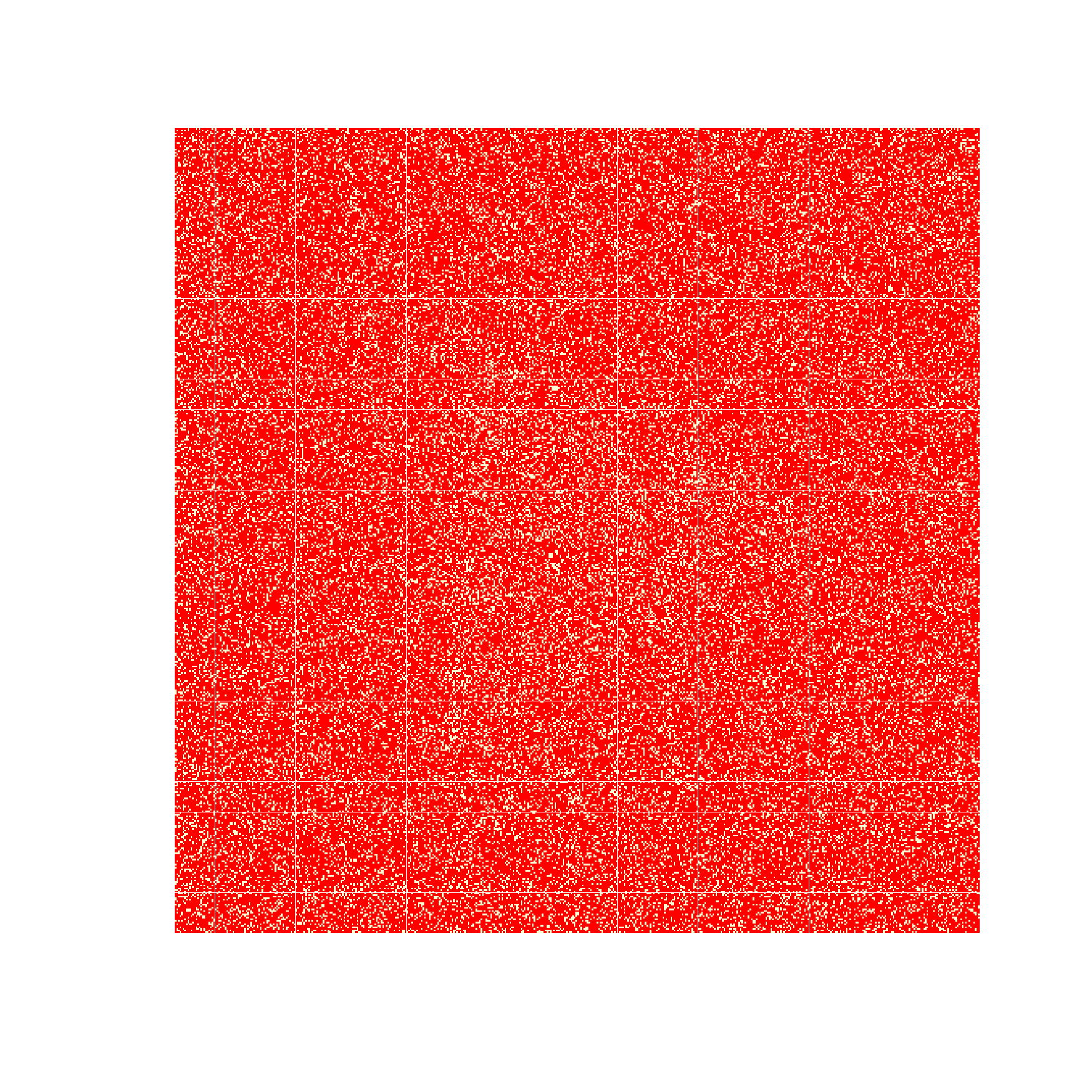}}
\subfloat[Bayesian Est.]{\includegraphics[width = 0.25\textwidth]{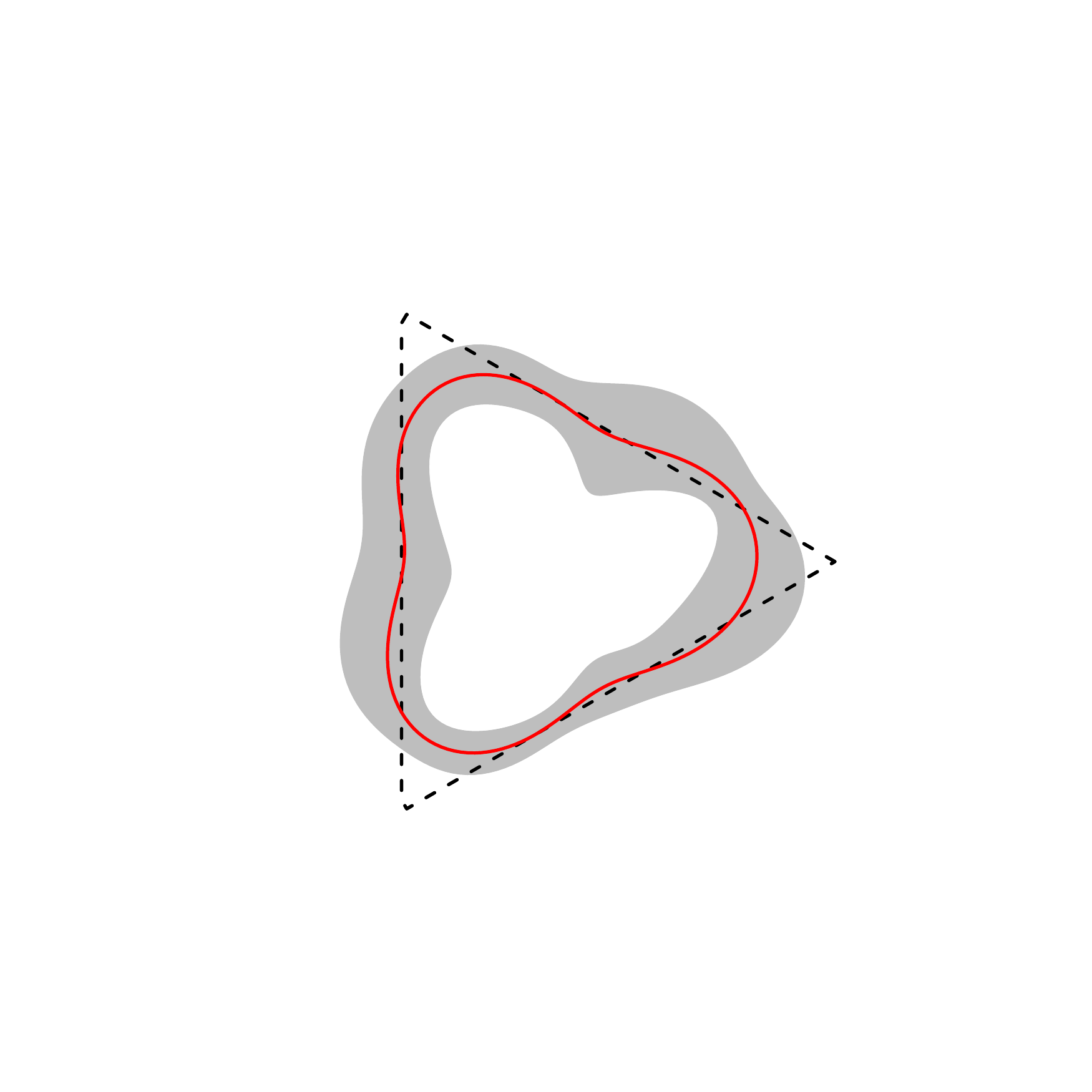}}\subfloat[MCE (5 basis)]{\includegraphics[width = 0.25\textwidth]{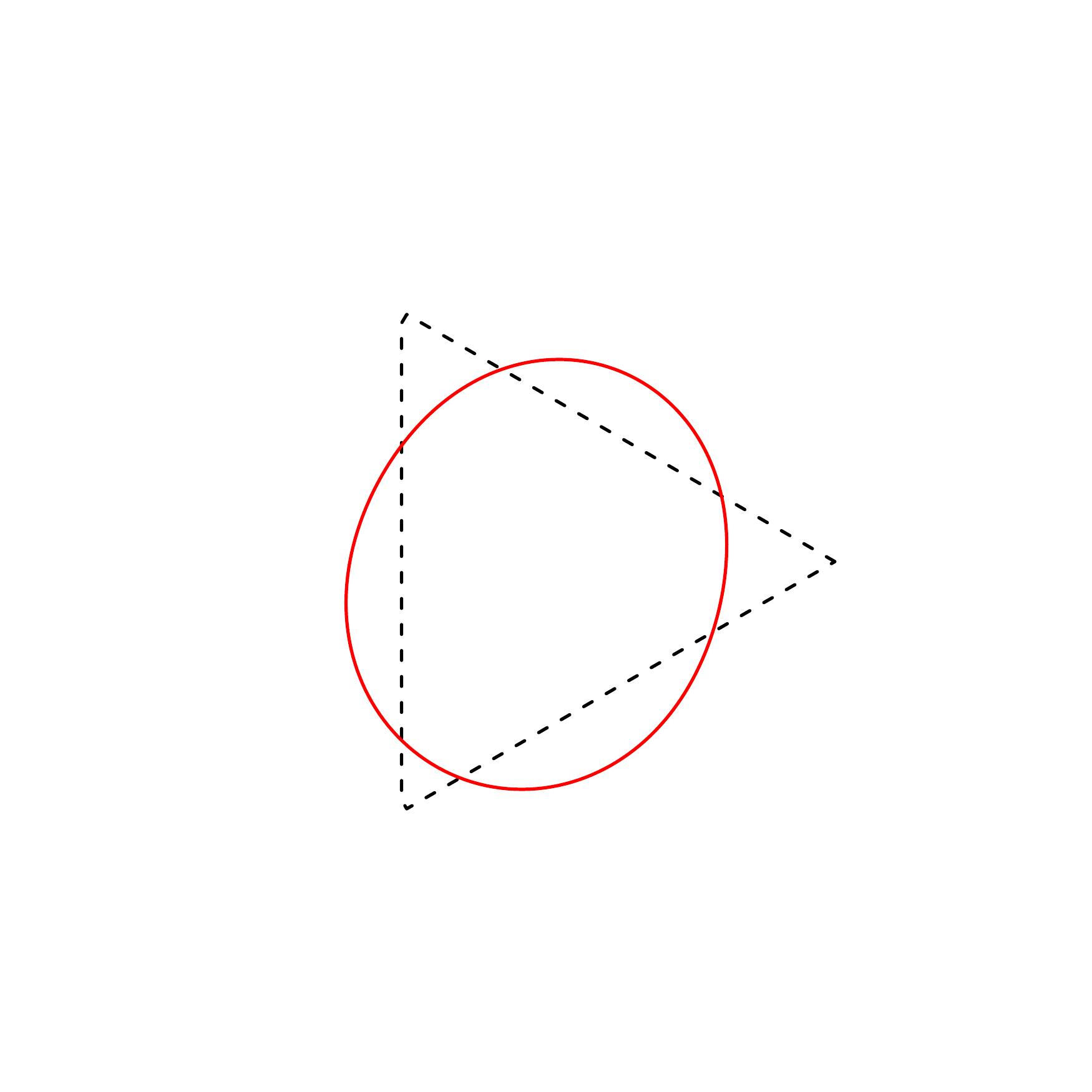}}
\subfloat[MCE (31 basis)]{\includegraphics[width = 0.25\textwidth]{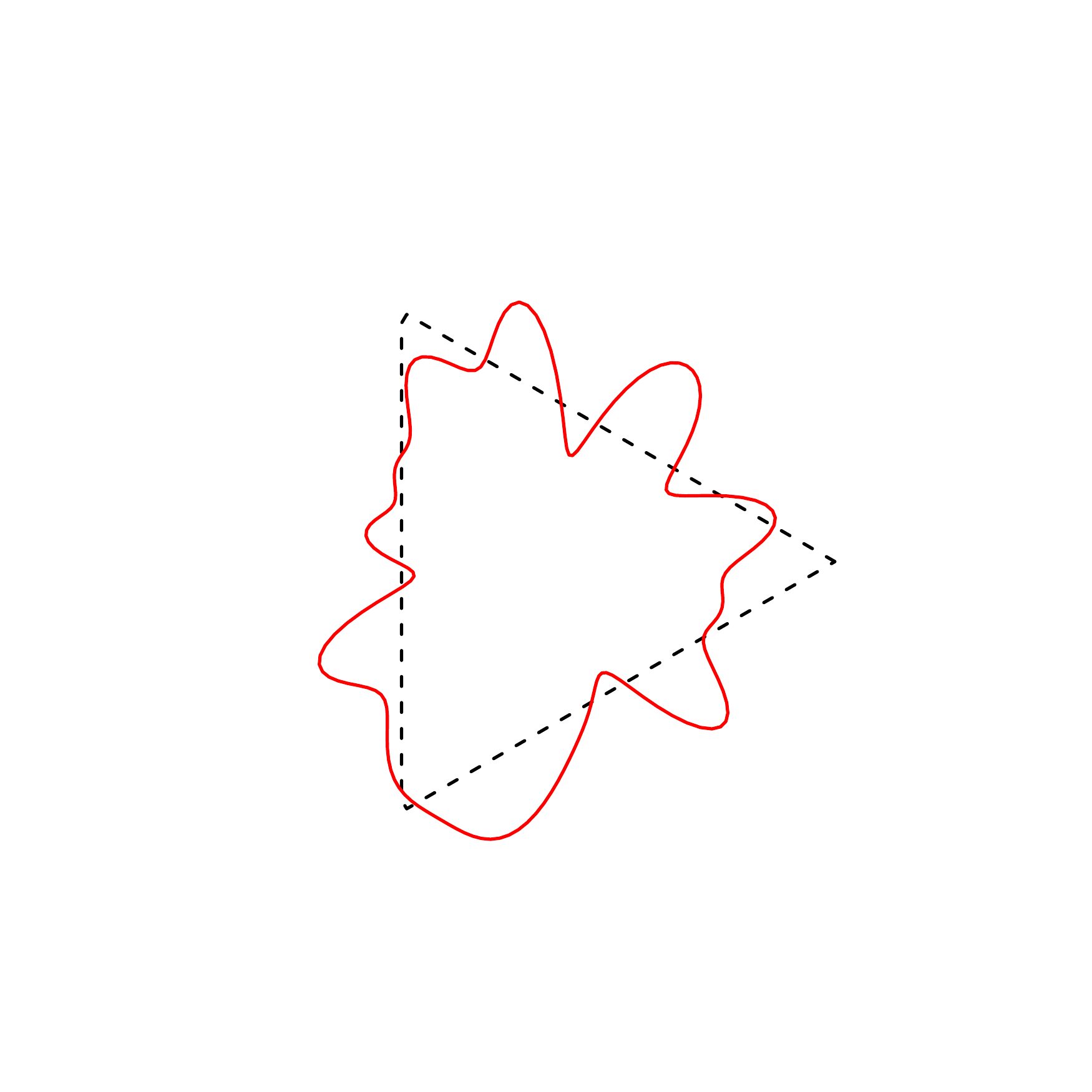}} \\
\subfloat[Case B3: $m = 500$, $\pi_1 = 0.50$]{\includegraphics[width = 0.25\textwidth]{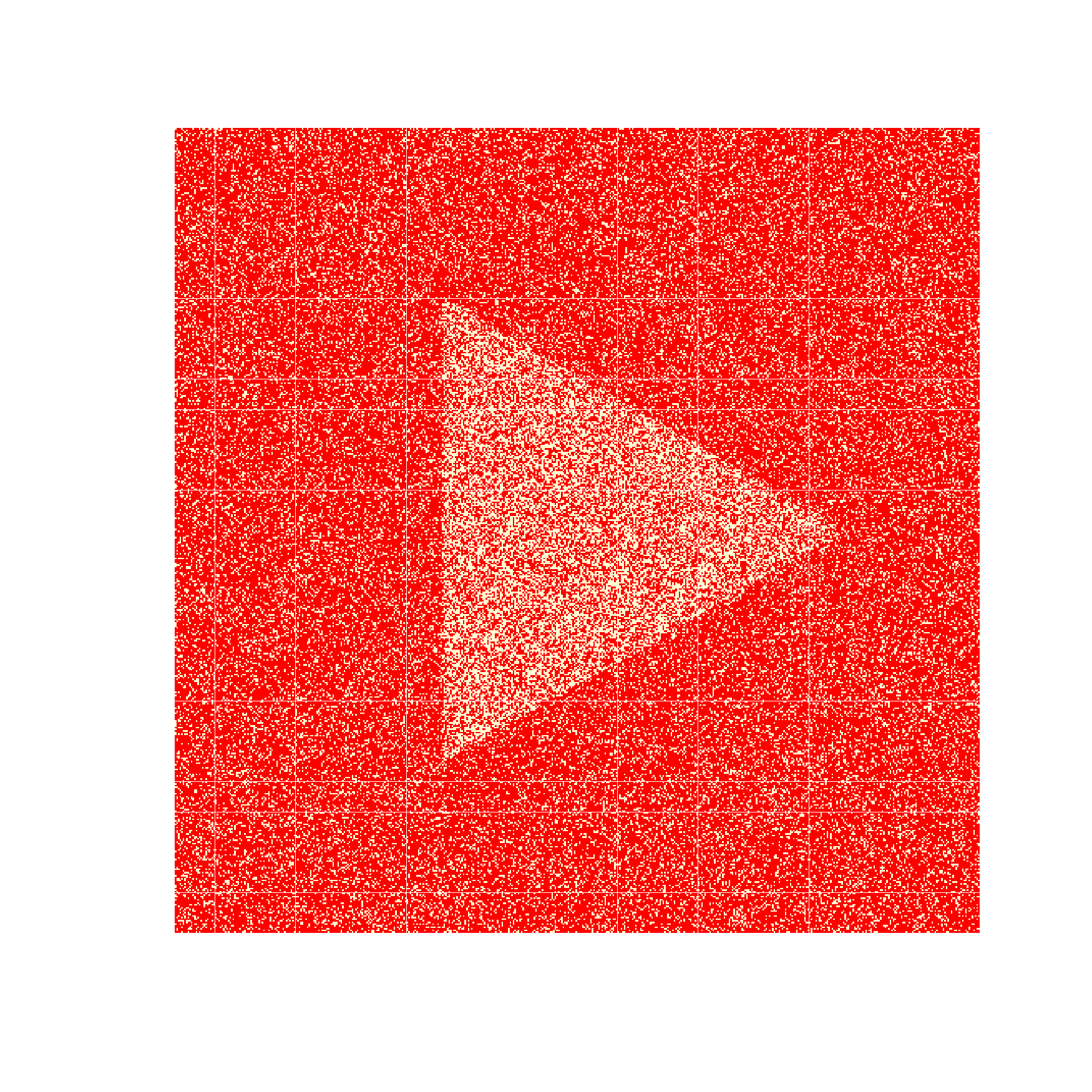}}
\subfloat[Bayesian Est.]{\includegraphics[width = 0.25\textwidth]{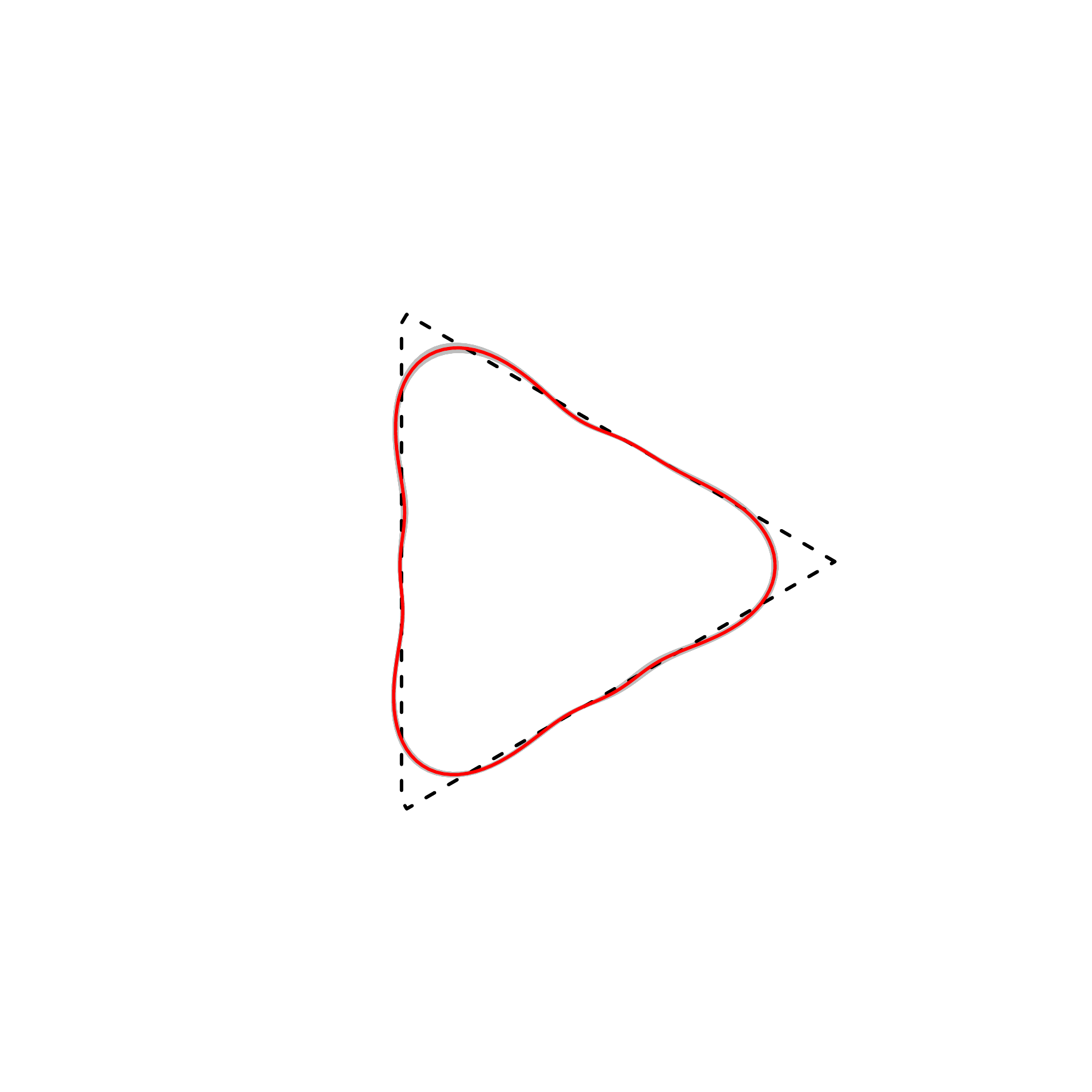}}\subfloat[MCE (5 basis)]{\includegraphics[width = 0.25\textwidth]{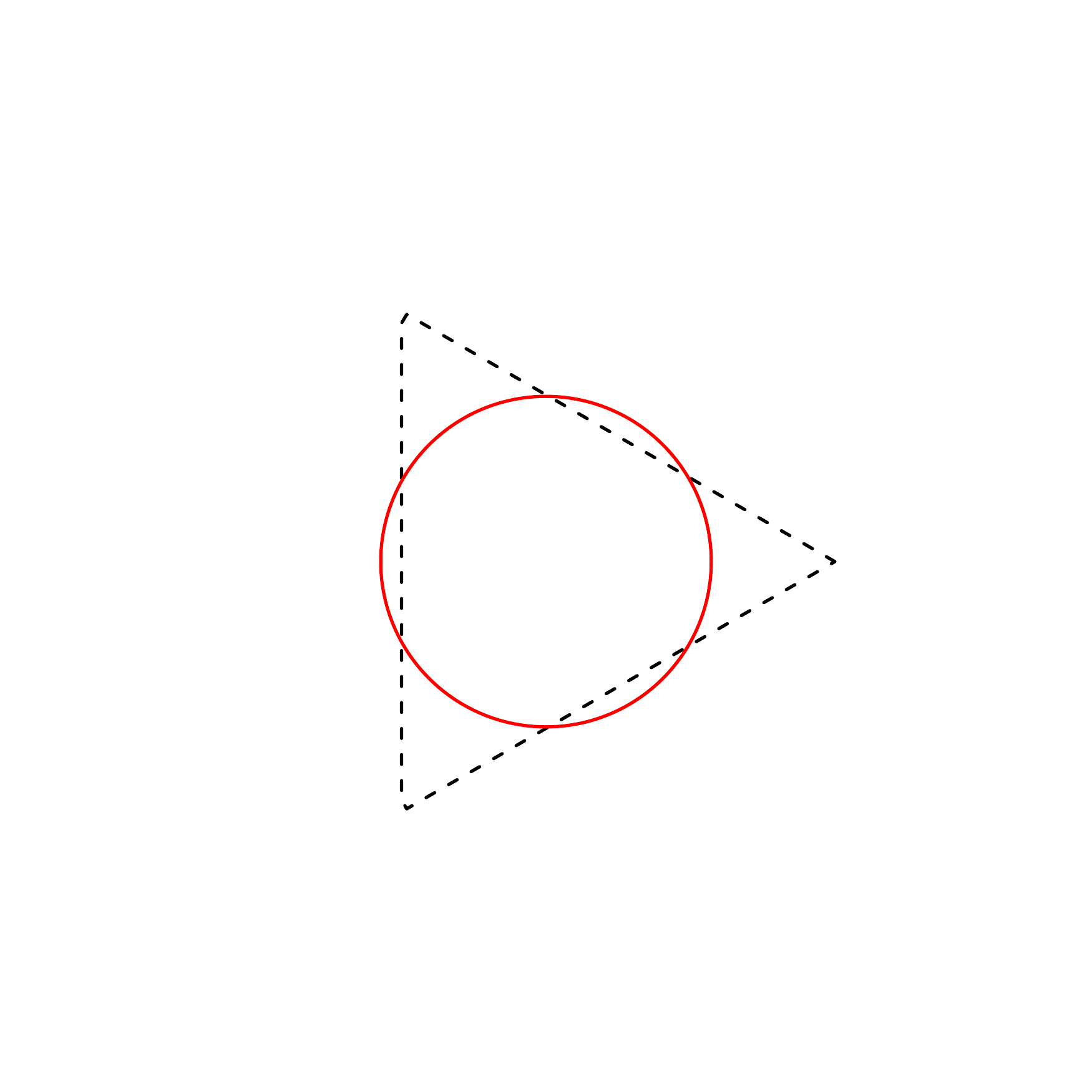}}
\subfloat[MCE (31 basis)]{\includegraphics[width = 0.25\textwidth]{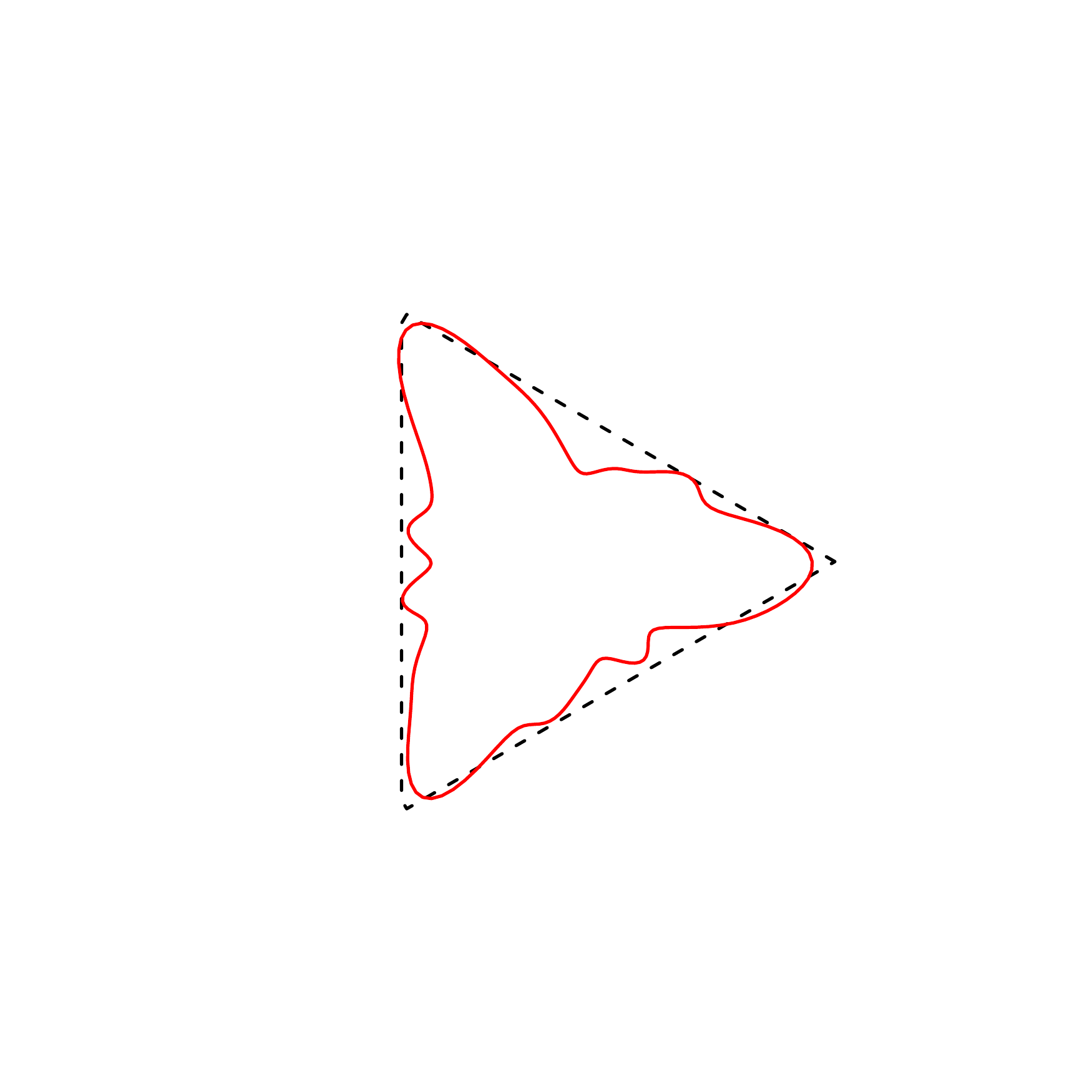}} \\
\caption{Performance on binary images (Column 1) when the boundary curve is an regular triangle. Column 2--4 plot the estimate (solid line in red) against the true boundary (dotted line in black). A 95\% uniform credible band (in gray) is provided for the Bayesian estimate (Column 2). }
\label{fig:triangle}
\end{figure}

\begin{figure}
	\centering 
	\begin{tabular}{cc} 
		%\vspace{-0.1in}
		\includegraphics[trim = 0.25in 0.5in 0.25in 0.5in, clip = TRUE, width=0.5\linewidth]{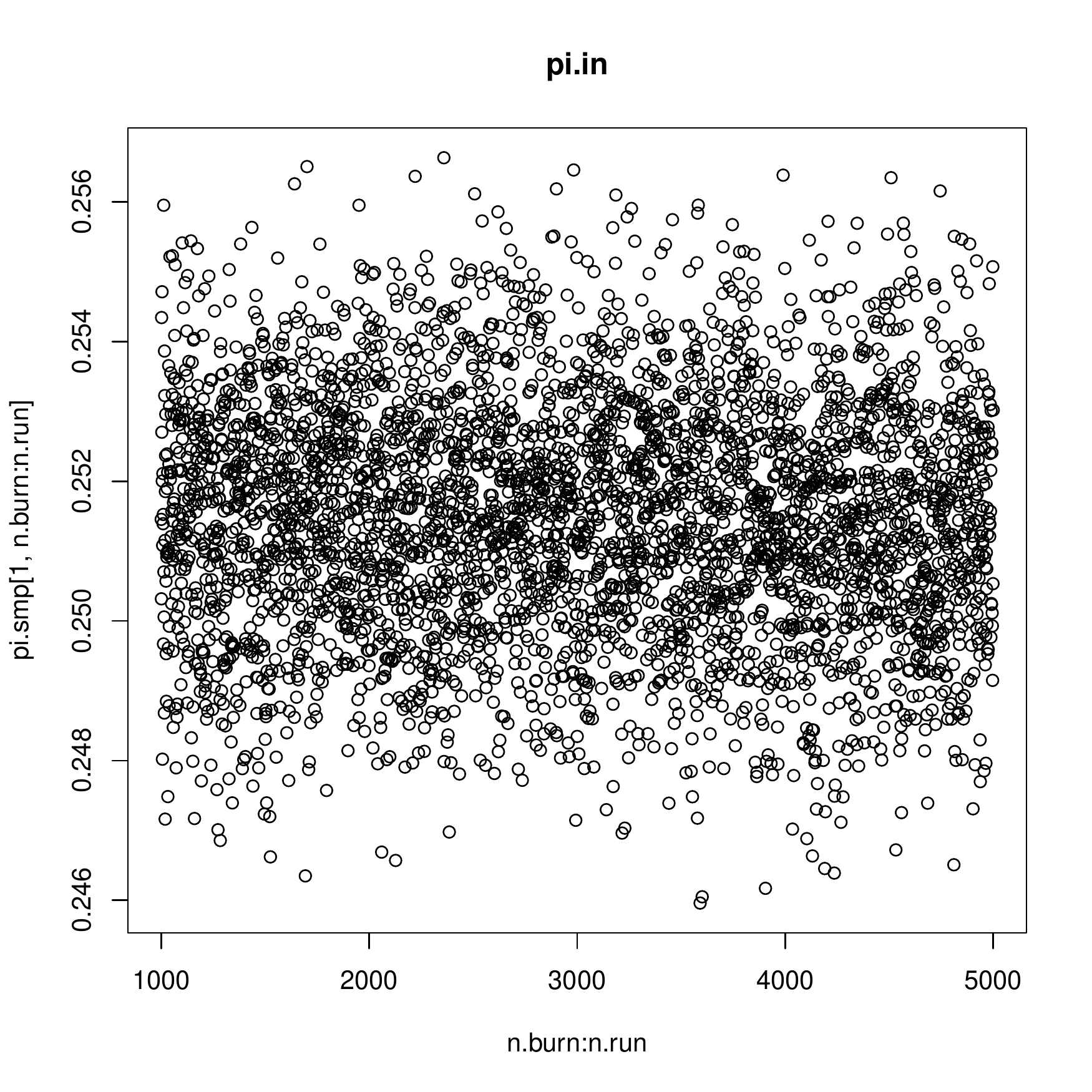} & 
		\includegraphics[trim = 0.25in 0.5in 0.25in 0.5in, clip = TRUE, width=0.5\linewidth]{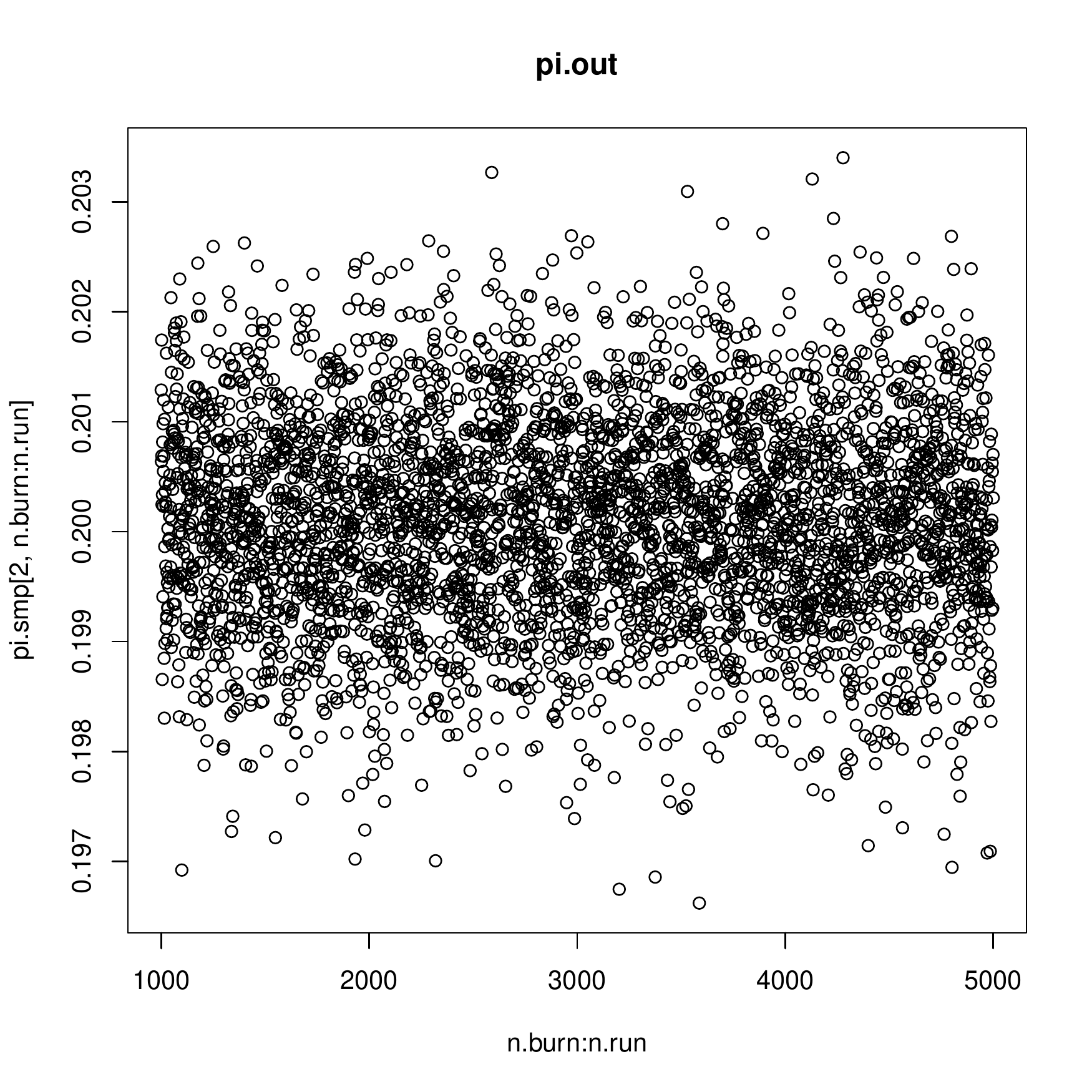} \\
		(a) Trace plot of $\pi_1$ & (b) Trace plot of $\pi_2$ \\
		%\vspace{-0.1in}
		\includegraphics[trim = 0.25in 0.5in 0.25in 0.5in, clip = TRUE, width=0.5\linewidth]{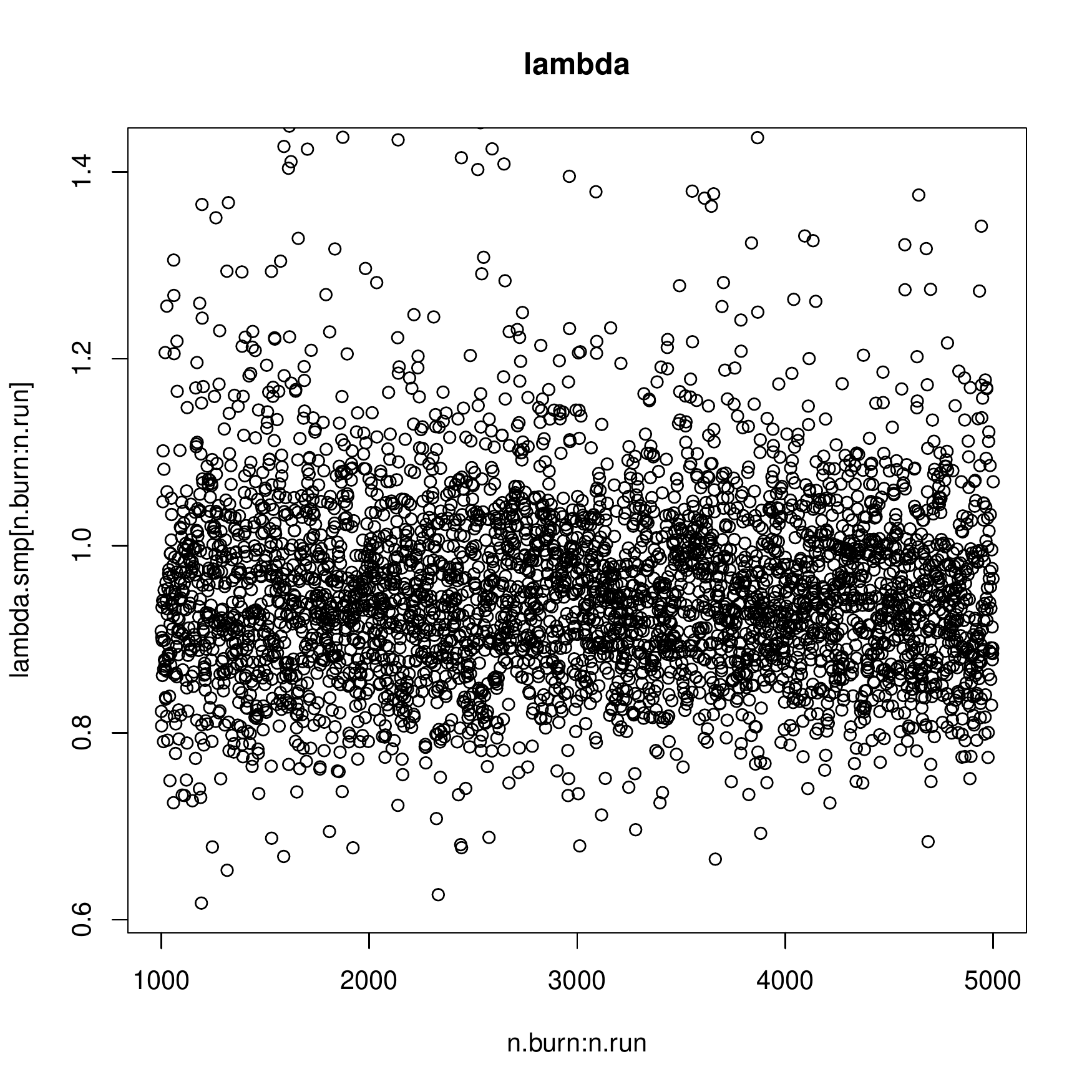} & 
		\includegraphics[trim = 0in 0.5in 0.25in 0.52in, clip = TRUE, width=0.5\linewidth]{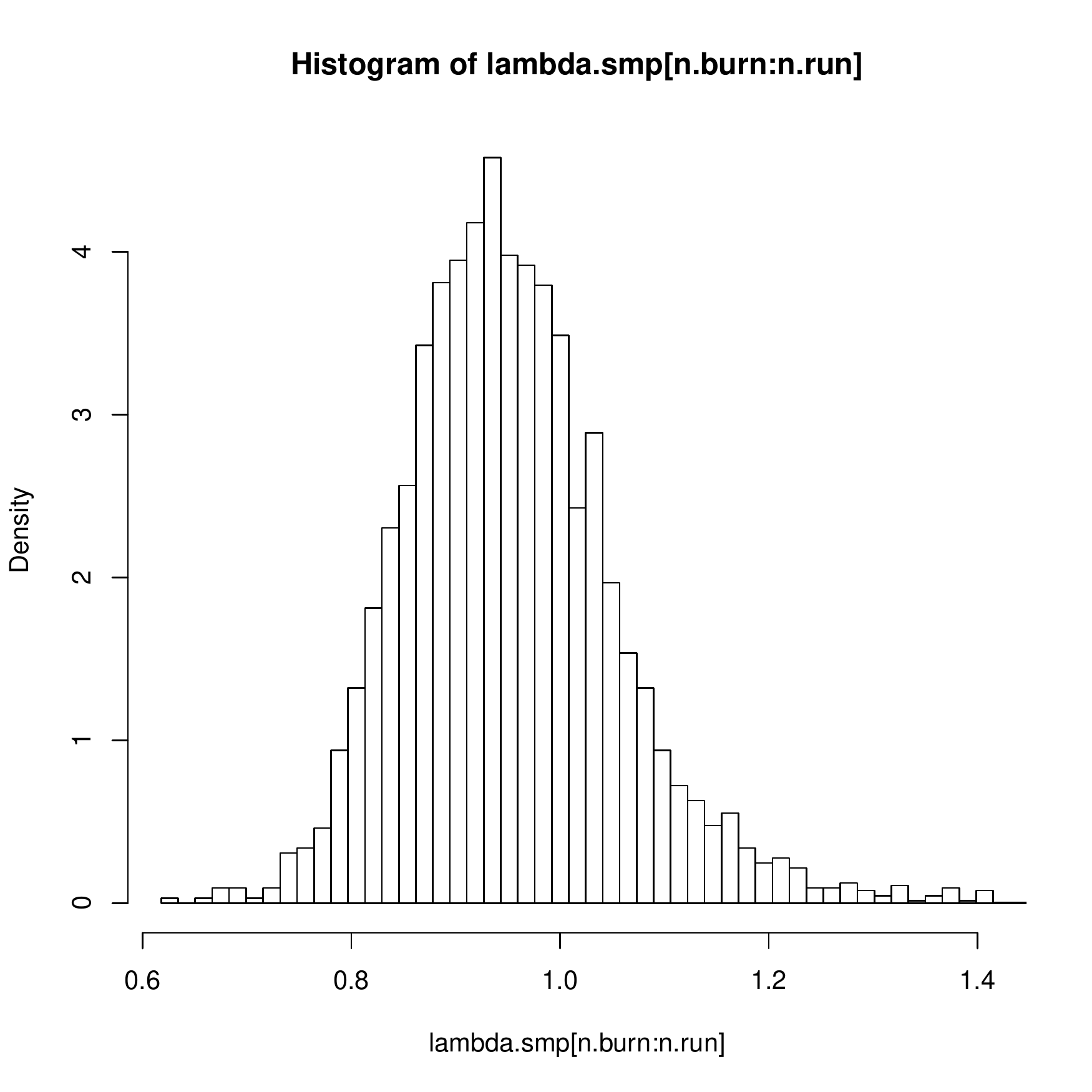} \\
		(c) Trace plot of $a$ & (d) Histogram of $a$ 
	\end{tabular}
	\caption{Trace plots and histograms of posterior samples of $(\pi_1, \pi_2, a)$ for Case B1 when $m = 500$ and $\pi_1 = 0.25$. }
	\label{fig:trace} 
\end{figure} 

\subsection{Numerical results for Gaussian noised images}
For Gaussian noised images, we keep using an ellipse with shift and rotation as the true boundary curve (i.e. Case B2). We consider the following four scenarios where the two standard deviations are all given by $(\sigma_1, \sigma_2) = (1.5, 1)$ and the observed image is $100 \times 100$:
\begin{itemize}
\item Case G1. $\mu_1 = 4, \mu_2 = 1$, i.e. the two regions differ in both the first two moments;
% the mean and the standard deviation;
\item Case G2. $\mu_1 = \mu_2 = 1$, i.e. the two regions only differ in the standard deviation;
\item Case G3. $(\mu_1, \mu_2)$ are functions of the location. Let $r^I$ be the smallest radius inside the boundary, and $r^O$ the largest radius outside the boundary. We use $\mu(i)$ for the mean of $Y_i$ and let $\mu(i) = r_i - r^I + 0.2$ if it is inside, while $\mu(i) = r_i + r^O$ if outside. Therefore, the mean values vary at each location but have a gap of 0.2 between the two regions.
\item Case G4. We use mixture normal distribution $0.6 \mathrm{N}(2, \sigma_1^2) + 0.4 \mathrm{N}(1, \sigma_2^2)$ for the inside distribution; the outside distribution is still Gaussian with mean $\mu_2 = 1$.
\end{itemize}
Cases G3 and G4 allow us to investigate the performance of the proposed method when the distribution $f(\cdot)$ in the model is misspecified. For comparison, we use a 1-dimensional change-point detection algorithm~\citep{Chen+Gupta:11,Kil+:12} via the {\tt R} package {\tt changepoint}~\citep{Kil+Eck:11}. For the post-smoothing step, we use a penalized Fourier regression with 5 and 31 basis functions (method CP5 and CP31 in Table~\ref{table:Gaussian}). Here we use the estimates of CP5 as the mean in the Gaussian process prior. Table~\ref{table:Gaussian} shows that the proposed method has good performance for all the four cases. The method of CP5 and CP10 produce small errors in Case G1, but suffer a lot from the other three cases. It shows that the change-point method highly depends on the distinction between the means (Case G2), and also it loses its way when the model is misspecified. In fact, for Cases G2, G3 and G4, the CP5 and CP31 methods lead to a curve almost containing the whole frame of the image. The proposed Bayesian approach which models the boundary directly, seems to be not affected even when the model is substantially misspecified (Case G3). Figure~\ref{figure:Gaussian} shows the noisy observation and our estimation from 1 replication for all the four cases. We can see the impressive performance of the proposed method. It also shows that the contrast between the two regions are visible for Cases G3 and G4, and the proposed method is capable to capture the boundary even though the distributions are misspecified.

% latex table generated in R 3.0.3 by xtable 1.7-3 package
% Tue Feb 10 22:18:20 2015
\begin{table}[ht]
\centering
\caption{Performance of the methods for Gaussian noised images based on $100$ simulations. The Lebesgue error ($\times 10^{-2}$) between the estimated boundary the true boundary is presented. The maximum standard errors of each column are reported in the last row. }
\label{table:Gaussian}
\begin{tabular}{lcccc}
  \hline
& Case G1 & Case G2 & Case G3 & Case G4 \\ \hline
%Bayesian Method & 0.00 & 0.01 & 0.01 & 0.01 \\
%%   & (0.00) & (0.00) & (0.00) & (0.00) \\
%CP5 & 0.03 & 0.63 & 0.62 & 0.61 \\
%%  & (0.00) & (0.00) & (0.00) & (0.00) \\
%CP31 & 0.02 & 0.64 & 0.63 & 0.62 \\ \hline 
%SE   & 0.00 & 0.00 & 0.00 & 0.01 \\
Bayesian Method  & 0.11 & 0.99 & 0.69 & 0.99 \\ 
%2 & (0.00) & (0.03) & (0.02) & (0.05) \\ 
CP5 & 2.90 & 62.91 & 62.2 & 61.12 \\ 
%4 & (0.01) & (0.26) & (0.19) & (0.27) \\ 
CP31 & 1.99 & 64.00 & 63.26 & 62.10 \\ \hline 
% 6 & (0.01) & (0.23) & (0.19) & (0.27) \\ 
SE & 0.01 & 0.26 & 0.19 & 0.27 \\ \hline 
\end{tabular}
\end{table}

\begin{figure}[h!]
\centering
\subfloat[Case G1]{\includegraphics[width = 0.25\textwidth]{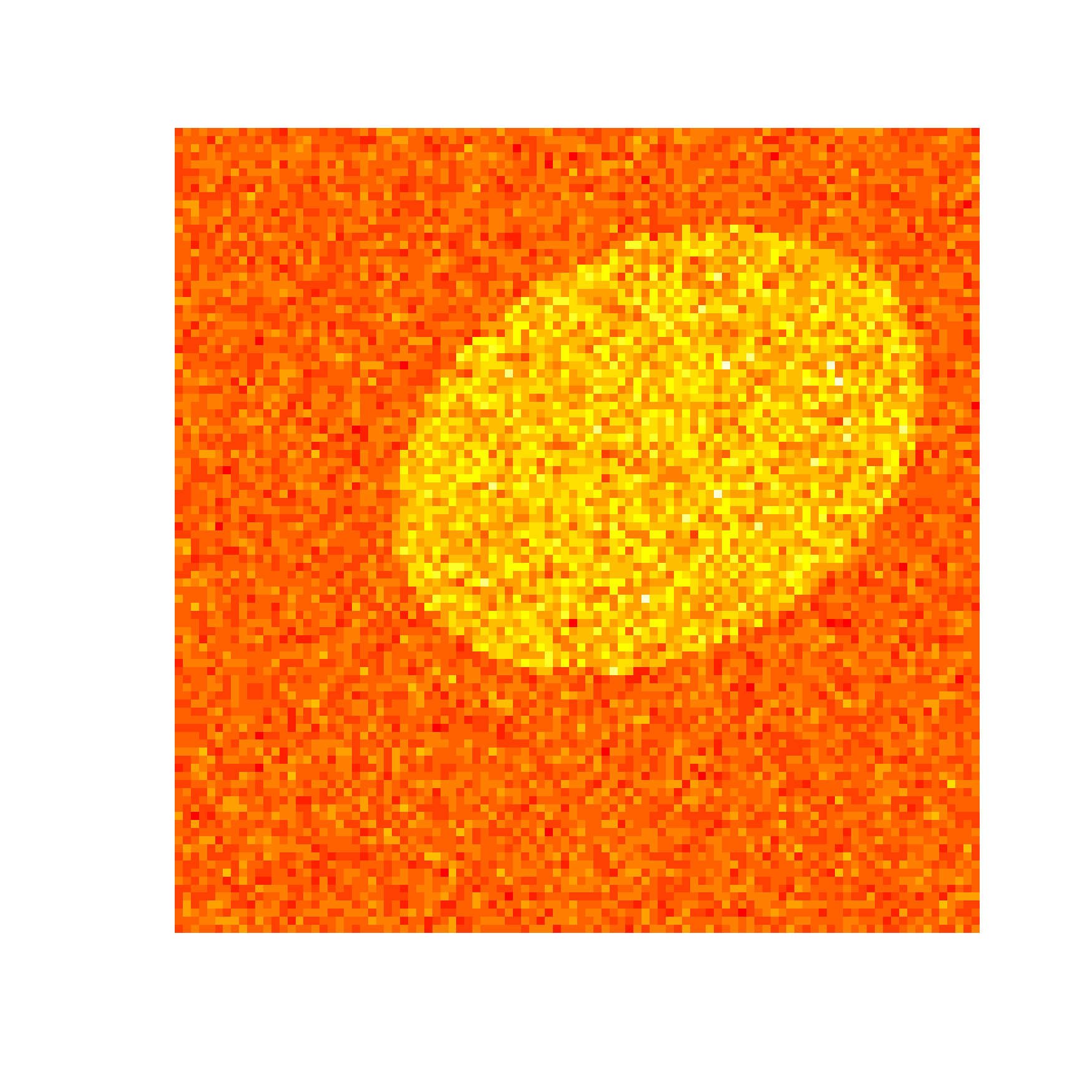}}
\subfloat[Case G2]{\includegraphics[width = 0.25\textwidth]{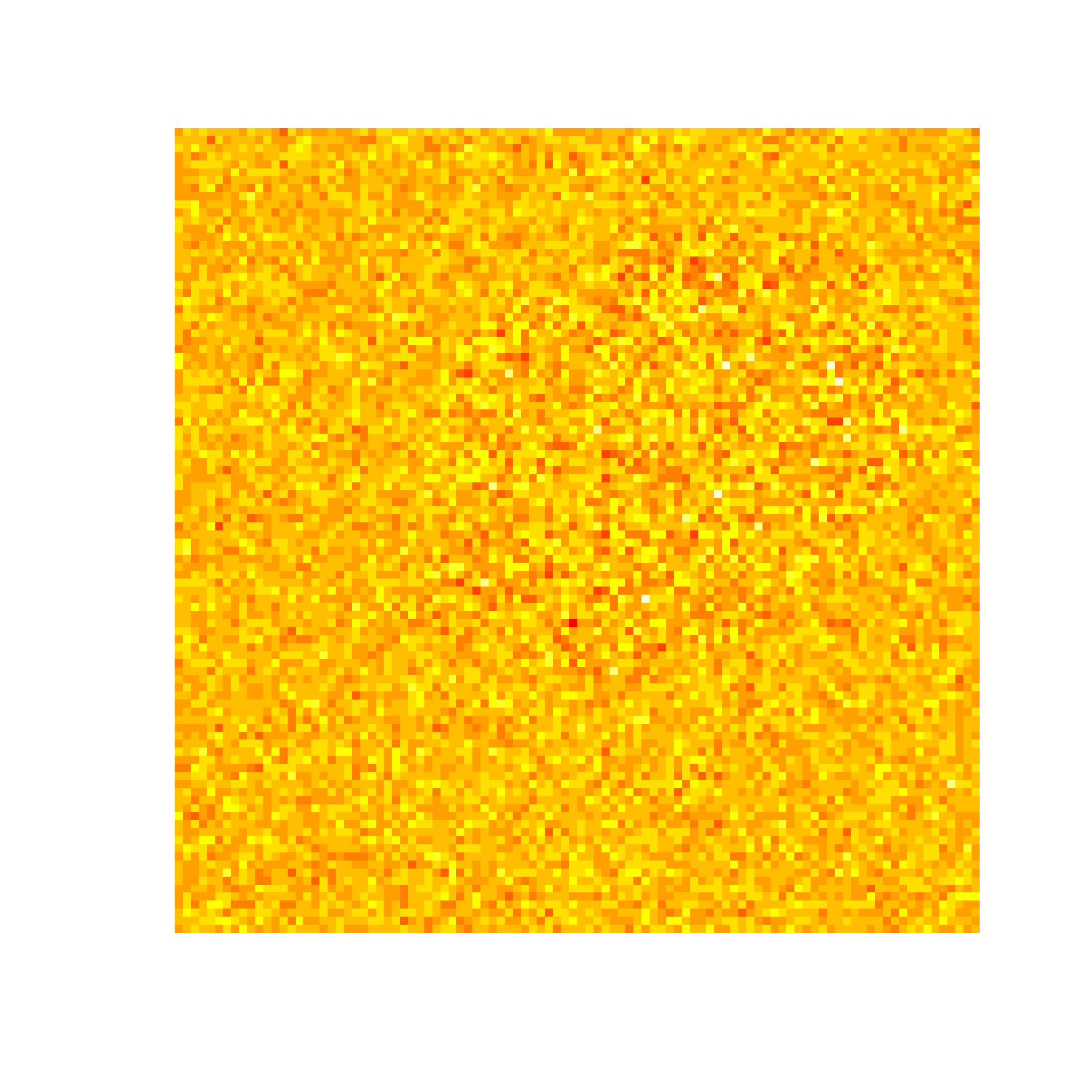}}
\subfloat[Case G3]{\includegraphics[width = 0.25\textwidth]{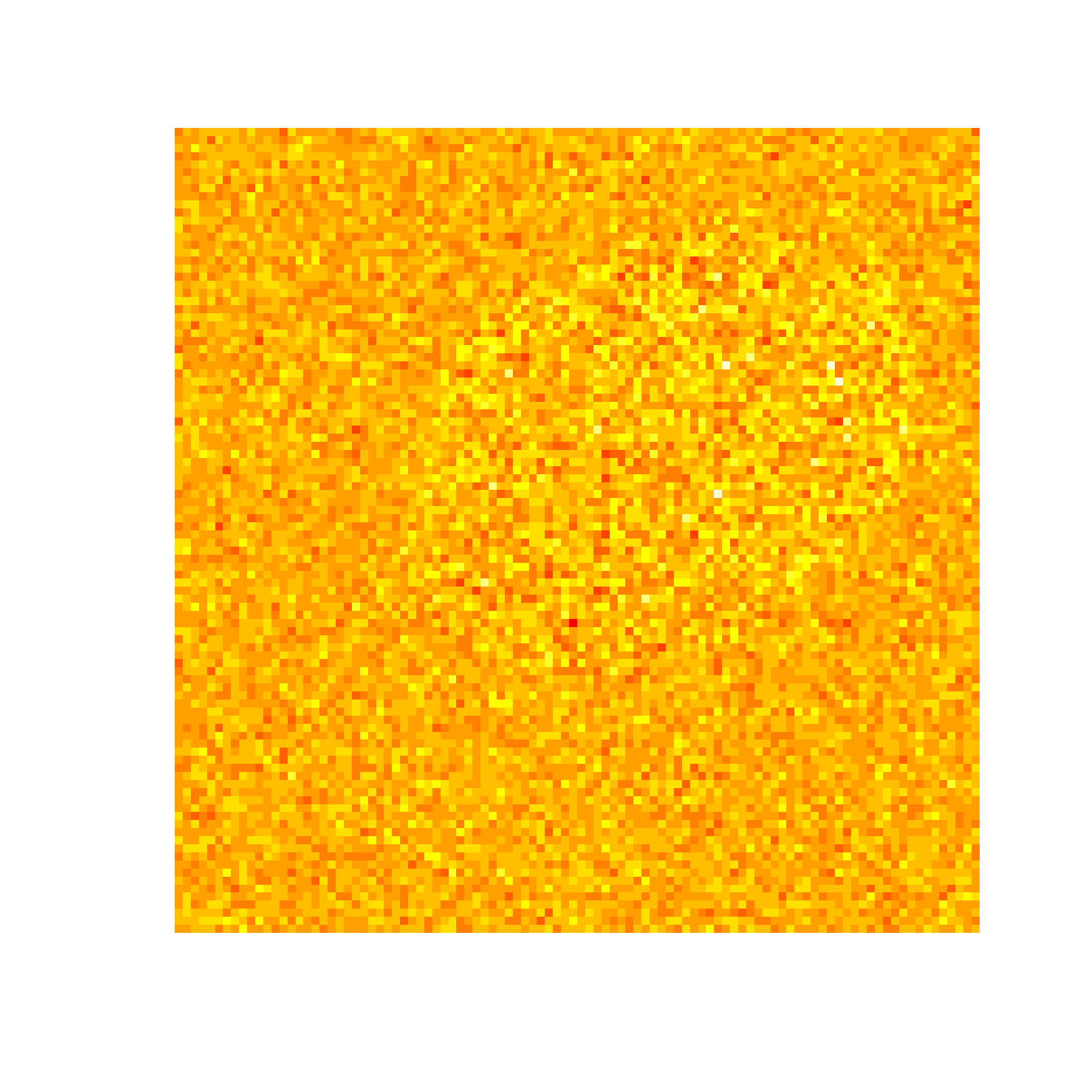}}
\subfloat[Case G4]{\includegraphics[width = 0.25\textwidth]{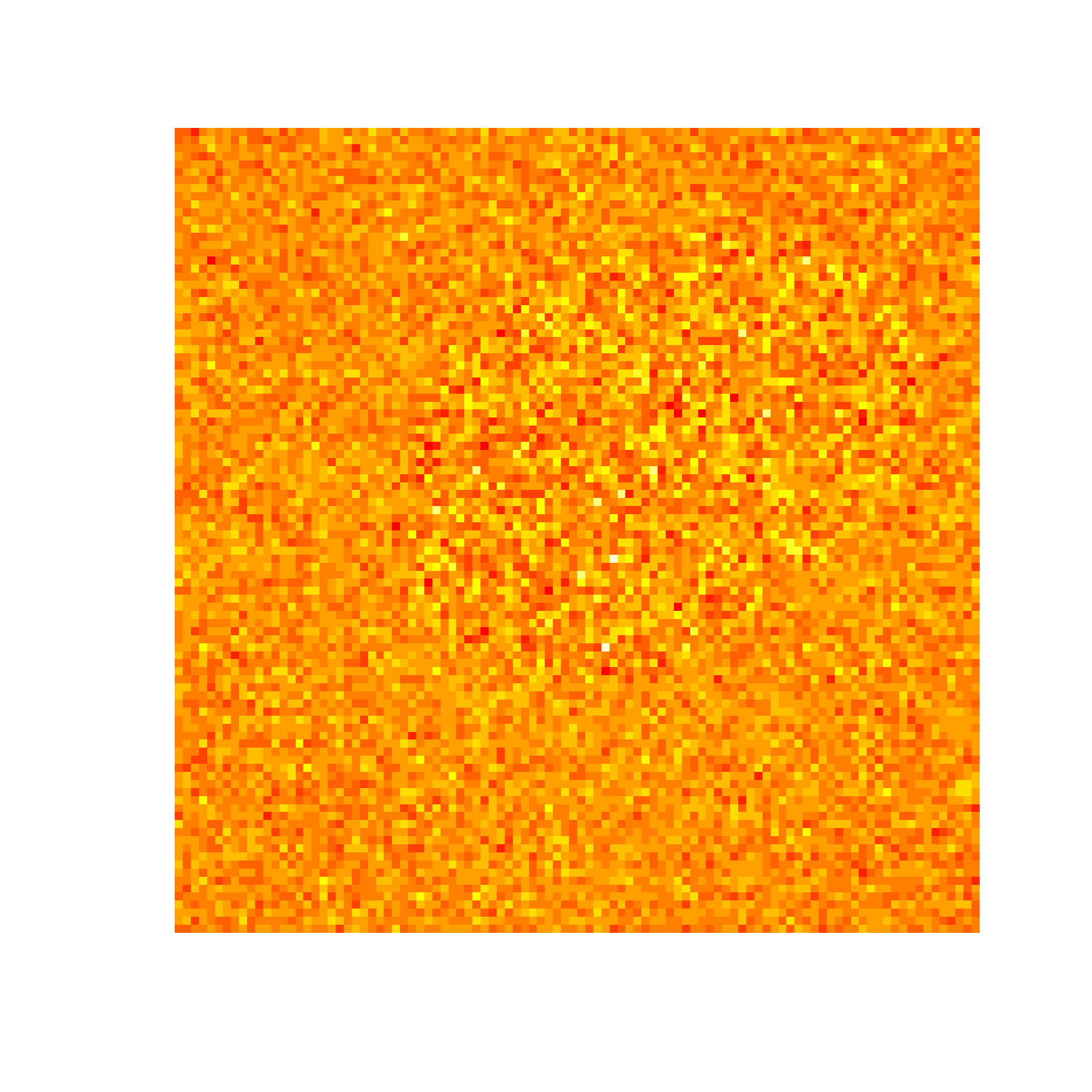}} \\
\subfloat[Case G1]{\includegraphics[width = 0.25\textwidth]{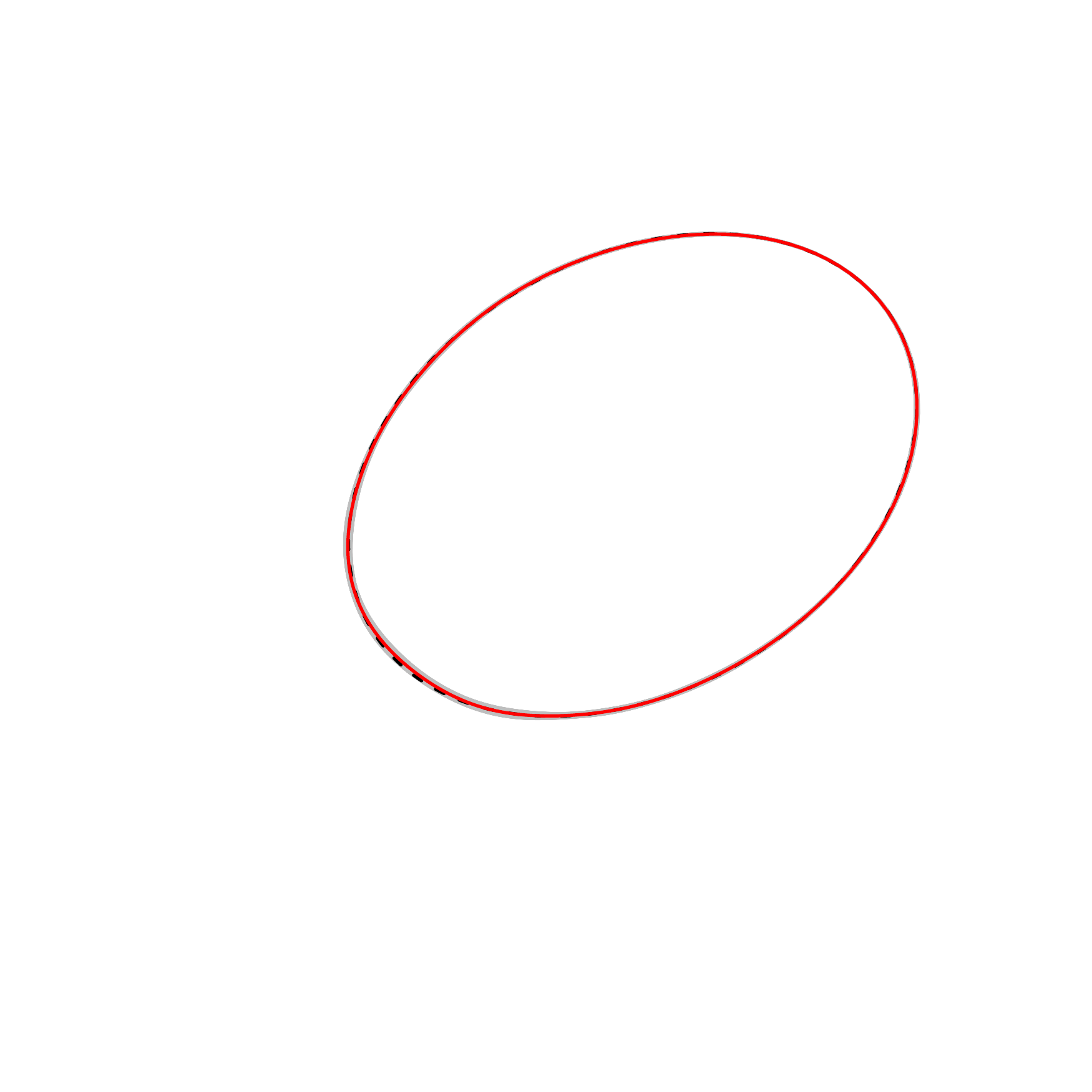}}
\subfloat[Case G2]{\includegraphics[width = 0.25\textwidth]{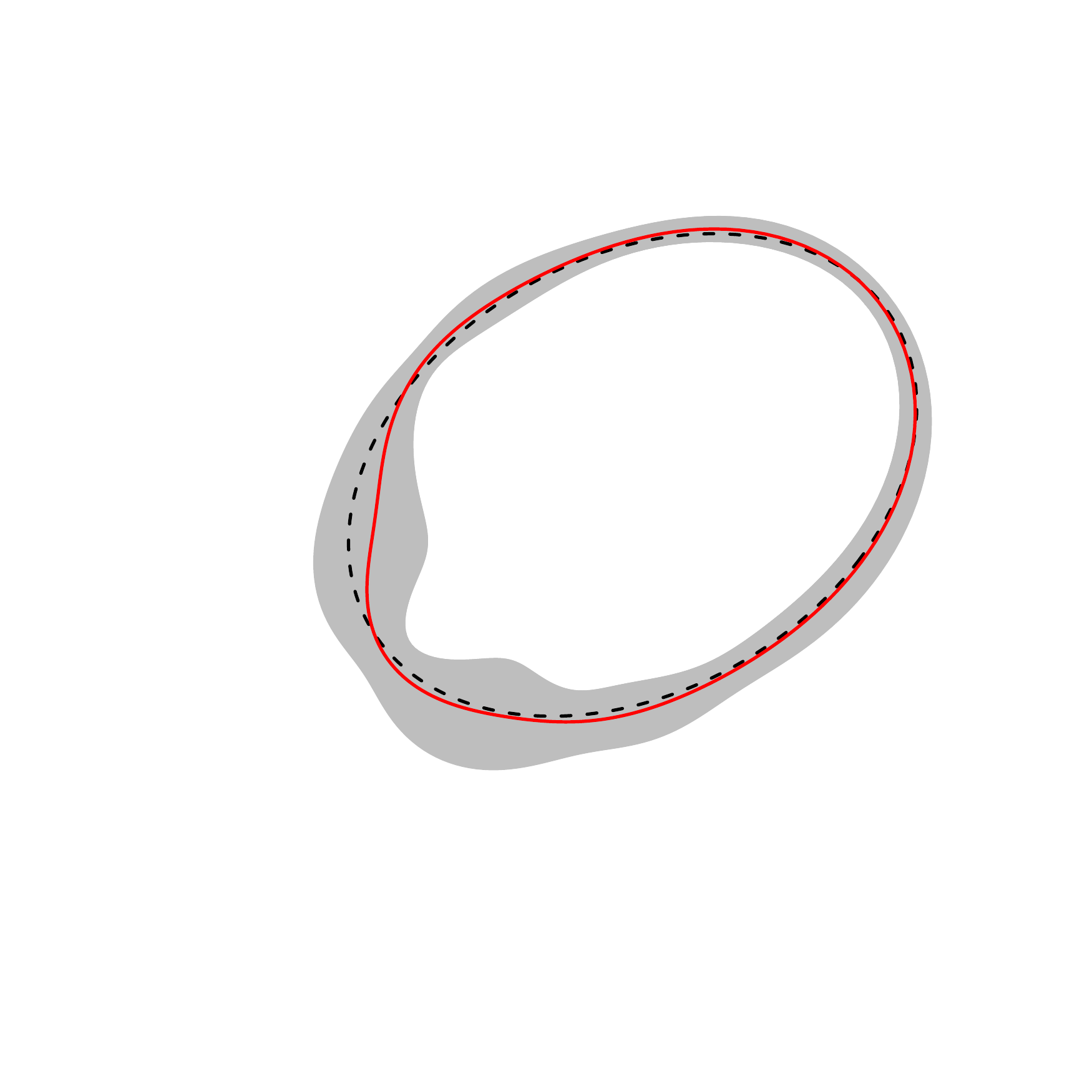}}
\subfloat[Case G3]{\includegraphics[width = 0.25\textwidth]{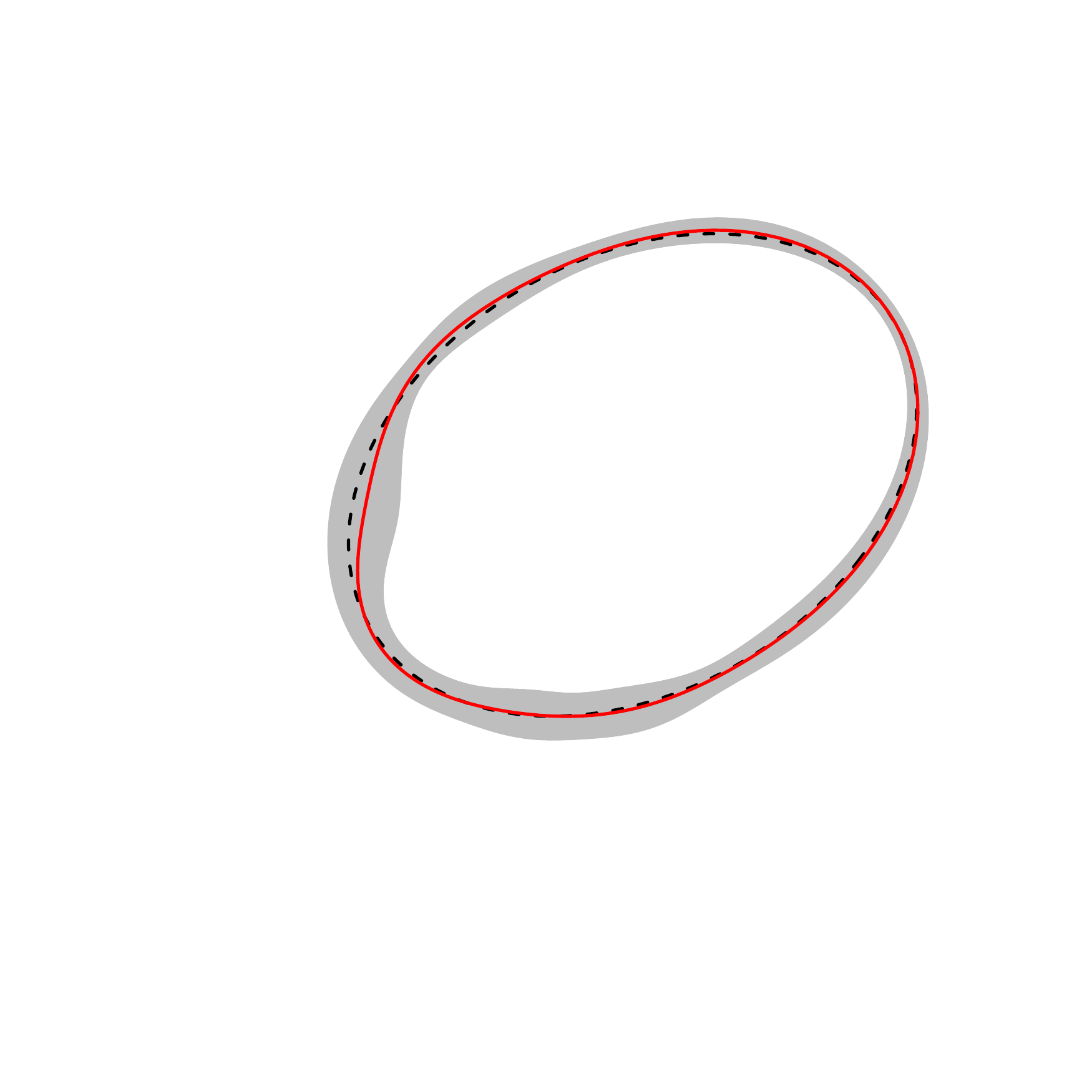}}
\subfloat[Case G4]{\includegraphics[width = 0.25\textwidth]{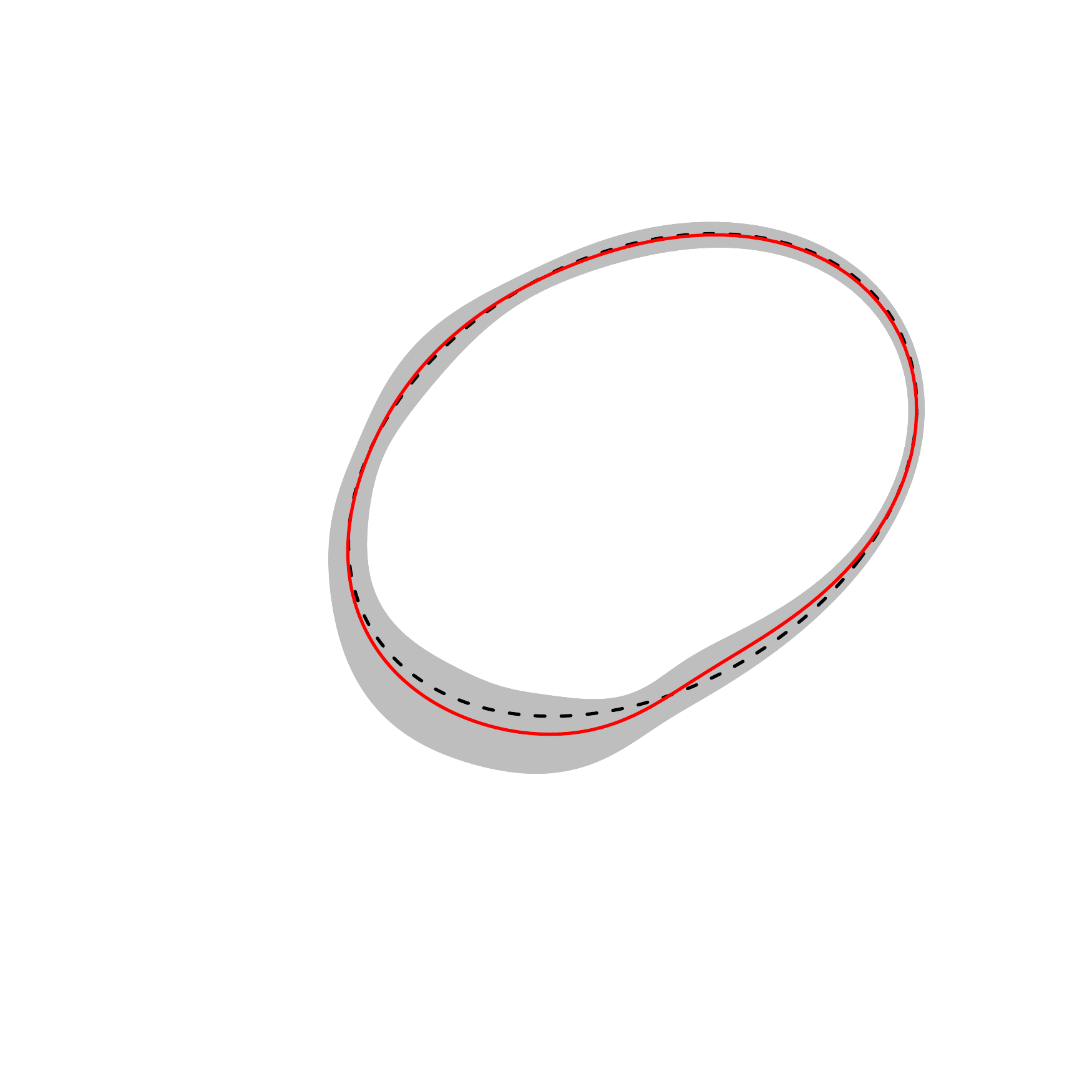}}
\caption{Proposed Bayesian estimates for Gaussian noised images with elliptic boundary. Plots (a)--(d) are the noisy observations. Figures (e)--(h) are the corresponding estimates (solid line in red) against the true boundary (dotted line in black), with a 95\% uniform credible band (in gray). }
\label{figure:Gaussian}
\end{figure}

\section{Proofs}
\label{sec:ch3.proof} 
%\newpage
\begin{proof}[Proof of Theorem~\ref{th:rate}]

{\it Step 1: Prior concentration.} Let
\begin{equation}
B_n^*(\theta_0, \epsilon) = \left\{\theta: \frac{1}{n} \sum_{i=1}^n K_{i}(\theta_0, \theta) \leq \epsilon^2, \frac{1}{n} \sum_{i=1}^n V_{i}(\theta_0, \theta) \leq \epsilon^2 \right\},
\end{equation}
where $K_{i}(\theta_0, \theta) = K(P_{\theta_0, {i}}, P_{\theta, {i}})$ and $V_{i}(\theta_0, \theta) = V(P_{\theta_0, {i}}, P_{\theta, {i}}).$ When $\|\xi - \xi_0 \| \leq \epsilon^2$ and $\|\rho - \rho_0 \| \leq \epsilon^2$ for some small $\epsilon$, it follows that
\begin{align}
K_i(\theta_0, \theta) & = K(\xi_0, \xi) P(X_{i} \in \Gamma_0 \cap \Gamma) + K(\rho_0, \rho) P(X_{i} \in \Gamma_0^c \cap \Gamma^c) \\
& \quad + K(\xi_0, \rho) P(X_{i} \in \Gamma_0 \cap \Gamma^c) + K(\rho_0, \xi) P(X_{i} \in \Gamma_0^c \cap \Gamma) \\
&\lesssim \|\xi_0 - \xi \|^2 + \|\rho_0 - \rho\|^2 + P(X_i \in \Gamma_0^c
\cap \Gamma ) + P(X_i \in \Gamma_0 \cap \Gamma^c) \\
\label{eq:KL.bound}
& =  \|\xi_0 - \xi \|^2 + \|\rho_0 - \rho\|^2  + n \lambda[(\Gamma_0 \bigtriangleup \Gamma) \cap T_i],
\end{align}
according to the Assumption (A). Consequently, the average Kullback-Leibler divergence
\begin{equation}
\label{eq:KL.bound}
\begin{split}
\frac{1}{n} \sum_i K_i(\theta_0, \theta) &\lesssim
 \|\xi_0 - \xi \|^2 + \|\rho_0 - \rho\|^2 + \frac{1}{n} n \lambda[(\Gamma_0
\bigtriangleup   \Gamma) \cap (\cup T_{\bs{i}})] \\
& =
 \|\xi_0 - \xi \|^2 + \|\rho_0 - \rho\|^2 +  \lambda( \Gamma_0
\bigtriangleup \Gamma).
% & = &  \|\xi_0 - \xi \|^2 + \|\rho_0 - \rho\|^2 + \| \gamma - \gamma_0 \|_1.
\end{split}
\end{equation} Similarly, the second moment $V_i$ of the log-likelihood ratio is also bounded in the same way, i.e. $
V_{i}(\theta_0, \theta) \lesssim \|\xi_0 - \xi \|^2 + \|\rho_0 - \rho\|^2 +  \lambda( \Gamma_0
\bigtriangleup \Gamma),
$
which leads to
\begin{equation}
  B_n^*(\theta_0, \epsilon) \supset \{(\xi, \rho, \gamma): \|\xi_0 -
  \xi\|^2 \leq \epsilon^2/3, \|\rho_0 - \rho\|^2 \leq \epsilon^2/3, \lambda( \Gamma_0
  \bigtriangleup \Gamma) \leq \epsilon^2/3 \}. 
\end{equation} 
\comment{
Therefore, we have 
$
	\Pi(B_n^*(\theta_0, \epsilon)) \gtrsim \Pi(\|\xi_0 - \xi\|^2
	\leq \epsilon^2/3, \|\rho_0 - \rho\|^2 \leq \epsilon^2/3) \times
	\Pi(\lambda( \Gamma_0
	\bigtriangleup \Gamma) \leq \epsilon^2/3),
$
or equivalently, 
\begin{align}
-\log \Pi(B_n^*(\theta_0,\epsilon_n))  \lesssim & -\log \Pi(\|\xi_0 - \xi\|^2
\leq \epsilon^2/3, \|\rho_0 - \rho\|^2 \leq \epsilon^2/3)  \\ & -\log \Pi(\gamma: \lambda( \Gamma_0
\bigtriangleup \Gamma) \leq \epsilon^2/3) .
\end{align}
By Assumption (B1), the prior density of $(\xi, \rho)$ is bounded below in a neighborhood of $(\xi_0, \rho_0)$, indicating that $\Pi(\|\xi_0 - \xi\|^2
\leq \epsilon^2/3, \|\rho_0 - \rho\|^2 \leq \epsilon^2/3)  \gtrsim \epsilon^{2p}$ and thus $-\log \Pi(\|\xi_0 - \xi\|^2
\leq \epsilon^2/3, \|\rho_0 - \rho\|^2 \leq \epsilon^2/3)  \lesssim 
\log(1/\epsilon^2)$. 
}

\comment{ 
Let $\epsilon_n$ be a sequence such that $\epsilon_n \rightarrow 0$ and $ n\epsilon_n^2 / \log n$ bounded away from 0, we then have $\log (1/\epsilon_n^2) \lesssim \log n \lesssim  n \epsilon_n^2$. Then in order to ensure that $ -\log \Pi(B_n^*(\theta_0,\epsilon_n)) \lesssim  n \epsilon_n^2$, it suffices that 
$-\log \Pi(\gamma: \lambda( \Gamma_0
\bigtriangleup \Gamma) \leq \epsilon_n^2) \lesssim  n \epsilon_n^2$ in equation~\eqref{eq:priormass}.
}

{\it Step 2: Sieves.} For each prior, we shall define a sieve $\Sigma_n$ for
 $\gamma$, and consider $\Theta_n = [-c_n, c_n]^p \times [-c_n, c_n]^p \times \Sigma_n$ as the sieve for $\theta$. Because 
 $$ \Pi(\Theta_n^c) \leq \Pi(\xi: \xi \notin [-c_n, c_n]^p) + \Pi(\rho: \rho \notin [-c_n, c_n]^p) + \Pi(\gamma: \gamma \notin \Sigma_n),$$
 in order to ensure that the
  sieve contains most of the prior mass, it is sufficient to show $-\log \Pi(\Sigma_n^c) \gtrsim  n\epsilon_n^2$ as in equation~\eqref{eq:sieve} provided that $-\log \Pi(\xi: \xi \notin [-c_n, c_n]^p) \gtrsim n \epsilon_n^2$ and $-\log \Pi(\rho: \rho \notin [-c_n, c_n]^p) \gtrsim n \epsilon_n^2$. For the later two conditions, we let $c_n = e^{n\epsilon_n^2}$. Then $-\log \Pi(\xi: \xi \notin [-c_n, c_n]^p) \gtrsim -\log c_n^{-t_2}$ by Assumption (B2), which is $t_2 \cdot n \epsilon_n^2 \gtrsim n \epsilon_n^2$; similarly, we have $-\log \Pi(\xi: \xi \notin [-c_n, c_n]^p) \gtrsim n \epsilon_n^2$.

{\it Step 3: Entropy bounds.}
Let $\sigma_n = \underset{\gamma \in {\Sigma}_n}{\sup} \| \gamma \|_{\infty}$, for $\gamma, \gamma' \in \Sigma_n$, we then have 
\begin{equation}
\lambda(\gamma, \gamma') = \int_{\sphere} \left|\int_{ \gamma(\bs{\omega})}^{\gamma'(\bs{\omega})} r^{d - 1} dr \right| d \bs{\omega} \leq  \sigma_n^{d - 1} \|\gamma' - \gamma\|_{\infty} \int_{\sphere} d\bs{\omega}\lesssim  \sigma_n^{d - 1} \| \gamma - \gamma'\|_{\infty}.
\end{equation}
Like in equation~\eqref{eq:KL.bound}, the average squared Hellinger distance $d_n^2$ has the following bound when $\|\xi - \xi'\| \leq \epsilon$ and $\|\rho - \rho'\| \leq \epsilon$ for some small $\epsilon$ \comment{and $\max(\|\xi\|, \|\xi'\|, \|\rho\|, \|\rho'\|) \leq M$:
\begin{equation}
\label{eq:d_nbound}
\begin{split}
  d_n^2(\theta, \theta') & = \frac{1}{n} \sum_{i=1}^n \int h^2(\phi_i, \phi_i') d P_{X_{i}} \lesssim h^2(\xi, \xi') + h^2(\rho, \rho')+ \lambda(\Gamma
  \bigtriangleup \Gamma') \\
 & \lesssim  (1+M^{b_0}) ( \|\xi - \xi'\|^2 + \|\rho - \rho'\|^2 )+ \comment{\sigma_n^{d - 1}}\|\gamma - \gamma'\|_{\infty}
\end{split}
\end{equation}
by Condition (A2).} Therefore the entropy $\log N(\epsilon_n, \Theta_n, d_n)$ is bounded by \comment{
\begin{align}
2\log N(\epsilon_n^2/(1+c_n)^{b_0}, [-c_n,c_n]^p, \|\cdot\|) &  + \log N(\epsilon_n^2 \comment{/\sigma_n^{d- 1}}, \Sigma_n, \|\cdot\|_{\infty}) \\
& \lesssim \log \frac{c_n}{\epsilon_n^2} +
\log N(\epsilon_n^2 \comment{/\sigma_n^{d- 1}}, \Sigma_n, \|\cdot\|_{\infty}).
\end{align}
Recall that $c_n = e^{n \epsilon_n^2}$. Therefore $ \log ({c_n}/{\epsilon_n^2}) \le  n \epsilon_n^2 + \log(1/\epsilon_n^2) \lesssim  n \epsilon_n^2$.} Hence, in order to ensure $\log N(\epsilon_n, \Theta_n, d_n)
\lesssim n \epsilon_n^2$, it is sufficient to verify that
$\log N(\epsilon_n^2\comment{/\sigma_n^{d- 1}}, \Sigma_n, \|\cdot\|_{\infty}) \lesssim n \epsilon_n^2$ which is equation~\eqref{eq:entropy}.

Then equation~\eqref{eq:theta.rate} follows by applying Theorem 4 of~\cite{Ghosal+van:07}.

% further topic: when nonparametric is nuisance parameters, and the parametric part is of interest. should movivated by real application.

\comment{Equation~\eqref{eq:gamma.rate} will follow if we show that
$d_n(\theta, \theta_0) \leq \epsilon_n$ implies $\lambda(\Gamma \bigtriangleup \Gamma_0) \lesssim \epsilon_n^2/c_{0,n}^2$.} As argued in the derivation of~\eqref{eq:KL.bound}, $d^2_n(\theta, \theta')$ is given by
\begin{equation}
\begin{split}
\frac{1}{n} \sum_{i} \int h^2(\phi, \phi') d P_{X_{i}}
& = h^2(\xi_0, \xi) \lambda(\Gamma_0 \cap \Gamma) + h^2(\rho_0,\rho) \lambda(\Gamma_0^c \cap \Gamma^c) \\
\label{eq:d2.category}
& \quad + h^2(\xi_0, \rho) \lambda(\Gamma_0 \cap \Gamma^c) + h^2(\rho_0, \xi) \lambda(\Gamma_0^c \cap \Gamma).
\end{split}
\end{equation}
The above expression is larger than each of the following three expressions:
\begin{align}
\nonumber
h^2(\xi_0, \xi) \lambda(\Gamma_0 \cap \Gamma) & + h^2(\rho_0, \xi) \lambda(\Gamma_0^c \cap \Gamma) \\ &  \geq \frac{(h(\xi_0, \xi) + h(\rho_0, \xi))^2}{2} \cdot (\lambda(\Gamma_0 \cap \Gamma) \wedge \lambda(\Gamma_0^c \cap \Gamma)),\\
\nonumber
h^2(\rho_0, \rho) \lambda(\Gamma_0^c \cap \Gamma^c) & + h^2(\xi_0, \rho) \lambda(\Gamma_0 \cap \Gamma^c)\\ &  \geq \frac{(h(\rho_0, \rho) + h(\xi_0, \rho))^2}{2} \cdot (\lambda(\Gamma_0^c \cap \Gamma^c)\wedge \lambda(\Gamma_0 \cap \Gamma^c)), \\
%\label{eq:dummy1.0216}
h^2(\xi_0, \rho) \lambda(\Gamma_0 \cap \Gamma^c) & + h^2(\rho_0, \xi) \lambda(\Gamma_0^c \cap \Gamma) \\ & \geq \frac{(h(\xi_0, \rho) + h(\rho_0, \xi))^2}{2} \cdot (
		 \lambda(\Gamma_0 \cap \Gamma^c)\wedge \lambda(\Gamma_0^c \cap \Gamma)).
\end{align}
We further have 
%\begin{align}
%\label{eq:dummy1.2.9}
$h(\xi_0, \xi) + h(\rho_0, \xi)  \geq h(\xi_0, \rho_0)$ and
%\label{eq:dummy2}
$h(\xi_0, \rho) + h(\rho_0, \rho) \geq h(\xi_0, \rho_0)$,
%\end{align}
by the triangle inequality, and $h(\xi_0, \rho) + h(\rho_0, \xi) \geq c_{0,n} > 0$ by Condition (C). 
%\begin{equation}
%\label{eq:dummy3}
%\end{equation}
Combining with the last three displays respectively, we obtain
\begin{align}
\label{eq:dummy5}
\lambda(\Gamma_0 \cap \Gamma)\wedge \lambda(\Gamma_0^c \cap \Gamma ) & \lesssim \epsilon_n^2/ c_{0,n}^2,  \\
\label{eq:dummy5-1}
\lambda(\Gamma_0^c \cap \Gamma^c) \wedge \lambda(\Gamma_0 \cap \Gamma^c ) & \lesssim \epsilon_n^2/ c_{0,n}^2,  \\
\label{eq:dummy6}
\lambda(\Gamma_0 \cap \Gamma^c)\wedge \lambda(\Gamma_0^c \cap \Gamma) & \lesssim \epsilon_n^2/c_{0,n}^2, 
\end{align}
whenever $d_n^2(\theta_0, \theta) \leq \epsilon_n^2$. By adding~\eqref{eq:dummy5} and~\eqref{eq:dummy5-1} to~\eqref{eq:dummy6}, we derive
\begin{equation}
\label{eq:dummy2.0216}
\lambda(\Gamma_0) \wedge \lambda(\Gamma_0^c \cap \Gamma )  \lesssim \epsilon_n^2/ c_{0,n}^2,   \quad
\lambda(\Gamma_0^c)\wedge \lambda(\Gamma_0 \cap \Gamma^c ) \lesssim \epsilon_n^2/ c_{0,n}^2 .
\end{equation}
Since $\Gamma_0$ is fixed with $\lambda(\Gamma_0) > 0$ and $\lambda(\Gamma_0^c) > 0$ by the assumption, \eqref{eq:dummy2.0216} implies that $\lambda(\Gamma_0^c \cap \Gamma) \lesssim \epsilon_n^2/ c_{0,n}^2 $, and $\lambda(\Gamma_0 \cap \Gamma^c) \lesssim \epsilon_n^2/ c_{0,n}^2 .$ Consequently $\lambda(\Gamma_0 \bigtriangleup \Gamma)  = \lambda(\Gamma_0^c \cap \Gamma) + \lambda(\Gamma_0 \cap \Gamma^c)
\lesssim \epsilon_n^2/ c_{0,n}^2$, which completes the proof.
\end{proof}

\begin{proof}[Proof of Theorem~\ref{th:random.series}]
We verify equations~\eqref{eq:priormass}, \eqref{eq:sieve} and~\eqref{eq:entropy} in Theorem~\ref{th:rate}. Since $\| \gamma_0 -  \bs{\beta}^T_{0, J_n} \bsxi \|_{\infty} \leq \epsilon_n^2/2$, we have 
\begin{eqnarray*}
\lefteqn{ \Pi\{ \gamma: \gamma = \bs{\beta}^T \bsxi, \| \gamma - \gamma_0\|_{\infty} \leq  \epsilon_n^2  \}}\nonumber \\   &&\geq   \Pi( J = J_n) \Pi(\| \bs{\beta}^T \bsxi - \bs{\beta}_0^T \bsxi \|_{\infty} \leq \epsilon_n^2/2 | J = J_n) \\
 &&\geq    \Pi(J = J_n) \Pi(\| \bs{\beta} - \bs{\beta}_0 \|_1 \leq t_3^{-1} J^{-t_4} \epsilon_n^2/2 | J = J_n) \label{eq:dummy20},
\end{eqnarray*}
where the last step follows because 
%\begin{equation}
%\| \bs{\beta}_{1,J}^T \bsxi - \bs{\beta}_{2,J}^T \bsxi \|_1 \leq \sum_{j = 1}^J |\beta_{1,j} - \beta_{2,j}| \|\bsxi_{j}\|_1 \leq  C_0 \sum_{j = 1}^J |\beta_{1,j} - \beta_{2,j}| =  C_0 \|\bs{\beta}_{1,J} - \bs{\beta}_{2,J}\|_1
%\end{equation}
$
\| \bs{\beta}_{1,J}^T \bsxi - \bs{\beta}_{2,J}^T \bsxi \|_{\infty} \leq \|\bs{\beta}_{1,J} - \bs{\beta}_{2,J}\|_1 \underset{ 1 \leq j \leq J}{\max} \|\xi_j\|_{\infty}
 \leq  t_3 J^{t_4} \|\bs{\beta}_{1,J} - \bs{\beta}_{2,J}\|_1
$
according to the triangle inequality and Assumption (D). Therefore, we prove equation~\eqref{eq:priormass} by noting that $-\log \Pi\{ \gamma = \bs{\beta}^T \bsxi: \| \gamma - \gamma_0\|_{\infty} \leq \epsilon_n^2 \} \leq -\log \Pi(J = J_n) -\log \Pi(\| \bs{\beta} - \bs{\beta}_0 \|_1 \leq  t_3^{-1} J^{-t_4} \epsilon_n^2/2 | J = J_n)
\lesssim J_n \log J_n + J_n \log(J_n/\epsilon_n)  \lesssim  J_n \log J_n + J_n \log n \lesssim n \epsilon_n^2.$

Considering the sieve $\Sigma_n = \{\gamma: \gamma = \bs{\beta}^T \bsxi, \bs{\beta} \in \bb{R}^j, j \leq J_n, \|\bs{\beta}\|_{\infty} \leq \sqrt{n/C} \}$, the estimate of the prior mass of the complement of the sieve is given by $\Pi(\gamma: \gamma \notin \Sigma_n) \leq \Pi(J > J_n) + J_n e^{-n}$ (see equation (2.10) in~\cite{Shen+Ghosal:14}). For any $a, b > 0$, we have $ \log (a + b) \leq \log (2 (a \vee b))$, leading to $ - \log(a + b) \geq -\log 2 + (-\log a) \wedge( - \log b).$ \comment{Noting that $-\log \Pi(J > J_n) \gtrsim J_n \log J_n$, and $-\log (J_n e^{-n}) = n - \log J_n \geq n - \log n \asymp n \gtrsim  J_n \log  J_n$ (because $J_n \log J_n \lesssim n \epsilon_n^2 \lesssim n$)}, we then obtain $
-\log \Pi(\gamma: \gamma \notin \Sigma_n) \gtrsim J_n \log J_n \gtrsim n \epsilon_n^2$ verifying equation~\eqref{eq:sieve}. 

\comment{ 
For the entropy calculation, we first notice that for any $\gamma \in \Sigma_n$, we have $\|\gamma\|_{\infty} = \| \bs{\beta}^T \bsxi \|_{\infty} \leq \underset{1 \leq j \leq J_n}{\max}\|\xi_j\|_{\infty} \|\bs{\beta}\|_{1} \lesssim J_n^{t_4} \sqrt{n} \lesssim n ^{t_4 + 1/2}$, which is an upper bound for $\sigma_n$.  We  estimate the packing number $D(\epsilon_n^2\comment{/\sigma_n^{d- 1}}, \Sigma_n, \|\cdot\|_{\infty})$ by 
\begin{align}
 \sum_{j=1}^{J_n} D(\epsilon_n^2\comment{/\sigma_n^{d- 1}}, \{\bs{\beta} \in \bb{R}^j, \|\bs{\beta}\|_{\infty} \leq \sqrt{n/C} \}, \|\cdot\|_{\infty}) 
% & \leq \log \left( \sum_{j=1}^{J_n} D(\epsilon_n^2/j, \{\bs{\beta} \in \bb{R}^j, \|\bs{\beta}\|_{\infty} \leq \sqrt{n/C} \}, \|\cdot\|_{\infty}) \right) \\
\leq \sum_{j = 1}^{J_n} \left(1 + \frac{ \sqrt{n/C} \comment{\sigma_n^{d- 1}}}{\epsilon_n^2}\right)^{j}, 
\end{align} }
which is further bounded by $J_n  (1 + { \sqrt{n/C} \comment{\sigma_n^{d- 1}}}/{\epsilon_n^2})^{J_n}$. Equation~\eqref{eq:entropy} follows since $\log N(\epsilon_n^2\comment{/\sigma_n^{d- 1}}, \Sigma_n, \|\cdot\|_{\infty}) \leq \log D(\epsilon_n^2\comment{/\sigma_n^{d- 1}}, \Sigma_n, \|\cdot\|_{\infty}) \lesssim \log J_n + J_n \log n + J_n \log (1/\epsilon_n^2)  \lesssim J_n \log n \lesssim n \epsilon_n^2.$ 
%\leq \log \left(J_n  \left(1 + \frac{ \sqrt{n/C} }{\epsilon_n^2}\right)^{J_n}
%& \lesssim  \log J_n + J_n \log n + J_n \log_- \epsilon_n^2  \lesssim J_n \log n \lesssim n \epsilon_n^2.
%\end{align}
%We conclude the proof by establishing equation~\eqref{eq:entropy} as follows. 
%\begin{align}
%\log N(\epsilon_n^2, \Sigma_n, \|\cdot\|_{\infty}) & \leq \log D(\epsilon_n^2, \Sigma_n, \|\cdot\|_{\infty}) \\
%& \leq \log \left( \sum_{j=1}^{J_n} D(\epsilon_n^2, \{\bs{\beta} \in \bb{R}^j, \|\bs{\beta}\|_{\infty} \leq \sqrt{n/C} \}, \|\cdot\|_{\infty}) \right) \\
%% & \leq \log \left( \sum_{j=1}^{J_n} D(\epsilon_n^2/j, \{\bs{\beta} \in \bb{R}^j, \|\bs{\beta}\|_{\infty} \leq \sqrt{n/C} \}, \|\cdot\|_{\infty}) \right) \\
%& \leq \log \left(\sum_{j = 1}^{J_n} \left(1 + \frac{ \sqrt{n/C} }{\epsilon_n^2}\right)^{j} \right)  \leq \log \left(J_n  \left(1 + \frac{ \sqrt{n/C} }{\epsilon_n^2}\right)^{J_n} \right)\\
%& \lesssim  \log J_n + J_n \log n + J_n \log_- \epsilon_n^2  \lesssim J_n \log n \lesssim n \epsilon_n^2.
%\end{align}
\end{proof}

\begin{proof}[Proof of Theorem~\ref{thm:GP.rate}]
	
	We first obtain the contraction rate for deterministic rescaling when 
	the smoothness level $\alpha$ is known. 
	
	Let $\mathbb{B}$ be $\bb{C}^{\alpha}(\mathbb{S}^{1})$ equipped with the $\|\cdot \|_{\infty}$ norm. Let $\phi_{\gamma_0}^a(\epsilon) =
	\underset{\gamma \in \bb{H}^a: \| \gamma - \gamma_0 \|_{\infty} \leq
		\epsilon}{\inf} \frac{1}{2} \|\gamma\|_{\bb{H}^a}^2 -\log \P(\|W^a\|_{\infty} \leq \epsilon)$ stand for the concentration function at $\gamma_0$. Note that $ \phi_0^a(\epsilon) = -\log \P(\|W^a\|_{\infty} \leq \epsilon)$.
	
	\comment{
		The selection of sieves and entropy calculation is similar to Theorem 2.1 in~\cite{van+van:08} with $\epsilon_n$ replaced by $\epsilon_n^2$ and adjustment because of the involvement of $\sigma_n$ later on. Define the sieve as $\Sigma_n = (M_n\bb{H}_1^a + \frac{1}{4} \epsilon_n^2 M_n^{-1} \bb{B}_1)$, where
		$\bb{H}_1^a$ and $\bb{B}_1$ are the unit balls of $\bb{H}^a$ and $\bb{B}$ respectively.
		Let $M_n = -2 \Phi^{-1}(\exp(- C n\epsilon_n^2))$ for a large constant $C > 1$. Then by Borell's
		inequality, we can bound $\Pi(\Sigma_n^c) \leq \Pi (\gamma \nin M_n
		\bb{H}_1^a + \epsilon_n^2 \bb{B}_1)\lesssim \exp(- C n \epsilon_n^2),$
		provided that $\phi_0^a(\frac{1}{4}\epsilon_n^2 M_n^{-1}) \leq n \epsilon_n^2$.
	}
	
	\comment{
		For the entropy calculation, first observe that $M_n^2 \lesssim n \epsilon_n^2$ since $|\Phi^{-1}(u)| \leq \sqrt{2 \log(1/u)}$ for $u \in (0, 1/2)$. We further notice that $\mathbb{H}_1^{a} \subset \mathbb{B}_1$ since by Lemma~\ref{lemma:rkhs}, for any function $h \in \mathbb{H}^{a}_1$, we have $\|h\|_{\infty}^2 \leq \{\sum_{n = -\infty}^{\infty} |b_{n, a}| e^{-2a^2} I_n(2 a^2)\}^2 \leq \sum_{n = -\infty}^{\infty} |b_{n, a}|^2 e^{-2a^2} I_n(2 a^2) \cdot \sum_{n = -\infty}^{\infty}  e^{-2a^2} I_n(2 a^2) = \|h\|_{\mathbb{H}^a}^2 \cdot 1$ (the last step uses Proposition~\ref{prop:bessel.basic} (a) by letting $z = 1$ and $x = a^2$). Therefore, the sieve $\Sigma_n$ is a subset of $(M_n + \epsilon_n^2 M_n^{-1}/4) \mathbb{B}_1$ and thus $\sigma_n = \sup \{\|\gamma\|_{\infty}: \gamma \in \Sigma_n\} \leq M_n + \epsilon_n^2 M_n^{-1}/4\leq 2 M_n$ for sufficiently large $n$. By construction of $\Sigma_n$, a $\frac{1}{4}\epsilon_n^2 M_n^{-2}$-net for $\bb{H}_1^a$ is a $\frac{1}{2}\epsilon_2^2 M_n^{-1}/2$-net for $\Sigma_n$. 
		Therefore by Lemma~\ref{lemma:entropy.lambda}, we have 
		\begin{align}
		\log N \left(\frac{\epsilon_n^2}{\sigma_n}, \Sigma_n, \|\cdot\|_{\infty}\right) & \leq  \log N \left( \frac{\epsilon_n^2}{4  M_n^2}, \mathbb{H}_1^a, \|\cdot\|_{\infty}\right) \\
		& \lesssim a \left(\log \frac{M_n^2}{2 \epsilon_n^2}\right)^2 \lesssim a (\log n)^2. 
		\end{align}
	}

	To evaluate the prior concentration probability, we proceed as follows.  Let $\Pi^a(\cdot)$ be a SEP Gaussian process with the rescaling factor $a$. By the approximation property of $\bb{H}^a$ in Lemma~\ref{lemma:approx.RHKS}, there exists $h_0 \in \bb{H}^a$ such that $\| h_0 - \gamma_0 \|_{\infty}
	\lesssim a^{-\alpha}$ and \comment{$ \|h_0 \|_{\bb{H}^a}^2 \lesssim
		a$}. Therefore, if $a^{-\alpha} \leq \epsilon_n^2/2$, then
	\comment{ 
		$$ \Pi^a(\gamma: \| \gamma - \gamma_0 \|_{\infty} \leq \epsilon_n^2)
		\geq \Pi^a(\gamma: \| \gamma - h_0 \|_{\infty} \leq \epsilon_n^2/2)
		\geq \exp\{ - \phi_{0}^a(\epsilon_n^2/4) \},$$ leading to
		$-\log \Pi^a(\gamma: \| \gamma - \gamma_0 \|_{\infty} \leq \epsilon_n^2)
		\lesssim \phi_0^a(\epsilon_n^2/4).$
	}
	
	Note that $ \phi_0^a(\epsilon) \lesssim a (\log(a/\epsilon))^2$ (Lemma~\ref{lemma:small.ball.prob}). To satisfy the conditions in Theorem~\ref{th:rate}, we choose $a = a_n$ depending on the sample size such that
	$
	\epsilon_n^2 \asymp a_n^{-\alpha},$ and $a_n (\log n)^{2} \asymp n \epsilon_n^2.
	$
	Then the posterior contraction rate is obtained as $\epsilon_n^2 = n^{-\alpha/(\alpha + 1)} (\log n)^{-2\alpha /(\alpha
		+ 1)}$, with $
	a_n = n^{1/(\alpha + 1)} (\log n)^{-2/(\alpha + 1)}.$ 
	
	Now consider the random rescaling when the smoothness $\alpha$ is unknown.  
	The established properties of the RKHS of $W^a$ from Lemma~\ref{lemma:rkhs} to Lemma~\ref{lemma:constant.rkhs} are parallel to the case when a GP is indexed by $[0, 1]^{d - 1}$ with a stationary kernel, therefore, we can directly follow the argument in the proof of Theorem 3.1 in~\citep{van+van:09}. There is need for a slight modification since the nesting property given in Lemma~\ref{lemma:nesting.rkhs} has a universal constant $c$, but this does not affect the asymptotic rate. The posterior contraction rate $\epsilon_n^2$ is thus obtained.  
\end{proof}

\begin{proof}[Proof of Lemma~\ref{lemma:bd.isometric}]
	For any $f'_{s'} \in \mathbb{H}'$, there is a series of $\alpha_i$'s such that $f'_{s'}(\cdot) = \sum \alpha_i K(s', \cdot)$, and there exists $s \in [0,1]^{d - 1}$ such that $s' = Qs$. Let $f_{s} = \sum \alpha_i G(s, \cdot) \in \mathbb{H}$. Therefore, the map $\phi: \mathbb{H} \rightarrow \mathbb{H}', \phi f_{s} = f'_{Q s}$ is surjective. 
	If there exists another $s_2 \in [0, 1]^{d - 1}$ such that $s' = Qs_2$ and $f_{s_2} = \sum \alpha_i G(s_2, \cdot) \in \mathbb{H}$, then $f_{s_2} = f_{s}$ because $f_s = \sum \alpha_i K(Qs, Q\cdot) = \sum \alpha_i K(Qs_2, Q\cdot) = f_{s_2}$. Therefore, the map $\phi$ is bijective. In addition, the definition of $G(\cdot, \cdot): G(s_1, s_2) = K(Qs_1, Qs_2)$ also implies that the map $\phi$ is distance preserving when using $\|\cdot\|_{\mathbb{H}}$ and $\|\cdot\|_{\mathbb{H}'}$ as the norms. Therefore $\phi$ extends to an isometric isomorphism between $\mathbb{H}$ and $\mathbb{H}'$. The map $\phi$ also preserves the distance if we use the $\|\cdot\|_{\infty}$ norm. 
\end{proof}

\begin{proof}[Proof of Lemma~\ref{lemma:spectral.measure}]
The generating function of $I_n(2x)$ is given in Proposition~\ref{prop:bessel.basic} (a): $e^{x(z + 1/z)} = \sum_{n = -\infty}^{\infty} I_n(2x) z^n$ for $z \in \mathbb{C}$ and $z \neq 0$. Let $x = a^2$ and $z = e^{-2\pi i t}$. Noting that $z + 1/z - 2 = -4 \sin^2(\pi t)$, we then have $\phi_a(t) =  \exp\{- 4 \scale^2 \sin^2(\pi t) \} = \sum_{n = -\infty}^{\infty} e^{-2\scale^2} I_n(2\scale^2) e^{-2 n \pi i t}.$ By defining $\mu_a$ as the discrete measure in Lemma~\ref{lemma:spectral.measure}, we obtain that 
$\phi_a(t)= \int e^{-its} d \mu_a(s).$

Furthermore, $G_a(t, t') = \phi_a(t - t') = \sum_{n = -\infty}^{\infty} e^{-2\scale^2} I_n(2\scale^2) e^{-2n\pi i (t - t')}. $ In view of the orthonormality of $\{e^{2n\pi i}: n \in \mathbb{Z}\}$, we obtain that the integral $\int_0^1 K(t, t') e^{2\pi n i t} dt = e^{-2\scale^2} I_n(2\scale^2) e^{2\pi n i t'}$ for any $n \in \mathbb{Z}$. Hence the trigonometric polynomials $\{1, \cos 2n\pi t, \sin 2n\pi t: n\geq 1\}$ form the eigenfunction basis of the covariance kernel $G_a(\cdot, \cdot)$, where the corresponding eigenvalues are $\{ e^{-2\scale^2} I_n(2\scale^2), e^{-2\scale^2} I_n(2\scale^2): n\geq 0\}.$
\end{proof} 

\begin{proof}[Proof of Lemma~\ref{lemma:rkhs}]
Since the measure $\mu_{\scale}$ has subexponential tail, Lemma 2.1 in~\cite{van+van:07} is directly applicable. Using the discrete measure $\mu_{\scale}$ defined in Lemma~\ref{lemma:spectral.measure}, the proof follows. 
\end{proof} 

\begin{proof}[Proof of Lemma~\ref{lemma:approx.RHKS}]
	\comment{
By Jackson's theorem~\citep{Jackson:30}, for any function $w \in \mathbb{C}^{\alpha}[0, 1]$ and any $N$, there exists a complex-valued sequence $(s_n)$ such that $\|w - h_N^* \|_{\infty} \leq A_w N^{-\alpha}$, where $h_N^*(x) = \sum_{n = -N}^{N} s_n e^{-i2\pi n x}$ and $A_w$ is a constant depending only on $w$. Let $h_N$ be the real part of $h_N^*$
, then clearly this $h_N \in \mathbb{H}^{\scale}$ and by Lemma~\ref{lemma:rkhs} its square norm is 
\begin{equation}
\|h_N\|^2_{\mathbb{H}^{\scale}} = \sum_{n = -N}^{N}  \frac{1}{e^{-2\scale^2}I_n(2\scale^2)} \left|\int_0^1 h_N(t) e^{it2\pi n} dt \right|^2.
\end{equation}
Using the monotonicity of $I_n(2\scale^2)$ for fixed $\scale$ (Proposition~\ref{prop:bessel.basic} (b) and (c2)), we have $e^{-2\scale^2}I_n(2\scale^2) \geq e^{-2\scale^2}I_N(2\scale^2)$ for $n = 0, \pm 1, \ldots, \pm N$. Therefore, 
\begin{equation}
\|h_N\|^2_{\mathbb{H}^{\scale}} \leq \frac{1}{e^{-2\scale^2}I_N(2\scale^2)} \sum_{n = -N}^{N} \left|\int_0^1 h_N(t) e^{it2\pi n} dt \right|^2 \leq \frac{1}{e^{-2\scale^2}I_N(2\scale^2)} \|h_N\|_2^2.
\end{equation}
Let $N \asymp \scale$, then $N^{-\alpha} \asymp \scale^{-\alpha}$, and ${e^{-2\scale^2}I_N(2\scale^2)} \asymp \scale^{-1}$ according to Proposition~\ref{prop:bessel.limit}. On the other hand, $\|h_N\|^2_2 \leq \|h_N^*\|^2_2 \rightarrow \|w\|_2^2$ which is finite. Therefore, it follows that $ \|h_N\|^2_{\mathbb{H}^{\scale}} \leq D_w \scale$ for a constant $D_w$ depending only on $w$. 
}
\end{proof}

\begin{proof}[Proof of Lemma~\ref{lemma:entropy.lambda}]
We construct an $\epsilon$-net of piecewise polynomials over $\mathbb{H}_1^{\scale}$ as in the proof of Lemma 2.3 in~\cite{van+van:07}. Let $\beta_{2k}^{\scale}$ to be the $2k$th absolute moments of the spectral measure $\mu_{\scale}$, i.e. $
\beta_{2k}^\scale = \sum_{n = -\infty}^{\infty} e^{-2\scale^2} I_n(2\scale^2) (2\pi |n|)^{2k}. 
$
In \cite{van+van:07}, $\beta_{2k}^{\scale} = \beta_{2k}^1/\scale^{2k}$, but here we do not have this simple scaling relationship and need to work with $\beta_{2k}^\scale$ directly. Following the same construction in~\cite{van+van:07} but use $\beta_{2k}^\scale$, we can obtain that 
\begin{equation}
\log N( 2\epsilon, \mathbb{H}^{\scale}_1, \|\cdot\|_{\infty}) \leq \left(\frac{1}{\delta} + 1\right) \sum_{j = 0}^{k - 1} \log \left(2 \sqrt{\beta_{2j}^\scale} \frac{\delta^j}{j!} \cdot \frac{k}{\epsilon} + 1\right), 
\end{equation}
where $k \in \mathbb{N}$ such that $\sqrt{\beta_{2k}^{\scale}} d^k/k! \leq \epsilon$ for given $\epsilon, \delta > 0$. 

For any $\scale \geq 0$ and $j \geq 0$, applying  Proposition~\ref{prop:bessel.moment} with $x = \scale^2$, we get 
\begin{equation}
\beta_{2j}^{\scale} = (2 \pi)^{2j} \sum_{n = -\infty}^{\infty} e^{-2\scale^2} I_n(2\scale^2) n^{2j} \leq (2 \pi)^{2j} \frac{(4j)!}{(2j)!}  \max(\scale^{2j}, 1).
\end{equation}
Choosing $\delta = 1/\{16 \pi \max(\scale, 1)\}$, we have $
\sqrt{\beta_{2j}^{\scale}} {\delta^j}/{j!} \leq 8^{-j}\sqrt{(4j)!/(2j)!}/(j!).$ 
Using Stirling's approximation with explicit bounds: $\sqrt{2\pi} n^{n + 1/2} e^{-n} \leq n! \leq e n^{n + 1/2} e^{-n}$ for all positive integers $n$, we obtain for all $j \geq 1$, 
\begin{equation}
\sqrt{\beta_{2j}^{\scale}} \frac{\delta^j}{j!} \leq  \frac{\sqrt{e} (4j)^{2j + 1/4} e^{-2j}}{(2\pi)^{3/4}(2j)^{j + 1/4}e^{-j} j^{j + 1/2} e^{-j}\cdot 8^j} \leq \sqrt{e} (2\pi)^{-3/4} 2^{1/4}. 
\end{equation}
When $j = 0$, $\sqrt{\beta_{2j}^{\scale}} {\delta^j}/{j!} \leq 1$. Therefore, we have a uniform bound for $\sqrt{\beta_{2j}^{\scale}}{\delta^j}/{j!}, j \geq 0$. Let $k \sim \log(1/\epsilon)$, we then have 
$
\log N( 2\epsilon, \mathbb{H}^{\scale}_1, \|\cdot\|_{\infty}) \lesssim ( {1}/{\delta} + 1 ) k \log \left({k}/{\epsilon}\right) \lesssim \max(\scale, 1) \{\log({1}/{\epsilon})\}^2,$ concluding the proof.  
\end{proof} 

\begin{proof}[Proof of Lemma~\ref{lemma:small.ball.prob}]
In view of Lemma~\ref{lemma:entropy.lambda}, the proof follows the argument in Lemma 4.6 in~\citep{van+van:09} by letting $d = 2$. 
\end{proof}

\begin{proof}[Proof of Lemma~\ref{lemma:nesting.rkhs}]
We need to show that if $a \leq b$ and $f \in \sqrt{a}\mathbb{H}_1^a$, i.e., $\|f\|_{\mathbb{H}^{a}} \leq \sqrt{a}$, then $\|f\|_{\mathbb{H}^{b}} \leq \sqrt{c b}. $ By Lemma~\ref{lemma:rkhs}, it is sufficient to show that $ a e^{-2 a^2}I_n(2 a^2) \leq  c b e^{-2 b^2}I_n(2 b^2)$ for any $n \geq 0$. Consider the function $f_n(x) = \sqrt{x} e^{-x} I_n(x)$. We only need to show that $f_n(\sqrt{2a})/f_n(\sqrt{2b}) \leq c$. By Proposition~\ref{prop:bessel.increasing}, we have $f_n(\sqrt{2a})/f_n(\sqrt{2b}) \leq 1$ for $n \geq 2$. 

When $n = 0$, Proposition~\ref{prop:bessel.increasing} indicates that  $f_0(x)$ is increasing in $x$ for $x \leq 1/2$. For $x \in [1/2, \infty)$,   Proposition~\ref{prop:bessel.limit} shows that $f_0(x)$ is bounded above and below since the function $f_0(x)$ is continuous, positive and converges to $1/\sqrt{2 \pi} > 0$ as $x \rightarrow \infty$ (meaning that both $f_0(x)$ and $1/f_0(x)$ are bounded above). In other words, there exists constants $c_1, c_2 > 0$ such that $f_0(x) \in (c_1, c_2)$ for $x \in [1/2, \infty)$. Therefore, if $b \leq 1/8$, we have $f_0(\sqrt{2a}) \leq f_0(\sqrt{2b})$. If $b > 1/8$, we then have $f_0(\sqrt{2a})/f_0(\sqrt{2b}) \leq \max\{f_0(1/2), c_2\}/c_1 < \infty$. Consequently, for $a \leq b$, we have $f_0(\sqrt{2a}) \leq \max\{f_0(1/2)/c_1, c_2/c_1, 1\}  f_0(\sqrt{2b})$. Similarly when $n = 1$, the function $f_1(x)$ is increasing in $x$ for $x \leq 3/2$ and there exists two constants $c_3, c_4 > 0$ such that $f_1(x) \in (c_3, c_4)$ for $x \in [3/2, \infty)$. Consequently, we have $f_1(\sqrt{2a}) \leq \max\{f_1(3/2)/c_3, c_4/c_3, 1\}  f_1(\sqrt{2b})$. We conclude the proof by letting $c = \max\{f_0(1/2)/c_1, c_2/c_1, f_1(3/2)/c_3, c_4/c_3, 1\}$.
\end{proof}

\begin{proof}[Proof of Lemma~\ref{lemma:constant.rkhs}]
By Lemma~\ref{lemma:rkhs}, an element in $\mathbb{H}_1^{\scale}$ can be viewed as the real part of $h(t) = \sum_{n = -\infty}^{\infty} e^{-i t 2 \pi n} b_{n,a} e^{-2\scale^2} I_n (2\scale^2)$, where $b_{n,a}$ satisfies that $\sum_{n = -\infty}^{\infty} |b_{n,a}|^2 e^{-2\scale^2} I_n(2\scale^2) \leq 1. $ Applying the Cauchy-Schwartz inequality twice, we have $|h(0)|^2 \leq \sum_{n = -\infty}^{\infty} e^{-2\scale^2} I_n(2\scale^2) = 1$ and $|h(t) - h(0)|^2 \leq t^2 \sum_{n = -\infty}^{\infty} (2\pi n)^2 e^{-2\scale^2} I_n(2\scale^2).$ By Proposition~\ref{prop:2nd.moment}, we have $|h(t) - h(0)|^2 \leq 8 \pi^2 \scale^2 t^2$. This concludes the proof. 
\end{proof}

\section{Modified Bessel function of the first kind}
\label{section:bessel}
 
The modified Bessel function of the first kind are solutions to the modified Bessel's equation~\citep{Watson:95}. Throughout the paper, we consider integer orders and positive argument, i.e. $I_n(x)$ with $n \in \mathbb{Z}$ and $x > 0$. We first introduce some basic properties of $I_n(x)$ in Proposition~\ref{prop:bessel.basic}, for easy reference.  

\begin{proposition}
\label{prop:bessel.basic}
The modified Bessel functions $I_n(2x)$ has the following properties:
\begin{itemize}
\item[(a)] Generating functions. For $x \in \mathbb{R}$, \begin{equation}
G(x, z) =: e^{x(z + 1/z)} = \sum_{n = -\infty}^{\infty} I_n(2x) z^n, \quad z \in \mathbb{C}, z \neq 0. 
\end{equation}
\item[(b)] Symmetry about the order: $I_n(2x) = I_{-n}(2x)$ for $x \in \mathbb{R}$ and $n \in \mathbb{Z}$. 
\item[(c)] For $n \geq 0$ and fixed $x > 0$, the following properties hold:
	\begin{itemize}
	\item[(c1)] Series representation: $
	I_n(2x) =  x^n \sum_{j = 0}^{\infty} x^{2j}/{(j! (n + j)!)} $. 
	\item[(c2)] $I_n(2x)$ is positive and strictly decreasing in $n$.
	\item[(c3)] $I_{n}(2x) \leq I_0(2x) {(2x)^n/}{n!}.$ 
	% \item[(c4)] $I_0(2x) + \sum_{n = 1}^{\infty} 2 I_n(2x) = e^{2x}$.
	\end{itemize}

\end{itemize}
\end{proposition}

\begin{proof}
%[Proof of Proposition~\ref{prop:bessel.basic}] 
Properties (a), (b) and (c1) can be found in most literature on Bessel functions, e.g., see Chapter II of~\cite{Watson:95} and 8.51--8.52 of~\cite{Jeffrey+Zwillinger:07}. The positivity of $I_n(2x)$ follows its series representation. For (c2) and (c3), we let $r_n(x) = I_{n+1}(2x)/I_n(2x)$ where $n \geq 0$ and $x > 0$. Then by~\cite{Amos:74}, we have 
\begin{equation}
\label{eq:dummy2.2.9}
r_n(x) \leq \frac{2 x}{n + 1/2 + (4 x^2 + (n + 1/2)^2)^{1/2}} < 1, 
\end{equation}
leading to the monotonicity in (c2). Equation~\eqref{eq:dummy2.2.9} also implies that $r_n(x) \leq 2x/(n + 1)$ for all $n \geq 0$, and hence $I_{n} (2 x ) / I_0(2 x) = \prod_{k = 0}^{n - 1} r_k \leq  (2x)^n / n!$, concluding (c3).
\end{proof}

The estimate below is obtained when $x\to\infty$ with $n$ being fixed or $n\to\infty$ in such a way that $n x^{-1/2}$ tends to a finite nonnegative number. 

\begin{proposition}
\label{prop:bessel.limit}
Let $n\in \mathbb{Z}$ and $n x^{-1/2}  \rightarrow c$ for some constant $c \geq 0$ as $x \rightarrow \infty$. Then $\sqrt{x} e^{-x}I_n(x) \rightarrow (2 \pi)^{-1/2} e^{-\frac{c^2}{2}}$ as $x \rightarrow \infty$. 
\end{proposition}

\begin{proof}
%[Proof of Proposition~\ref{prop:bessel.limit}]
The integral formula for the modified Bessel function of the first kind~\citep[page 181]{Watson:95} implies that, for $n \in \mathbb{Z}^+$, $I_n(x)$ is 
\begin{equation}
\frac{1}{\pi} \int_0^\pi e^{x \cos t} \cos(nt) dt = \frac{1}{\pi} \int_0^{{\pi}/{2}} e^{x \cos t} \cos(nt) dt + \frac{1}{\pi} \int_{{\pi}/{2}}^\pi e^{x \cos t} \cos(nt) dt. 
\end{equation} 
The second integral is bounded since $\cos t \leq 0$ for $t \in [\pi/2, \pi]$. For the first integral, we set $u = 2 \sqrt{x} \sin(t/2)$, then we have $I_n(x) = e^x /(\pi \sqrt{x}) \int_{0}^{\infty} f(u, x) du + O(1)$, where 
%\begin{equation}
%I_n(x) = \frac{e^x}{\pi \sqrt{x}} \int_0^{\sqrt{2 x}} e^{-\frac{u^2}{2}} \cos \left(2n \arcsin\left(\frac{u}{2 \sqrt{x}}\right)\right) \left(1 - \frac{u^2}{4x}\right)^{-1/2} du + O(1).
%\end{equation}
%View $x$ as a function of $n$, i.e., $x = x_n$, and denote the integrand by 
$$f_x(u) = e^{-\frac{u^2}{2}} \cos \left(2n \arcsin\left(\frac{u}{2 \sqrt{x}}\right)\right) \left(1 - \frac{u^2}{4x}\right)^{-1/2} \Ind(0 < u < \sqrt{2 x}).$$ 
If $n x^{-1/2} \rightarrow c$ for a constant $c \geq 0$, we have $f_x(u) \rightarrow e^{-{u^2}/{2}} \cos(c u)$. Note that for any $x > 0$, we have $|f_x(u)| \leq \sqrt{2} e^{-{u^2}/{2}}$ which is integrable. According to the dominated convergence theorem, we obtain that  
\begin{equation}
e^{-x} I_n(x) \sqrt{x} \rightarrow \pi^{-1} \int_0^{\infty} e^{-{u^2}/{2}} \cos(c u) du = (2 \pi)^{-1/2} e^{-{c^2}/{2}}, 
\end{equation}
where the last step uses the real part of the characteristic function of a standard normal. 

%the identity
%$
%\int_0^{\infty} e^{-{u^2}/{2}} \cos(cu) du = \sqrt{{\pi}/{2}} e^{-{c^2}/{2}}. 
%$

%On the other hand, if we have $n \sqrt{x} \leq a$ for some constant $a > 0$, we have $f_n(u) \leq  e^{-\frac{u^2}{2}} \cos(c u)
%
%we have $\liminf f_n(u) \geq e^{-\frac{u^2}{2}} \geq e^{-\frac{u^2}{2}} \cos(c u)$. Therefore by applying Fatou's lemma, we have 
%\begin{equation}
%\liminf I_n(x) \geq 
%\end{equation}
\end{proof} 

\begin{proposition}
\label{prop:bessel.moment}
For any $x \geq 0$ and $j = 0, 1, 2, \ldots$, we have 
\begin{equation}
\sum_{n = -\infty}^{\infty} e^{-2x} I_n(2x) n^{2j} \leq \frac{(4j)!}{(2j)!}  \max(x^j, 1). 
\end{equation}
\end{proposition}

\begin{proof}
%[Proof of Proposition~\ref{prop:bessel.moment}]
Let $G_j(x, z) = z^{2j}  e^{x(z + 1/z)}$, then by Proposition~\ref{prop:bessel.basic} (a), we have $G_j(x,z) = \sum_{n = -\infty}^{\infty} I_n(2x) z^{n + 2j}. $  When $j = 0$, we thus have $\sum_{n = -\infty}^{\infty} e^{-2x} I_n(2x) n^{2j} = 1$, leading to the statement of the proposition. For $j \geq 1$, we first take the $2j$th partial derivatives of $G_{2j}(x, z)$ and obtain 
\begin{align} 
\label{eq:dummy2.8.5}
\frac{\partial^{2j}G_j(x, z)}{\partial z^{2j}} = \sum_{n = -\infty}^{\infty} I_n(2x)  (n + 2j)_{2j}  z^{n},
\end{align}
where $(n + 2j)_{2j} = (n + 2j) \cdot (n + 2j - 1) \cdots \cdot (n + 1)$ is the descending factorial. It is easy to see that for $n \geq 0$, $(n + 2j)_{2j} \geq n^{2j}$; for $n \in [-2j, -1]$, $(n + 2j)_{2j} = 0$; for $n < -2j$, $(n + 2j)_{2j} = (-1)^{2j} (-n - 2j)\cdot \cdot \cdot (-n - 1) \geq 0$. Therefore, we have for $z > 0$, 
\begin{equation}
\label{eq:dummy3.8.5}
\frac{\partial^{2j}G_j(x,z)}{\partial z^{2j}} \geq \sum_{n = 0}^{\infty} I_n(2x) n^{2j} z^n = \sum_{n = 1}^{\infty} I_n(2x) n^{2j} z^n. 
\end{equation}
Let $F_j(x) = \left. \frac{\partial^{2j}G_j(x,z)}{\partial z^{2j}} \right|_{z = 1}$, then equation~\eqref{eq:dummy3.8.5} implies that 
\begin{equation}
\label{eq:dummy7.8.5}
\sum_{n = -\infty}^{\infty} e^{-2x} I_n(2x) n^{2j} = 2 \sum_{n = 1}^{\infty}  e^{-2x} I_n(2x) n^{2j} \leq 2e^{-2x} F_j(x). 
\end{equation}
To bound $F_j(x)$, consider $H_j(x, z) = \log\{G_j(x,z)\} = 2j \log z + x(z + z^{-1})$. By direct calculations, the $p$th order derivative of $H_j(x, z)$ is given by
\[  \frac{\partial^p H_j(x, z)}{\partial z^p} = \left\{
\begin{tabular}{cc}
$2j z^{-1} + x - x z^{-2},$ & if $p = 1,$ \\
$(-1)^p p!(-2jp^{-1} z + x)z^{-(p + 1)},$ & if $p \geq 2.$
\end{tabular} 
\right. \]
Applying the Fa\`{a} di Bruno's formula, we have 
\begin{align}
\label{eq:dummy4.8.5}
\frac{\partial^{2j}G_j(x,z)}{\partial z^{2j}} 
%& = \sum_{\mathcal{K}} \frac{(2j)!}{k_1!\cdots k_{2j}!} \left. \frac{d^{(k_1 + \cdots + k_{2j})}e^{u}}{d u} \right|_{u = H_j(x, z)} \cdot \prod_{i = 1}^{2j}\left\{\frac{\partial^i H_j(x, z)}{\partial z^i} \cdot \frac{1}{j!} \right\}^{k_i} \\
%& 
= \sum_{\mathcal{K}} \frac{(2j)!}{k_1!\cdots k_{2j}!} e^{H_j(x, z)} \cdot \prod_{i = 1}^{2j}\left\{\frac{\partial^i H_j(x, z)}{\partial z^i} \cdot \frac{1}{i!} \right\}^{k_i},  
\end{align}
where the sum is over the set $\mathcal{K} = \{(k_1, \ldots, k_{2j}): \sum_{i = 1}^{2j} i k_i = 2j, k_i \in \{0\} \cup \mathbb{Z}^+ \}.$ Plugging in the expression of ${\partial^p H_j(x, z)}/{\partial z^p}$ and $z = 1$, equation~\eqref{eq:dummy4.8.5} leads to
\begin{align}
F_j(x) & = \left. \frac{\partial^{2j}G_j(x,z)}{d z^{2j}} \right|_{z = 1} =  \sum_{\mathcal{K}} \frac{(2j)!}{k_1!\cdots k_{2j}!} e^{2x} \prod_{i = 2}^{2j}\left\{(-1)^i \left(-\frac{2j}{i} + x\right)\right\}^{k_i} (2j)^{k_1} \\
& =  \sum_{\mathcal{K}} \frac{(2j)!}{k_1!\cdots k_{2j}!} e^{2x} (-1)^{\sum_{i = 1}^{2j} i k_i} \prod_{i = 2}^{2j} \left\{\left(-\frac{2j}{i} + x\right)\right\}^{k_i} (-2j)^{k_1} \\
& =  \sum_{\mathcal{K}} \frac{(2j)!}{k_1!\cdots k_{2j}!} e^{2x} \prod_{i = 2}^{2j} \left\{\left(-\frac{2j}{i} + x\right)\right\}^{k_i} (-2j)^{k_1}, 
\end{align}
where the last step follows because $\sum_{i = 1}^{2j} i k_i = 2j$ which is even. 

Noting that $2j/i \geq 1$ for $i = 2, \ldots, 2j$, it is  easy to verify that $|-2j/i + x| \leq \max(x, 1)  2j/i$ for $x \geq 0$. Consequently, we have 
\begin{align}
|F_j(x)| & \leq \sum_{\mathcal{K}} \frac{(2j)!}{k_1!\cdots k_{2j}!} e^{2x} \prod_{i = 2}^{2j} \{\max(x, 1)\}^{k_i} \left\{\left(\frac{2j}{i}\right)\right\}^{k_i} (2j)^{k_1} \\
\label{eq:dummy6.8.5}
& \leq \sum_{\mathcal{K}} \frac{(2j)!}{k_1!\cdots k_{2j}!} e^{2x} \{\max(x, 1)\}^{\sum_{i = 2}^{2j} k_i} \prod_{i = 1}^{2j}  \left(\frac{2j}{i}\right)^{k_i}. 
\end{align}
Furthermore, for $(k_1, \ldots, k_{2j}) \in \mathcal{K}$, we have $2\sum_{i = 2}^{2j} k_i \leq \sum_{i = 2}^{2j} i k_i = 2j - k_1 \leq 2j$, and hence $\sum_{i = 2}^{2j} k_i \leq j$. Consequently, equation~\eqref{eq:dummy6.8.5} leads to 
\begin{equation}
|F_j(x)| \leq e^{2x}  \max(x^j, 1)  \sum_{\mathcal{K}} \frac{(2j)!}{k_1!\cdots k_{2j}!} \prod_{i = 1}^{2j}  \left(\frac{2j}{i}\right)^{k_i} =: e^{2x}  \max(x^j, 1) A_{2j}.
\end{equation}
To estimate $A_{2j}$, let $\mathcal{K}_{k} = \mathcal{K} \cap \{(k_1, \ldots, k_{2j}): \sum_{i = 1}^{2j} k_i = k\}$. We then have $A_{2j} = \sum_{k = 1}^{2j} (2j)^k \sum_{\mathcal{K}_k} {(2j)!}/{(k_1!\cdots k_{2j}!)} \prod_{i = 1}^{2j}  \left({1}/{i}\right)^{k_i} = \sum_{k = 1}^{2j} (2j)^k B_{2j, k}(0!, 1!, \ldots),$ where $B_{2j, k}(0!, 1!, \ldots)$ is the so-called Bell polynomials evaluated at $(0!, 1!, \ldots)$ and is equal to the unsigned Stirling number of the first kind $|s(2j, k)|$~\citep[Theorems A and B, page 133--134 in][]{Comtet:74}. Therefore, we have $A_{2j} = \sum_{k=1}^{2j}|s(2j, k)| (2j)^k,$ which is equal to $(2j)(2j+1)\cdots(2j + 2j - 1) = (4j)!/\{2(2j)!\}$ according to the generating function of $|s(2j, k)|$~\citep[equation (5f), page 213 in][]{Comtet:74}. Therefore, $|F_j(x)| \leq e^{2x}  \max(x^j, 1) (4j)!/\{2(2j)!\}.$ Combining equation~\eqref{eq:dummy7.8.5}, this yields the bound given in the statement of the proposition. 
\end{proof} 

\begin{proposition}
\label{prop:bessel.increasing}
The function $f_n(x) = \sqrt{x} e^{-x} I_n(x)$ is increasing in $x$ when $x \in B_n$, where $B_n = [0, n + 1/2]$ if $n = 0, 1$ and $B_n = [0, \infty)$ if $n \geq 2$. 
\end{proposition}
\begin{proof}
%[Proof of Proposition~\ref{prop:bessel.increasing}]
For given $n \geq 0$, let $g_n(x) = \log f_n(x) = (\log x)/2 - x + \log I_n(x)$. Then $g'_n(x) = 1/(2x) - 1 + I_n'(x)/I_n(x)$. Let $r_n(x) = I_{n + 1}(x)/I_n(x)$, then $I_n'(x)/I_n(x) = r_n(x) + n/x$ (equation (8) in~\cite{Amos:74}). Therefore, the increasing property of $f_n(x)$ follows if we show that $1/(2x) - 1 + r_n(x) + n/x \geq 0$, or equivalently $r_n(x) \geq 1 - {(n + 1/2)}/{x}$. 

Since $r_n(x) \geq 0$ for $n \geq 0$ and $x \geq 0$, $f_n(x)$ is thus increasing in $x \in [0, n + 1/2]$ for any $n \geq 0$. When $n \geq 2$, we shall use the lower bound for $r_n(x)$ given in~\cite{Amos:74}, i.e. 
$
r_n(x) \geq x / (n + 1/2 + \{x^2 + (n + 3/2)^2\}^{1/2})$ when $x \geq 0.$
Let $t = n + 1/2$. Then it is sufficient to show that $x / (t + \{x^2 + (t + 1)^2\}^{1/2}) \geq 1 - t/x$ for $x > t \geq 5/2$. We rewrite this inequality as   
$ x^2 - (x - t)t \geq (x - t) \sqrt{x^2 + (t + 1)^2}$, which is simplified as $x^2 t^2 \geq (x - t)^2 (2t + 1)$ by algebra. It follows from the observation that when $x > t \geq 5/2$, we have $x > x - t > 0$ and $t^2 > 2t + 1$. Therefore,  $f_n(x)$ is increasing in $x$ if $n \geq 2$. 
\end{proof}

\begin{proposition}
\label{prop:2nd.moment}
For any $x \geq 0$, we have
$
\sum_{n = -\infty}^{\infty} e^{-2x} I_n(2x) n^2 = 2x. 
$
\end{proposition}
\begin{proof}
%[Proof of Proposition~\ref{prop:2nd.moment}]
Let $G_2(x, z) = z^2 e^{x ( z + 1/z)}$. A direct calculation leads to the relations $\partial G_2(x, z)/\partial z = e^{x ( z + z^{-1})}(x z^2 - x + 2z)$ and $\partial^2 G_2(x, z) / \partial z^2 = e^{x(z + z^{-1})}\{(x - xz^{-2})(xz^2 - x + 2z) + 2xz + 2\}.$ On the other hand, by Proposition~\ref{prop:bessel.basic} (a), we have $G_2(x,z) = \sum_{n = -\infty}^{\infty} I_n(2x) z^{n + 2}.$ We take derivatives at the right hand side term by term and obtain that 
\begin{equation}
\sum_{n = -\infty}^{\infty} e^{-2x} I_n(2x) n^2 = \frac{\partial^2 G_2(x, 1)}{\partial z^2} - 3\frac{\partial G_2(x, 1)}{\partial z} + 4G_2(x, 1)=2x,
\end{equation} by the expression of $\partial G_2(x, z) / \partial z$, $\partial^2 G_2(x, z) / \partial z^2$ at $z = 1$. 
\end{proof}

\appendix
%
%\section{Appendix section}\label{app}

\section*{Acknowledgements}
We thank Professor Aad van der Vaart for many helpful discussions and pointing out important references. The work was conducted when the first author was a graduate student at North Carolina State University. 

%\begin{supplement}
%\sname{Supplement A}\label{suppA}
%\stitle{Title of the Supplement A}
%\slink[url]{http://www.e-publications.org/ims/support/dowload/imsart-ims.zip}
%\sdescription{Dum esset rex in
%accubitu suo, nardus mea dedit odorem suavitatis. Quoniam confortavit
%seras portarum tuarum, benedixit filiis tuis in te. Qui posuit fines tuos}
%\end{supplement}

% \bibliographystyle{bib/apalike}
%\bibliographystyle{imsart-nameyear}
\bibliography{bib/Meng}

%\begin{thebibliography}{9}
%
%\bibitem{r1}
%\textsc{Billingsley, P.} (1999). \textit{Convergence of
%Probability Measures}, 2nd ed.
%Wiley, New York.
%\MR{1700749}
%
%
%\bibitem{r2}
%\textsc{Bourbaki, N.}  (1966). \textit{General Topology}  \textbf{1}.
%Addison--Wesley, Reading, MA.
%
%\bibitem{r3}
%\textsc{Ethier, S. N.} and \textsc{Kurtz, T. G.} (1985).
%\textit{Markov Processes: Characterization and Convergence}.
%Wiley, New York.
%\MR{838085}
%
%\bibitem{r4}
%\textsc{Prokhorov, Yu.} (1956).
%Convergence of random processes and limit theorems in probability
%theory. \textit{Theory  Probab.  Appl.}
%\textbf{1} 157--214.
%\MR{84896}
%
%\end{thebibliography}

\end{document}